\documentclass[10pt]{amsart}
\usepackage{amssymb}
\usepackage{graphicx}
\usepackage{xcolor} 
\usepackage{tensor}
\usepackage{fullpage} 
\usepackage{amsmath}
\usepackage{amsthm}
\usepackage{verbatim}
\usepackage{slashed}
\usepackage{hyperref}
\usepackage{enumitem}
\setlist[enumerate]{leftmargin=2em}
\setlist[itemize]{leftmargin=2em}

\usepackage{seqsplit}

\definecolor{green}{rgb}{0,0.8,0} 

\newtheorem{theorem}{Theorem}[section]

\newtheorem{lemma}[theorem]{Lemma}
\newtheorem{proposition}[theorem]{Proposition}
\newtheorem{innercustomthm}{Theorem}
\newenvironment{customthm}[1]
{\renewcommand\theinnercustomthm{#1}\innercustomthm}
{\endinnercustomthm}
\theoremstyle{definition}

\theoremstyle{remark}
\newtheorem{remark}[theorem]{Remark}

\numberwithin{equation}{section}
\newcommand{\relphantom}[1]{\mathrel{\phantom{#1}}}

\makeatletter
\newcommand{\nrm}{\@ifstar{\nrmb}{\nrmi}}
\newcommand{\nrmi}[1]{\Vert{#1}\Vert}
\newcommand{\nrmb}[1]{\left\Vert{#1}\right\Vert}
\newcommand{\abs}{\@ifstar{\absb}{\absi}}
\newcommand{\absi}[1]{\vert{#1}\vert}
\newcommand{\absb}[1]{\left\vert{#1}\right\vert}
\newcommand{\brk}{\@ifstar{\brkb}{\brki}}
\newcommand{\brki}[1]{\langle{#1}\rangle}
\newcommand{\brkb}[1]{\left\langle{#1}\right\rangle}
\newcommand{\set}{\@ifstar{\setb}{\seti}}
\newcommand{\seti}[1]{\{#1\}}
\newcommand{\setb}[1]{\left\{ #1\right\}}
\makeatother

\newcommand{\tld}[1]{\widetilde{#1}}

\newcommand{\VERT}[1]{{\left\vert\kern-0.25ex\left\vert\kern-0.25ex\left\vert #1 
    \right\vert\kern-0.25ex\right\vert\kern-0.25ex\right\vert}}

\DeclareMathOperator{\supp}{supp}

\let\Re\relax
\DeclareMathOperator{\Re}{Re}

\newcommand{\aeq}{\simeq}
\newcommand{\aleq}{\lesssim}
\newcommand{\ageq}{\gtrsim}

\newcommand{\lap}{\Delta}

\newcommand{\ud}{\mathrm{d}}
\newcommand{\rd}{\partial}
\newcommand{\nb}{\nabla}

\newcommand{\peq}{\relphantom{=}}			

\newcommand{\alp}{\alpha}
\newcommand{\bt}{\beta}

\newcommand{\dlt}{\delta}

\newcommand{\eps}{\epsilon}

\newcommand{\kpp}{\kappa}
\newcommand{\lmb}{\lambda}

\newcommand{\tht}{\theta}

\newcommand{\omg}{\omega}
\newcommand{\Omg}{\Omega}

\newcommand{\bfb}{{\bf b}}

\newcommand{\bfe}{{\bf e}}

\newcommand{\bfp}{{\bf p}}

\newcommand{\bfu}{{\bf u}}
\newcommand{\bfv}{{\bf v}}

\newcommand{\bfx}{{\bf x}}
\newcommand{\bfy}{{\bf y}}

\newcommand{\bfB}{{\bf B}}

\newcommand{\bfU}{{\bf U}}
\newcommand{\bfV}{{\bf V}}
\newcommand{\bfW}{{\bf W}}

\newcommand{\bfPi}{{\mathbf{\Pi}}}

\newcommand{\bbN}{\mathbb N}

\newcommand{\bbP}{\mathbb P}

\newcommand{\bbR}{\mathbb R}

\newcommand{\bbT}{\mathbb T}

\newcommand{\bbZ}{\mathbb Z}

\newcommand{\calB}{\mathcal B}
\newcommand{\calC}{\mathcal C}

\newcommand{\calI}{\mathcal I}

\newcommand{\calL}{\mathcal L}

\newcommand{\calS}{\mathcal S}

\newcommand{\frkG}{\mathfrak G}

\newcommand{\frkb}{\mathfrak b}

\newcommand{\org}[1]{
	\addtocontents{toc}{\protect\setcounter{tocdepth}{1}}
	\subsection*{Organization of the paper} {#1}
	\addtocontents{toc}{\protect\setcounter{tocdepth}{2}} }

\newcommand{\ackn}[1]{
\addtocontents{toc}{\protect\setcounter{tocdepth}{1}}
\subsection*{Acknowledgements} {#1}
\addtocontents{toc}{\protect\setcounter{tocdepth}{2}} }

\newcommand{\bgB}{\mathring{\bfB}}				
\newcommand{\bfomg}{\boldsymbol{\omg}}		
\newcommand{\err}{\boldsymbol{\epsilon}}		
\newcommand{\errh}{\boldsymbol{\delta}}			

\newcommand{\tu}{\tilde{u}}					
\newcommand{\tb}{\tilde{b}}					
\newcommand{\tf}{\tilde{f}}
\newcommand{\tPi}{\tilde{\Pi}}
\newcommand{\tpsi}{\tilde{\psi}}					
\newcommand{\tomg}{\tilde{\omg}}				

\begin{document}

\title[]{On illposedness of the Hall and electron magnetohydrodynamic equations without resistivity on the whole space}
\author{In-Jee Jeong}%
\address{Department of Mathematical Sciences and RIM, Seoul National University, and School of Mathematics, Korea Institute for Advanced Study, Seoul, Republic of Korea.}%
\email{injee\_j@snu.ac.kr}%

\author{Sung-Jin Oh}%
\address{Department of Mathematics, UC Berkeley, Berkeley (CA), USA and School of Mathematics, Korea Institute for Advanced Study, Seoul, Republic of Korea.}%
\email{sjoh@math.berkeley.edu}%

\begin{abstract}
It has been shown in our previous work \cite{JO1} that the incompressible and irresistive Hall- and electron-magnetohydrodynamic (MHD) equations are illposed on flat domains $M = \mathbb{R}^k \times \mathbb{T}^{3-k}$ for $0 \le k \le 2$. The data and solutions therein were assumed to be independent of one coordinate, which not only significantly simplifies the systems but also allows for a large class of steady states. In this work, we remove the assumption of independence and conclude strong illposedness for compactly supported data in $\mathbb{R}^3$. This is achieved by constructing degenerating wave packets for linearized systems around time-dependent axisymmetric magnetic fields. A few main additional ingredients are: a more systematic application of the generalized energy estimate, use of the Bogovski\v{i} operator, and a priori estimates for axisymmetric solutions to the Hall- and electron-MHD systems. 
\end{abstract}

\maketitle


\section{Introduction}

\subsection{The systems} In this paper, we are concerned with the Cauchy problem for the incompressible \emph{Hall-magnetohydrodynamics} (Hall-MHD for short) equations, which describe the motion of a plasma whose electron speed is significantly faster than the ion speed. Still, the equations give a single-fluid description of the plasma with $\bfu(t) : M \rightarrow \mathbb{R}^3$ as the so-called bulk plasma velocity. Together with the magnetic field $\bfB(t): M \rightarrow \mathbb{R}^3$, the equations are given by 
\begin{equation} \tag{Hall-MHD} \label{eq:hall-mhd}
\left\{
\begin{aligned}
&\rd_{t} \bfu + \bfu \cdot \nb \bfu + \nb \bfp  - \nu \lap \bfu = (\nb \times \bfB) \times \bfB,  \\
&\rd_{t} \bfB - \nb \times (\bfu \times \bfB) + \nb \times ((\nb \times \bfB) \times \bfB)= 0, \\
&\nb \cdot \bfu = \nb \cdot \bfB = 0.
\end{aligned}
\right.
\end{equation} Here, $\bfp(t): M \rightarrow \mathbb{R}$ is the plasma pressure and $\nu \ge 0$ is the plasma viscosity. We assume that the domain is given by a flat manifold $M = \bbT^k\times\bbR^{3-k}$ ($0\le k \le 3$). Note that there is no dissipation term for the magnetic field (i.e. no magnetic resistivity), which will be assumed throughout the paper. The special case $\nu = 0$ is called the \emph{ideal Hall-MHD equation}. This is just the usual ideal MHD equation with an extra term $ \nb \times ((\nb \times \bfB) \times \bfB)$ on the left hand side, which is referred to as the Hall current term. This modification over the usual MHD system was suggested by Lighthill \cite{Light} to take into account the disparity between the electron speed and the ion speed, and is known to be directly responsible for various physical phenomena including collisionless magnetic reconnection, planetary magnetospheres, and magnetic field dynamics in neutron stars (\cite{BSD,SDRD,HoGr,DKM,YBMH,SRF,RaFi,DGCT,LSDP,GoRe,Jones,GoCu,WoHoLy}). We shall also consider the simpler variant obtained by formally taking $\bfu = 0$, called the electron-MHD (E-MHD for short) equations: 
\begin{equation} \tag{E-MHD} \label{eq:e-mhd}
\left\{
\begin{aligned}
&\rd_{t} \bfB + \nb \times ((\nb \times \bfB) \times \bfB)= 0, \\
&\nb \cdot \bfB = 0. 
\end{aligned}
\right.
\end{equation} It turns out that, as long as the issue of local wellposedness is concerned, the reduced system \eqref{eq:e-mhd} serves as a good approximation for the more realistic \eqref{eq:hall-mhd}. Due to the potential loss of derivatives present in the Hall current term, the basic question of local wellposedness had been open in the irresistive case. In \cite{JO1} we have established that the Cauchy problem for the equations \eqref{eq:hall-mhd} and \eqref{eq:e-mhd} are \textit{strongly illposed} in arbitrarily high Sobolev spaces for smooth initial data, under the simplifying assumption that the data is independent of one coordinate (this will be always assumed to be the $z$-direction). This in particular forced the domain $M$ to have at least one component of $\bbT$, for the solutions to have finite energy. The resulting simplified equations are commonly referred to as $2+\frac12$ dimensional systems, and have been extensively investigated both in physical and mathematical literature (\cite{BaeKang,RaYa,Ya,CWo2,Du1,Dai3,JO1}). 



The goal of the current paper is to remove the assumption of $z$-independence and conclude illposedness of the Cauchy problem for \eqref{eq:hall-mhd} and \eqref{eq:e-mhd} in the remaining case of $M = \bbR^3$. A simple version of our main result is as follows:
\begin{theorem}[Nonlinear illposedness, nonexistence, simple version] \label{thm:main-simple}
\begin{enumerate}
\item For any $s > \frac{7}{2}$ and $\eps > 0$, there exists $\bfB_{0} \in H^{s}(\bbR^{3})$ satisfying $\nb \cdot \bfB_{0} = 0$ and $\nrm{\bfB_{0}}_{H^{s}} < \eps$ such that there is no associated solution to \eqref{eq:e-mhd} in $L^{\infty}_{t}([0, \dlt]; H^{s})$  with initial data $\bfB(t=0) = \bfB_{0}$ for any $\dlt > 0$.
\item For any $s > \frac{7}{2}$ and $\eps > 0$, there exists $\bfB_{0} \in H^{s}(\bbR^{3})$ satisfying $\nb \cdot \bfB_{0} = 0$ and $\nrm{\bfB_{0}}_{H^{s}} < \eps$ such that there is no solution to \eqref{eq:hall-mhd} in $L^{\infty}_{t}([0, \dlt]; H^{s})$ with initial data $(\bfu, \bfB)(t=0) = (0, \bfB_{0})$ for any $\dlt > 0$.
\end{enumerate}
\end{theorem}

The precise statements of the main theorems may be found in Section~\ref{subsec:main-results} below. We remark that in the case of \eqref{eq:e-mhd}, we may even take $\bfB_{0}$ to be compactly supported. Our nonlinear illposedness results (Theorems~\ref{thm:illposed-strong-prime} and \ref{thm:illposed-strong2-prime}) are based on quantitative illposedness around \textit{time-dependent degenerate axisymmetric} magnetic fields (Theorem \ref{thm:norm-growth-prime}). 

The remainder of the introduction is structured as follows. After reviewing the setup and results from \cite{JO1} in {\bf Sections~\ref{subsec:stationary}} and {\bf \ref{subsec:linear-review}}, we introduce the aforementioned class of time-dependent degenerate axisymmetric solutions to \eqref{eq:e-mhd} and \eqref{eq:hall-mhd} in {\bf Section~\ref{sec:background-new}}. Then in {\bf Section~\ref{subsec:main-results}}, we state the precise versions of the main theorems, and in {\bf Section~\ref{subsec:ideas}}, we discuss the main ideas of the proof. Finally, in {\bf Section~\ref{subsec:notation}}, we collect some of the notation and conventions used throughout the paper and conclude with an outline of the paper.

\subsection {Energy identities, stationary solutions and linearization} \label{subsec:stationary}
A fundamental property satisfied by \eqref{eq:hall-mhd} and \eqref{eq:e-mhd}  is the energy identity. In the following proposition, it will be assumed that $M = \bbT^{k} \times \bbR^{3-k}$  and the solutions have sufficient smoothness and decay at spatial infinity. 
\begin{proposition} \label{prop:nonlin-en}
	For a solution $(\bfu, \bfB)$ to \eqref{eq:hall-mhd}, we have
	\begin{equation*}
	\frac{\ud}{\ud t} \left( \frac{1}{2} \int_{M} (\abs{\bfu}^{2} + \abs{\bfB}^{2})(t) \, \ud x \ud y \ud z \right) = - \nu \int_{M} \abs{\nb \bfu}^{2}(t) \, \ud x \ud y \ud z.
	\end{equation*}
	Similarly, for a solution $\bfB$ to \eqref{eq:e-mhd}, we have
	\begin{equation*}
	\frac{\ud}{\ud t} \left( \frac{1}{2} \int_{M} \abs{\bfB}^{2}(t) \, \ud x \ud y \ud z \right)= 0.
	\end{equation*}
\end{proposition} 

\noindent The proofs of the above identities are straightforward; one simply multiply the equations for $\bfu$ and $\bfB$ by $\bfu$ and $\bfB$ respectively, integrate over the domain $M$, and use integration by parts. Instead of supplying a proof, we just note that the Hall current term drops out by the symmetric nature of $\nabla\times$:  
\begin{equation*}
\int_{M} \bfB \cdot (\nb \times ((\nb \times \bfB) \times \bfB) \, \ud x \ud y \ud z
= \int_{M} (\nb \times \bfB) \cdot ((\nb \times \bfB) \times \bfB) \, \ud x \ud y \ud z = 0.
\end{equation*} This symmetry in general allows us to gain one derivative back from the Hall current term; if one attempts to control $\nrm{\rd^{(N)} \bfB(t)}_{L^{2}}$ for a solution $\bfB$ to \eqref{eq:e-mhd} ($\rd^{(N)}$ refers to an $N$-th order spatial derivative), we have from the Hall term a contribution of the form
\begin{equation} \label{eq:d-loss}
\frac{1}{2} \frac{\ud}{\ud t} \int_{M} \abs{\rd^{(N)} \bfB}^{2} \, \ud x \ud y \ud z
= - \int_{M} (\nb \times \rd^{(N)} \bfB) \cdot ((\nb \times \bfB) \times \rd^{(N)} \bfB) \, \ud x \ud y \ud z + \cdots
\end{equation}
where the other terms only involve up to $N$ derivatives of $\bfB$. It is not clear how to handle this integral on the right hand side, and the main results of this paper and \cite{JO1} confirms that there are situations where this expression is an actual loss of a derivative. (In contrast to equations for which the Cauchy--Kovalevskaya theorem is applicable, the situation is not better at all in analytic or even in Gevrey regularity classes; see \cite[Section 6]{JO1}.) 


While the systems \eqref{eq:hall-mhd} and \eqref{eq:e-mhd} possess a large set of stationary solutions, a special class of them played a crucial role in \cite{JO1}, namely \textit{planar stationary magnetic fields with an additional symmetry}. We first recall the definition: \begin{itemize}
	\item (Stationary magnetic field) The solution is of the form $\bfB = \bgB$ for \eqref{eq:e-mhd}, and $(\bfu, \bfB) = (0, \bgB)$ for \eqref{eq:hall-mhd}, where $\bgB$ is a $t$-independent vector field on $\bbR^{3}$ such that $\nb \cdot \bgB = 0$ (divergence-free) and $(\nb \times \bgB) \times \bgB$ is a pure gradient.
	
	\item (Planarity) $\bgB$ is independent of the $z$-coordinate and $\bgB^{z}$ = 0. 
	\item (Additional symmetry) $\bgB = \bgB^{x} \rd_{x} + \bgB^{y} \rd_{y}$, viewed as a vector field on $\bbR^{2}_{x, y}$, is invariant under a one-parameter family of isometries of $\bbR^{2}_{x, y}$.
\end{itemize} The following classification result can be shown: \begin{proposition}[{{\cite[Section 2.2]{JO1}}}] \label{prop:planar-stat}
A smooth planar stationary magnetic field with an additional symmetry is, up to symmetries, one of the following forms: ($f, g$ are smooth and $c_{0}, c_{1}, d \in \bbR$)
\begin{equation*}
\bgB = f(y) \rd_{x}, \quad (c_{1} y + c_{0}) \rd_{x} + d \rd_{y}, \quad g(x^{2} + y^{2}) (x \rd_{y} - y \rd_{x}).
\end{equation*}
\end{proposition}

\subsubsection{Linearized systems around stationary solutions} Given a stationary solution to \eqref{eq:hall-mhd} of the form $(0, \bgB)$, let us consider perturbations of the form $(\bfu, \bfB) = (u, \bgB + b)$. The linearized equation satisfied by $(u, b)$ (i.e., the linearization of \eqref{eq:hall-mhd} around $\bgB$) is:
\begin{equation} \label{eq:hall-mhd-lin}
\left\{
\begin{aligned}
& \rd_{t} u - \nu \lap u = \bbP ((\nb \times \bgB) \times b + (\nb \times b) \times \bgB) \\
& \rd_{t} b + \nb \times (u \times \bgB) + \nb \times ((\nb \times b) \times \bgB) + \nb \times ((\nb \times \bgB) \times b) = 0, \\
& \nb \cdot u = \nb \cdot b = 0,
\end{aligned}
\right.
\end{equation}
where $\bbP$ is the Leray projection operator onto divergence-free vector fields. In the case of \eqref{eq:e-mhd}, the linearization is simply given by 
\begin{equation} \label{eq:e-mhd-lin}
\left\{
\begin{aligned}
& \rd_{t} b + \nb \times ((\nb \times b) \times \bgB) + \nb \times ((\nb \times \bgB) \times b) = 0, \\
& \nb \cdot b = 0.
\end{aligned}
\right.
\end{equation} The $L^2$-norm of the perturbation is (formally) under control, for the linearized systems.
\begin{proposition} \label{prop:lin-en}
	For a sufficiently regular and decaying solution $(u, b)$ to \eqref{eq:hall-mhd-lin}, we have
	\begin{equation} \label{eq:lin-en-hall}
	\begin{aligned}
	& \frac{\ud}{\ud t} \left(\frac{1}{2} \int_{M} \abs{u}^{2}(t) + \abs{b}^{2}(t) \, \ud x \ud y \ud z \right) + \nu \int_{M} \abs{\nb u}^{2}(t) \, \ud x \ud y \ud z \\
	& =  \int_{M} ((b \cdot \nb) \bgB_{j}) u^{j} - ((u \cdot \nb) \bgB_{j}) b^{j} \, \ud x \ud y \ud z 
	+ \int_{M} ((b \cdot \nb) (\nb \times \bgB)_{j}) b^{j} \, \ud x \ud y \ud z.
	\end{aligned}
	\end{equation} Similarly, for a sufficiently regular and decaying solution $b$ to \eqref{eq:e-mhd-lin}, we have
	\begin{equation} \label{eq:lin-en-e}
	\frac{\ud}{\ud t} \left(\frac{1}{2} \int_{M} \abs{b}^{2}(t) \, \ud x \ud y \ud z \right)
	=  \int_{M} ((b \cdot \nb) (\nb \times \bgB)_{j}) b^{j} \, \ud x \ud y \ud z.
	\end{equation}
\end{proposition}  

\subsubsection{Notion of an $L^2$-solution.}

The above energy identities stated above suggest that any ``reasonable'' solution to the linearized systems would enjoy good local-in-time $L^2$-bounds. This motivates the following definition for an $L^2$-solution: on an interval of time $I$, an $L^2$-solution means: \begin{itemize}
\item ({linearized \eqref{eq:hall-mhd} with $\nu > 0$}) a pair of vector fields $(u, b)$ such that $u \in C_{w} (I; L^{2}) \cap L^{2}_{t}(I; \dot{H}^{1})$ and $b \in C_{w}(I; L^{2})$ that satisfies \eqref{eq:hall-mhd-lin} in the sense of distributions;
\item ({linearized \eqref{eq:hall-mhd} with $\nu = 0$}) a pair of vector fields $(u, b) \in C_{w}(I; L^{2})$ that satisfies \eqref{eq:hall-mhd-lin} with $\nu = 0$ in the sense of distributions; or
\item ({linearized \eqref{eq:e-mhd}}) a vector field $b \in C_{w}(I; L^{2})$ that satisfies \eqref{eq:e-mhd-lin} in the sense of distributions.
\end{itemize} In the case $M = \bbT^{3}$, we assume that 
\begin{equation} \label{eq:mean-zero}
\int_{M} u(t) = \int_{M} b(t) = 0 \quad \hbox{ for all } t \in I,
\end{equation}
where we ignore the condition for $u$ in the case of \eqref{eq:e-mhd}. Given an initial data $(u_0,b_0) \in L^2$, it is not difficult to show \textit{existence} of an $L^2$-solution $(u,b)$ for \eqref{eq:hall-mhd-lin} defined in the sense above (similarly, an $L^2$-solution $b$ given $b_0 \in L^2$ for \eqref{eq:e-mhd-lin}), see \cite[Appendix A]{JO1}. We can further demand that the solution satisfies \begin{equation*}
\begin{split}
	\frac{1}{2}\left(\nrm{u(t)}_{L^2(M)}^2 + \nrm{b(t)}_{L^2(M)}^2\right) + \nu \nrm{u}_{L^2([0,t];\dot{H}^1)}^2 \le \frac{1}{2}\left(\nrm{u_0}_{L^2(M)}^2 +  \nrm{b_0}_{L^2(M)}^2 \right) e^{ Ct\nrm{\nabla\bgB}_{C^1(M)}}
\end{split}
\end{equation*} for all $t > 0$ and  \begin{equation*}
\begin{split}
	\frac{1}{2} \nrm{b(t)}_{L^2(M)}^2 \le \frac{1}{2}  \nrm{b_0}_{L^2(M)}^2   e^{ Ct\nrm{\nabla^2\bgB}_{L^\infty(M)}}
\end{split}
\end{equation*} for \eqref{eq:hall-mhd-lin} and \eqref{eq:e-mhd-lin}, respectively. These estimates are natural in view of Proposition \ref{prop:lin-en}. For the simplicity of presentation we shall assume that our $L^2$-solution satisfies these additional estimates.

\subsection{Illposedness results under $z$-independence} \label{subsec:linear-review} Let us review the linear illposedness results under the assumptions of $z$-independence. We shall take our domain to be periodic in $z$ and assume that all the functions are independent of $z$. In particular, $H^{m}$ and $\nb$ are identified with $H^{m}_{x,y}$ and $\nb_{x,y}$, respectively.

\subsubsection{Assumptions on the background magnetic fields.}\label{subsubsec:background} Two distinct classes of stationary magnetic fields were considered in \cite{JO1}, translation symmetric and axisymmetric ones. In the former, we take the domain $M_2 = (\mathbb{T},\mathbb{R})_x\times (\mathbb{T},\mathbb{R})_y $\footnote{The notation $(\bbT, \bbR)_{x}$ means that both $\bbT_{x}$ and $\bbR_{x}$ are allowed. } and set $\bgB = f(y)\rd_x.$ In the latter, we take $ M_2 = \mathbb{R}_{x,y} $ and $\bgB = f(r)\rd_\theta = f(\sqrt{x^2+y^2})(x\rd_y-y\rd_x).$ In both cases, it was assumed that $f$ and its derivatives are uniformly bounded and $f$ is \textit{linearly degenerate} in the sense that there exists $r_{0} \ne 0$ such that $f(r_{0}) = 0$, $f'(r_{0}) \ne 0$ (in the axisymmetric case). 

\subsubsection{Linear illposedness: Growth of $L^2$-solutions in higher Sobolev spaces.}
We are now in a position to recall the main linear result from \cite{JO1}: 
\begin{customthm}{A}[Sharp norm growth] \label{thm:norm-growth}
	Take the stationary, planar, and linearly degenerate magnetic field $\bgB$ on $M = M_2 \times \bbT_z$ as described in the above. Let $s \ge 0, p \ge 2$ satisfy $s + 1/p > 1/2$. Then the following statements hold.
	\begin{enumerate}
		\item Consider the linearized \eqref{eq:hall-mhd} with $\nu \ge 0$ around $\bgB$ on a time interval $I \ni 0$. For each $\lmb \in \bbN$, there exists an initial data set of the form
		\begin{itemize}
			\item (Case (a): translationally-symmetric background)
			\begin{equation*} 
			u_{0} = 0, \quad b_{0} = \Re (e^{i (\lmb x + \lmb \frkG(y))}) \frkb(x,y)  
			\end{equation*}
			where $G(y) \in C^{\infty}((\bbT, \bbR)_{y})$ and $\frkb(x, y) \in \calS((\bbT, \bbR)_{x} \times (\bbT, \bbR)_{y})$ with compact support in $y$ and either compact support in $x$ or real-analyticity in $x$; or
			\item (Case (b): axisymmetric background)
			\begin{equation*}
			u_{0} = 0, \quad b_{0} = \Re (e^{i (\lmb \tht + \lmb \frkG(r))}) \frkb(r)
			\end{equation*}
			where $G(r) \in C^{\infty}((0, \infty))$ and $\frkb(r) \in \calC^{\infty}((0, \infty))$ with compact support in $r$,
		\end{itemize}
		and $\dlt > 0$ depending only on $\bgB$ such that any corresponding $z$-independent $L^{2}$-solution $(u, b)$ exhibits norm growth of the form
		\begin{equation*}
		\nrm{b(t)}_{W^{s, p}(M)} \geq c_{s,p,\bgB} e^{c_{0}(\bgB) \cdot (s + \frac{1}{2} - \frac{1}{p})\lmb t} \nrm{b_0}_{W^{s,p}(M)}
		\end{equation*}
		for $t \in I$ satisfying $0 \leq t < \dlt$.
		
		\item Consider the linearized \eqref{eq:e-mhd} around the stationary solution $\bgB$ on a time interval $I \ni 0$. For each $\lmb \in \bbN$, there exists an initial data set of the form
		\begin{itemize}
			\item (Case (a): translationally-symmetric background)
			\begin{equation*}
			\quad b_{0} = \Re (e^{i (\lmb x + \lmb \frkG(y)}) \frkb
			\end{equation*}
			where $G(y)$ and $\frkb(x, y)$ are as in part~(1); or
			\item (Case (b): axisymmetric background)
			\begin{equation*}
			b_{0} = \Re (e^{i (\lmb \tht + \lmb \frkG(r))}) \frkb(r)
			\end{equation*}
			where $G(r)$ and $\frkb(r)$ are as in part~(1),
		\end{itemize}
		and $\dlt > 0$ depending only on $\bgB$ such that any corresponding $z$-independent $L^{2}$-solution $b$ on $I$ exhibits norm growth of the form
		\begin{equation*}
		\nrm{b(t)}_{W^{s, p}(M)} \geq c_{s,p,\bgB}  e^{c_{0}(\bgB) \cdot (s + \frac{1}{2} - \frac{1}{p})\lmb t}\nrm{b_0}_{W^{s,p}(M)}
		\end{equation*}
		for $t \in I$ satisfying $0 \leq t < \dlt  $.
	\end{enumerate}
\end{customthm}
The choice of profiles $\frkb$ and $\frkG$ is explicit and will be explained in the next section. Note that in the above statements, $z$-independence was assumed for the $L^2$-solution. This extra assumption can be removed as a simple consequence of the results from this paper. 

\subsubsection{Nonlinear illposedness results under $z$-independence} We now describe streamlined versions of main nonlinear illposedness results obtained in the work \cite{JO1}. There are two main statements. The first shows \textit{unboundedness} of the solution operator, assuming that the solution is well-defined. This in particular shows failure of the continuity of the solution map. The second result shows \textit{nonexistence} for the Cauchy problem, which may be considered as the strongest notion of illposedness. 

We shall need some additional notations to describe the results: let us denote a function-space ball of radius $\eps$ with respect to a norm $\nrm{\cdot}_{X}$ centered at $\bfx$ by
\begin{equation*}
\calB_{\eps}(\bfx; X) = \set{\bfy \in \bfx + X : \nrm{\bfy - \bfx}_{X} < \eps},
\end{equation*}
and its restriction to compactly supported functions by
\begin{equation*}
\calB_{\eps}(\bfx; X_{comp}) = \set{\bfy \in \calB_{\eps}(\bfx; X) : \bfy - \bfx \hbox{ has compact support in $M$}}.
\end{equation*}

\begin{customthm}{B}[Unboundedness of the solution map] \label{thm:illposed-strong}
	Let $M = (\bbT, \bbR)_{x} \times (\bbT, \bbR)_{y} \times \bbT_{z}$ and the stationary magnetic field $\bgB$ is given either by $f(y) \rd_{x}$ or $f(r)\rd_{\theta}$ as in Theorem \ref{thm:norm-growth}. Assume that for some $\eps, \dlt, r, s, s_{0} > 0$, the solution map for \eqref{eq:hall-mhd} (resp. \eqref{eq:e-mhd}) exists as a map 
	\begin{gather*}
	\calB_{\eps}((0, \bgB); H^{r}_{comp} \times H^{s}_{comp})  \to L^{\infty}_{t} ([0, \dlt]; H^{s_{0}-1}) \times L^{\infty}_{t} ([0, \dlt]; H^{s_{0}})  \\
	\left( \hbox{resp. } \calB_{\eps}(\bgB; H^{s}_{comp}) \to L^{\infty}_{t} ([0, \dlt]; H^{s_{0}}) \right).
	\end{gather*}Then this solution map is unbounded for  $s_0 \ge 3$. 
	
	Moreover, since $\bgB$ can be taken arbitrarily close to 0 in $H^s_{comp}$, the same statement holds with $\bgB \equiv 0$; that is, the solution map, if well-defined, is unbounded near the trivial solution for $s_0 \ge 3$. 
\end{customthm}   

We now state the nonexistence statement.  

\begin{customthm}{C}[Nonexistence near $0$]\label{thm:illposed-strong2}
	Let $s > 7/2$ and  {$M = (\bbT, \bbR)_{x} \times \bbR_{y} \times \bbT_{z}$ in the case of \eqref{eq:hall-mhd} and $M = (\bbT, \bbR)_{x} \times (\bbT, \bbR)_{y} \times \bbT_{z}$ for \eqref{eq:e-mhd}}. Given any $\epsilon > 0$, there exist initial data $(\mathbf{u}_0,\mathbf{B}_0) \in H^{s-1}_{comp} \times H^s_{comp}$ for  \eqref{eq:hall-mhd} satisfying $\nrm{\mathbf{u}_0}_{H^{s-1}} + \nrm{\mathbf{B}_0}_{H^s} < \epsilon$ (resp.~$\mathbf{B}_0 \in H^s_{comp} $ for \eqref{eq:e-mhd} satisfying $\nrm{\mathbf{B}_0}_{H^s} < \epsilon$) such that for any $\delta > 0$, there is no corresponding $ L^\infty_t([0,\delta];H^{s-1} \times H^s)$ solution to \eqref{eq:hall-mhd} (resp.~$L^\infty_t([0,\delta];H^s)$ solution to \eqref{eq:e-mhd}). \end{customthm}

\subsection{Time-dependent axisymmetric magnetic fields and linearized equations} \label{sec:background-new}
In this work, we focus on axisymmetric background magnetic fields (rather than translationally symmetric ones), as they are directly responsible for illposedness in $\bbR^{3}$. In the axisymmetric case, removing the assumption of  $z$-independence forces one to work with time-dependent magnetic fields. 


To this end, take an axisymmetric vector field of the form $\bfB = \Pi(r,z)\rd_\theta = r \Pi(r,z)e_\theta $. Then using that $\nabla\times (\Pi\rd_\theta) = -r\rd_z \Pi \rd_r + r^{-1}\rd_r(r^2\Pi)\rd_z$, we compute \begin{equation*} 
\begin{split}
(\bfB\cdot\nabla)(\nabla\times\bfB) & = -r\Pi\rd_z\Pi e_\theta =  -\Pi\rd_z \Pi \rd_\theta 
\end{split}
\end{equation*} and \begin{equation*} 
\begin{split}
((\nabla\times\bfB)\cdot\nabla)\bfB & = (- \rd_z \Pi \rd_r(r\Pi) + r^{-1}\rd_r(r^2\Pi)\rd_z\Pi)\rd_\theta = \Pi\rd_z \Pi \rd_\theta ,
\end{split}
\end{equation*} so that $\bfB(t) = \Pi(t,r,z)\rd_\theta$ satisfies \eqref{eq:e-mhd} if $\Pi$ solves the inviscid Burgers' equation: \begin{equation}\label{eq:e-mhd-axisym-background} 
\begin{split}
\rd_t \Pi - 2\Pi\rd_z\Pi = 0. 
\end{split}
\end{equation} Let us move on to the Hall-MHD case. Unlike the translationally symmetric case, now that $(\nabla\times\bfB)\times\bfB$ is not a gradient, we cannot keep $\bfu = 0$. However, it is still possible to propagate that velocity is axisymmetric and swirl-free, i.e. $\bfu \cdot e^\theta = 0$. In this case, it is convenient to introduce $\bfomg = \nabla\times\bfu =: \Omega(r,z) \rd_\theta$; we obtain that $(\Omega,\Pi)$ is a solution to Hall-MHD system if \begin{equation}\label{eq:hall-mhd-axisym-background}
\left\{
\begin{aligned}
& \rd_t\Omega + (V^r\rd_r + V^z\rd_z)\Omega + 2 \Pi\rd_z\Pi = \nu(\rd_r^2 + \frac{3}{r}\rd_r + \rd_z^2) \Omega  \\
& \rd_t\Pi + (V^r\rd_r + V^z\rd_z)\Pi - 2 \Pi\rd_z\Pi = 0. 
\end{aligned}
\right.
\end{equation} This system (as well as \eqref{eq:e-mhd-axisym-background}) seems to have appeared for the first time in \cite{CWe}, from where we have taken the notation $(\Omega,\Pi)$. It is to be supplemented with the following relation between $\Omega$ and $V^r,V^z$: \begin{equation}\label{eq:hall-mhd-axisym-BS} 
\left\{
\begin{aligned}
& \rd_z V^r - \rd_r V^z= r\Omega,  \\
& \rd_r V^r + \frac{1}{r} V^r + \rd_z V^z = 0 . 
\end{aligned}
\right.
\end{equation} As in \cite{CWe,Lei}, it will be convenient to introduce a stream function $\Phi$, such that \begin{equation*} 
\begin{split}
-(\rd_r^2 + \frac{3}{r}\rd_r + \rd_z^2)\Phi = \Omega, \quad V^r = -r\rd_z\Phi, \quad V^z = \frac{1}{r}\rd_r(r^2\Phi). 
\end{split}
\end{equation*} Existence, uniqueness, and regularity of $\Phi$ satisfying the above can be seen by observing that the second-order operator on the left hand side is simply $-\lap$ in $\bbR^5$ (see \cite{CWe}). With some abuse of notation, from now on we shall refer to a solution of \eqref{eq:e-mhd-axisym-background} by $\mathbf{\Pi} = \Pi \rd_\theta$. Similarly, given a solution $(\Omega,\Pi)$ of  \eqref{eq:hall-mhd-axisym-background}, the corresponding axisymmetric solution of \eqref{eq:hall-mhd} will be written as $(\bfV, \mathbf{\Pi})$, with $\bfV = V^r \rd_r  +  V^z \rd_z. $

\begin{remark}
	In the translationally symmetric case, a parallel magnetic field $\bfB = F(y,z)\rd_x$  defines a stationary solution to \eqref{eq:e-mhd}. Similarly, $(\bfu, \bfB) = (0, F\rd_x)$ defines a stationary solution to \eqref{eq:hall-mhd}. While one may repeat the analysis in the current paper to the case of $\bfB = F(y,z)\rd_x$, we omit this case entirely for the sake of brevity.  
\end{remark}

\subsubsection{Linearized equations around time-dependent backgrounds} \label{subsec:linearizations}

Given a stationary solution of the form $(0,\bgB)$ ($\bgB$, resp.) for \eqref{eq:hall-mhd} (\eqref{eq:e-mhd}, resp.), the corresponding linearization and energy identity were given in \eqref{eq:hall-mhd-lin} and \eqref{eq:lin-en-hall} (\eqref{eq:e-mhd-lin} and \eqref{eq:lin-en-e}, resp.). In the case of the time-dependent axisymmetric background $(\bfV, \bfPi)$   for \eqref{eq:hall-mhd}, the linearized systems are more complicated. With perturbations of the form $(\bfu,\bfB) = (\bfV + u, \bfPi + b)$, the system is given by 
\begin{equation}\label{eq:axisym-hall-pert}
\left\{
\begin{aligned}
& \rd_t b + \bfPi \cdot\nabla (\nabla\times b) - (\nabla\times b) \cdot \nabla \bfPi - (\nabla \times \bfPi) \cdot \nabla b + \bfV\cdot \nabla b - \bfPi\cdot\nabla u \\
& \quad  = -b\cdot\nabla(\nabla\times\bfPi) + b\cdot\nabla\bfV - u\cdot\nabla\bfPi  , \\
& \rd_t u + \bfV \cdot \nabla u + \nabla p - \nu\lap u - (\nabla\times b)\times \bfPi = -u\cdot\nabla\bfV + (\nabla\times\bfPi)\times b ,  \\
& \nabla \cdot b = \nabla \cdot u = 0.
\end{aligned}
\right.
\end{equation} We then obtain formally that \begin{equation}\label{eq:lin-en-hall-axi}
\begin{split}
& \frac{\ud}{\ud t} \left(\frac{1}{2} \int_{M} \abs{u}^{2}(t) + \abs{b}^{2}(t) \, \ud x \ud y \ud z \right) + \nu \int_{M} \abs{\nb u}^{2}(t) \, \ud x \ud y \ud z \\
& =  \int_{M} ((b \cdot \nb) \bfPi_{j}) u^{j} - ((u \cdot \nb) \bfPi_{j}) b^{j} \, \ud x \ud y \ud z 
+ \int_{M} ((b \cdot \nb) (\nb \times \bfPi)_{j}) b^{j} \, \ud x \ud y \ud z \\
& \phantom{=} + \int_{M} ((b\cdot\nabla)\bfV_j ) b^j  - ((u\cdot\nabla)\bfV_j ) u^j \, \ud x \ud y \ud z 
\end{split}
\end{equation} for a sufficiently smooth and decaying solution $(u,b)$ to \eqref{eq:axisym-hall-pert}. In the case of \eqref{eq:e-mhd}, the corresponding system is simply given by \eqref{eq:e-mhd-lin} where $\bgB$ is replaced by $\bfPi(t)$. Similarly, the energy identity \eqref{eq:e-mhd-lin} holds by replacing $\bfB$ with $\bfPi(t)$, for a sufficiently regular solution $b$. For completeness, we write \begin{equation} \label{eq:e-mhd-lin-axi}
\left\{
\begin{aligned}
& \rd_{t} b + \nb \times ((\nb \times b) \times \bfPi) + \nb \times ((\nb \times \bfPi) \times b) = 0, \\
& \nb \cdot b = 0.
\end{aligned}
\right.
\end{equation}   and 
	\begin{equation} \label{eq:lin-en-e-axi}
	\frac{\ud}{\ud t} \left(\frac{1}{2} \int_{M} \abs{b}^{2}(t) \, \ud x \ud y \ud z \right)
	=  \int_{M} ((b \cdot \nb) (\nb \times \bfPi)_{j}) b^{j} \, \ud x \ud y \ud z.
	\end{equation} 
We now naturally extend the notion of $L^2$-solution to this setting of time-dependent background: on an interval of time $I$, we demand that \begin{itemize}
	\item ({linearized \eqref{eq:hall-mhd} with $\nu > 0$}) a pair of vector fields $(u, b)$ such that $u \in C_{w} (I; L^{2}) \cap L^{2}_{t}(I; \dot{H}^{1})$ and $b \in C_{w}(I; L^{2})$ that satisfies \eqref{eq:axisym-hall-pert} in the sense of distributions;
	\item ({linearized \eqref{eq:hall-mhd} with $\nu = 0$}) a pair of vector fields $(u, b) \in C_{w}(I; L^{2})$ that satisfies \eqref{eq:axisym-hall-pert} with $\nu = 0$ in the sense of distributions; or
	\item ({linearized \eqref{eq:e-mhd}}) a vector field $b \in C_{w}(I; L^{2})$ that satisfies \eqref{eq:e-mhd-lin-axi} in the sense of distributions.
\end{itemize} Again, in the case $M = \bbT^{3}$, we assume \eqref{eq:mean-zero}. 
where we ignore the condition for $u$ in the case of \eqref{eq:e-mhd}. Given an initial data $(u_0,b_0) \in L^2$, it is not difficult to show {existence} of an $L^2$-solution $(u,b)$ for \eqref{eq:axisym-hall-pert} defined in the sense above (similarly, an $L^2$-solution $b$ given $b_0 \in L^2$ for \eqref{eq:e-mhd-lin-axi}), using an argument parallel to the one given in \cite[Appendix A]{JO1}. In view of the identities \eqref{eq:lin-en-hall-axi} and \eqref{eq:lin-en-e-axi}, it is reasonable to demand further that  
\begin{equation}\label{eq:apriori-hall-axi}
\begin{split}
&\frac{1}{2}\left(\nrm{u(t)}_{L^2(M)}^2 + \nrm{b(t)}_{L^2(M)}^2\right) + \nu \nrm{u}_{L^2([0,t];\dot{H}^1)}^2  \le \frac{1}{2}\left(\nrm{u_0}_{L^2(M)}^2 +  \nrm{b_0}_{L^2(M)}^2 \right) e^{ Ct(\nrm{\nabla\bfPi}_{L^\infty_I W^{1,\infty}} + \nrm{\nabla\bfV}_{L^\infty_I L^\infty})}
\end{split}
\end{equation} for all $t > 0$ and  \begin{equation}\label{eq:apriori-e-axi}
\begin{split}
\frac{1}{2} \nrm{b(t)}_{L^2(M)}^2 \le \frac{1}{2}  \nrm{b_0}_{L^2(M)}^2   e^{ Ct\nrm{\nabla^2\bfPi}_{L^\infty_IL^\infty(M)}}
\end{split}
\end{equation} for \eqref{eq:axisym-hall-pert} and \eqref{eq:e-mhd-lin-axi}, respectively. We shall assume these estimates for an $L^2$-solution for simplicity. 

\subsection{Main results}\label{subsec:main-results}

We are now ready to state the main results in this work, which are natural extensions of Theorems \ref{thm:norm-growth}--\ref{thm:illposed-strong2} to the case of $\bbR^3$. 

\subsubsection{Linear illposedness} Let $b_{0} = b_{0}^{(\lmb)}$ and $h_{0} = h_{0}^{(\lmb)}$ be sequences of functions parameterized by $\lmb$. For simplicity, let us write \begin{equation*} 
\begin{split}
&b_{0} = h_0( 1 + O(\lmb^{-1}))
\end{split}
\end{equation*} to denote that the support of $b_0$ is contained in that of $h_0$, and $\nrm{b_0 - h_0}_{W^{s,p}} \lesssim_{s,p} \lmb^{-1} \nrm{h_0}_{W^{s,p}}$ for any $s \ge 0$ and $1 <p<\infty$. 

\begin{customthm}{A*}[Norm growth for the linearized (E-MHD) and  (Hall-MHD)]\label{thm:norm-growth-prime}
	Take $M = \bbR^{3}$ and let $s \ge 0, p \ge 2$ satisfy $s + 1/p > 1/2$. Fix some smooth cutoff function $\chi(z)$ satisfying $\chi(z)=1$ for $|z|\le1$. 
	
	\begin{enumerate}
		\item  Case of (E-MHD): Take a time interval $I \ni 0$ and a time-dependent axisymmetric solution $\bfPi$ to \eqref{eq:e-mhd} defined on $I$, such that the initial data takes the form $\bfPi_0(r,z)= f(r)\rd_\tht$ on $|z|\le1$ where $f(r)$ satisfies the assumptions from \ref{subsubsec:background}. We consider the linearized \eqref{eq:e-mhd} \eqref{eq:e-mhd-lin-axi} on around $\bfPi$.

		Then, there exist $\dlt>0$ and a compactly supported smooth function $\chi$ depending only on $\bfPi_{0}$ such that with  $\frkG(r)$ and $\frkb(r)$ as in Theorem \ref{thm:norm-growth}, for $\lmb \in \bbN$, there exists an initial data set of the form \begin{equation} \label{eq:E-MHD-id}
			\quad b_{0}^{(\lmb)} = \Re (e^{i (\lmb \tht + \lmb \frkG (r))}) \frkb(r)\chi(z)  ( 1 + O(\lmb^{-1}))
		\end{equation} such that any corresponding $L^{2}$-solution $b^{(\lmb)}$ on $I$ exhibits norm growth of the form \begin{equation*}
			\nrm{b^{(\lmb)}(t)}_{W^{s, p}(M)} \geq c_{s,p,\bfPi_0}  e^{c_{0}(\bfPi_0) \cdot (s + \frac{1}{2} - \frac{1}{p})\lmb t}\nrm{b_{0}^{(\lmb)}}_{W^{s,p}(M)} \quad \mbox{for all} \quad t \in I \cap [0,\dlt]. 
		\end{equation*} 
		\item Case of (Hall-MHD): Take a time-dependent axisymmetric solution $(\bfV, \bfPi)$ of \eqref{eq:hall-mhd} defined on $I$, such that the initial data is of the form $(0,\bfPi_0)$ where $\bfPi_0$ is as in Case (1). Then, for each $\lmb \in \bbN$, any $L^2$-solution $(u^{(\lmb)},b^{(\lmb)})$ corresponding to the initial data $(0, b_{0}^{(\lmb)})$ ($b_{0}^{(\lmb)}$ as in Case (1)) satisfies \begin{equation*} 
			\begin{split}
				& \nrm{b^{(\lmb)}(t)}_{W^{s, p}(M)} \geq c_{s,p,\bfPi_0}  e^{c_{0}(\bgB) \cdot (s + \frac{1}{2} - \frac{1}{p})\lmb t}\nrm{b_{0}^{(\lmb)}}_{W^{s,p}(M)} \quad \mbox{for all} \quad t \in I \cap [0,\dlt_{1}],
			\end{split}
		\end{equation*} where $\dlt_{1} = c_1 \lmb^{-1}\ln\lmb $ for some $c_1 > 0$ depending only on $\bfPi_0$ and $\nu\ge0$. 
	\end{enumerate} 
\end{customthm}
 
\begin{remark}
	Unfortunately, in the Hall-MHD case, the growth of $b(t)$ is guaranteed only in the time interval of length $O(\lmb^{-1}\ln\lmb)$, see \ref{subsubsec:apriori-idea}. This is still sufficient to obtain unbounded growth as $\lmb\to\infty$. 
\end{remark}

\subsubsection{Nonlinear illposedness}

We now present our nonlinear illposedness results on $\bbR^{3}$. 

\begin{customthm}{B*}[Unboundedness of the solution map] \label{thm:illposed-strong-prime}
	Let $M = \bbR^{3}$ and the axisymmetric solution  $(\bfV,\bfPi)$ (resp. $\bfPi$) is given as in the case (2) (resp. case (1)) of Theorem \ref{thm:norm-growth-prime}.  Assume that for some $\eps, \dlt, s, r,  s_{0} > 0$, the solution map for \eqref{eq:hall-mhd} (resp. \eqref{eq:e-mhd}) exists as a map 
	\begin{gather*}
	\calB_{\eps}((0, \bfPi_0); H^{r}_{comp} \times H^{s}_{comp})  \to L^{\infty}_{t} ([0, \dlt]; H^{s_{0}-1}) \times L^{\infty}_{t} ([0, \dlt]; H^{s_{0}})  \\
	\left( \hbox{resp. } \calB_{\eps}(\bfPi_0; H^{s}_{comp}) \to L^{\infty}_{t} ([0, \dlt]; H^{s_{0}}) \right).
	\end{gather*} In the \eqref{eq:hall-mhd} case, assume further that $s<(1+\alp_0)s_{0}$ where $\alp_{0}>0$ is an absolute constant. Then this solution map is unbounded for  $s_0 > 7/2$, even if we restrict the initial data to be $C^{\infty}_{comp}$. Moreover,  the same statement holds with $\bfPi_0 \equiv 0$; that is, the solution map cannot be bounded near the trivial solution for $s_0 > 7/2$. 
\end{customthm}

\begin{remark}[Norm inflation]
	One may rephrase the above statement as follows: in the \eqref{eq:e-mhd} case, assuming that the solution operator is well-defined from $H^{s}$ to $L^\infty([0,\dlt];H^{s_{0}})$, for any $\varepsilon, M, \dlt> \dlt' > 0$, there exists an initial data $\bfB_0 \in C^\infty_{comp}(\bbR^3)$ with corresponding solution $\bfB(t)$ defined on $[0,\dlt]$ such that \begin{equation*}
		\begin{split}
			\nrm{\bfB_0}_{H^{s}}< \varepsilon, \quad \nrm{\bfB(\dlt')}_{H^{s_{0}}} > M. 
		\end{split}
	\end{equation*} Note that $s$ can be arbitrarily large independently of $s_{0}$, as long as $s_{0}>7/2$. 
\end{remark}


\begin{customthm}{C*}[Nonexistence]\label{thm:illposed-strong2-prime}
	Let $s > 7/2$ and $M = \bbR^{3}$. Given any $\epsilon > 0$, there exists an initial data $\mathbf{B}_0 \in H^s_{comp}$  satisfying  {$\nb \cdot \bfB_{0} = 0$ and} $\nrm{\mathbf{B}_0}_{H^s} < \epsilon$ such that for any $\dlt>0$, there is \emph{no} associated solution to \eqref{eq:e-mhd} belonging to $L^\infty_t([0,\delta];H^s)$. In the \eqref{eq:hall-mhd} case, there is $\mathbf{B}_{0} \in H^{s}$ satisfying  {$\nb \cdot \bfB_{0} = 0$ and} $\nrm{\mathbf{B}_0}_{H^s} < \epsilon$ such that for any $\dlt>0$, there is \emph{no} solution to \eqref{eq:hall-mhd} belonging to $ {L^\infty_t([0,\delta];H^{s-1} \times H^s)}$ corresponding to the initial data 
	$(\mathbf{u}_0 \equiv 0,\mathbf{B}_0)$. \end{customthm}

Some remarks about these theorems are in order.
\begin{remark}[Degenerate dispersion]
As in \cite{JO1}, the main illposedness mechanism exploited in this paper is \emph{degenerate dispersion}, i.e., the rapid growth of frequencies of solutions that travel towards a degeneracy of the principal term. This mechanism may be found in other classes of quasilinear dispersive equations as well; see \cite{CJO, JO4}.
\end{remark}

\begin{remark}[Wellposedness in the nondegenerate case]
In \cite{JO2-2}, the following local well(!)posedness theorem for \eqref{eq:e-mhd} on $\bbR^{3}$ was proved:
\begin{quote}
\it 
Let $s > \frac{7}{2}$ and consider a vector field $\bfB_{0} : \bbR^{3} \to \bbR^{3}$ satisfying $\nb \cdot \bfB_{0} = 0$. Assume  furthermore that $\bfB_{0}$ satisfies the following properties:
\begin{enumerate}
\item {\bf Nondegeneracy.} $\bfB_{0}(x) \neq 0$ at every point $x \in \bbR^{3}$,
\item {\bf Asymptotic uniformity.} $\nrm{\bfB_{0} - \bfe_{3}}_{\ell^{1}_{\calI} H^{s}} < + \infty$,
\item {\bf Nontrapping.} Every nonconstant solution $(X, \Xi)(t)$ to the Hamiltonian system associated with $\bfp_{\bfB_{0}}(x, \xi) = \bfB_{0}(x) \cdot \xi \abs{\xi}$ escapes to $x^{3} = \pm \infty$, i.e., $X^{3}(t) \to \infty$ or $X^{3}(t) \to - \infty$ as $t \to \infty$.
\end{enumerate}
Then the Cauchy problem for \eqref{eq:e-mhd} with $\bfB(t=0) = \bfB_{0}$ is locally wellposed. 
\end{quote}  
We remark that the techniques in \cite{JO2-2} can be used to prove a similar theorem for \eqref{eq:hall-mhd} as well; see \cite[Remark~1.3]{JO2-2}. The theorems in this paper complement this result by exhibiting a strong nonlinear illposedness phenomenon in a neighborhood of certain \emph{degenerate} initial data (e.g., $\bgB = 0$).

Observe that the above theorem from \cite{JO2-2} requires $\bfB_{0}$ to be asymptotic to the uniform magnetic field $\bfe_{3}$. Another natural setting where local wellposedness may hold is when $\bfB_{0}$ close to a \emph{nondegenerate} axisymmetric magnetic field. Such a result would justify the axisymmetric ansatz under a nondegeneracy assumption, which is interesting in view of the finite time blow-up result of Chae--Weng \cite{CWe} for the axisymmetric \eqref{eq:hall-mhd}. However, the proof in \cite{JO2-2} fails due to the difficulty of establishing local smoothing effect in this setting.
\end{remark}

\begin{remark}
	The nonlinear illposedness results can be similarly stated and proved in the scale of $C^{k,\alpha}$-spaces as long as $k + \alpha \ge 2$ for $\bfB$ in the \eqref{eq:e-mhd} case. In the \eqref{eq:hall-mhd} case, it is natural to put $(\bfu,\bfB) \in C^{k-1,\alpha}\times C^{k,\alpha}$ with $k+\alpha \ge 2$. 
\end{remark}

\subsection{Ideas of the proof} \label{subsec:ideas}

In this section, we discuss key ideas of the proof, mainly emphasizing the additional difficulties arising from having compact $z$-support. Main ideas in the $z$-independent case are given in \cite[Sections 1.6, 1.7]{JO1}.

\subsubsection{Analysis at the level of bicharacteristics.}\label{subsubsec:bicharacteristics} Here we give a heuristic explanation as to why the same illposedness results are expected in the $z$-dependent case as well. This is most easily seen under the (formal) framework of bicharacteristics. As argued in \cite{JO1}, our basic viewpoint towards \eqref{eq:hall-mhd} is to regard it as a quasilinear dispersive system, and the basic mechanism for the illposedness is given by the existence bicharacteristics quickly converging to the \textit{degeneracy} of the background magnetic field. Let us recall the notion of bicharacteristics: Linearizing around a stationary magnetic field $\bgB$ and neglecting the equation for the velocity, we may write the resulting equation in the form \begin{equation*}
\begin{split}
 \rd_tb + (\bgB\cdot\nabla)\nabla\times b = l.o.t.,
\end{split}
\end{equation*} where the right hand side contains either first or zeroth order terms in $b$. We may then diagonalize the matrix-valued symbol $\bfp = -(\bgB\cdot\xi)\xi\times$ on the subspace $\{ \zeta \in \bbR^3 : \xi\cdot\zeta = 0 \}$ with eigenvalues $\pm p = \pm i \bgB \cdot\xi |\xi|$.  Then neglecting the lower order terms, the above vector equation splits into two scalar dispersive PDE of the form \begin{equation*}
\begin{split}
\rd_t b_{\pm} \pm ip (i^{-1}\nabla)b_\pm = 0.
\end{split}
\end{equation*} Then the associated ODE system with the Hamiltonian vector field $(\nabla_\xi p, -\nabla_{\mathbf{x}} p)$ on $T^*M$ \begin{equation*}
\left\{\begin{aligned}
\dot{X} & = \nabla_\xi p(X,\Xi),\\
\dot{\Xi} & = -\nabla_{\mathbf{x}}  p(X,\Xi)
\end{aligned}
\right. 
\end{equation*} is referred to as the Hamiltonian ODE, and its solution $(X,\Xi)(t)$ a bicharacteristic. This is a natural generalization of the notion of the group velocity from constant coefficient dispersive PDEs. Given a Gaussian-like initial data with physical and frequency centers $(X_0,\Xi_0)$, the associated bicharacteristics $(X,\Xi)(t)$ describes the evolution of the (approximate) centers, at least for some period of time. 

We now take the concrete choice $\bgB = y\chi(z)\rd_x$ where $\chi(\cdot) \ge 0 $ is a smooth cutoff of $O(1)$ length scale. The associated Hamiltonian is $y(X)\chi(z(X))\Xi_x|\Xi|$ and since $\Xi_x$ is conserved along bicharacteristics (from $x$-invariance of $\bgB$), we conclude that $y(X)\chi(z(X))|\Xi|$ is conserved as well. Hence as in the case $\chi \equiv 1$, $y \rightarrow 0$ will imply $|\Xi| \rightarrow +\infty$. 

Towards this goal we first note that, as long as $|z(X)| \lesssim 1$, the Hamiltonian simply coincides with the one from the $z$-independent case. Take initial data of the form $X(0) = (0,1,z(0))$ ($x(0) = 0$ can be taken without loss of generality) and $\Xi(0) = (\lmb,-\lmb, \Xi_z(0))$ where $|z(0)|, |\Xi_z(0)| \lesssim 1$. As long as $|z(X)| \lesssim 1$, we have \begin{equation*}
\left\{\begin{aligned}
\dot{z} & = \lmb y \frac{\Xi_z}{|\Xi|}, \\
\dot{\Xi_z} & = 0 ,
\end{aligned}
\right. 
\end{equation*} and hence we see that if we have $|y(X)| \lesssim 1$ and $|\Xi_z|/|\Xi| \lesssim \lmb^{-1}$ for some interval of time, we can guarantee that $|z(X)| \lesssim 1 $  as well. But then this means that the behavior of the bicharacteristics is the same as in the $z$-independent case; in particular, \begin{equation*}
\begin{split}
|\Xi(t)| \simeq \lmb e^{\lmb t},\quad y(t) \simeq e^{-\lmb t}
\end{split}
\end{equation*} which in turn guarantees $|y(X)| \ll 1$ and $|\Xi_z|/|\Xi| \ll\lmb^{-1}$! (This circular argument can be made rigorous with a continuity argument.) This shows that the bicharacteristics from the $z$-independent case still provides a good approximation within a time interval which is \textit{independent} of $\lmb$, as long as the initial $z$-frequency is kept at $O(1)$. Essentially the same analysis can be repeated for the axisymmetric magnetic field of the form $f(r)\chi(z)\rd_\tht$ with the caveat that it actually does not define a stationary solution to \eqref{eq:hall-mhd}; see \ref{subsubsec:apriori-idea} below. This heuristic argument suggests that given a $z$-independent linear approximate solution, a simple cutoff in the $z$-direction still gives an approximate solution. Of course, this procedure gives rise to various technical difficulties, which shall be discussed in the following. 

\subsubsection{Generalized energy identities.}\label{subsubsec:gei} Our main technical tool in the illposedness arguments (both linear and nonlinear) is the so-called \textit{generalized energy identity}, which is a deceptively simple consequence of the (usual) energy identity. Let us briefly explain the ideas. Consider a linear equation of the form \begin{equation}\label{eq:linear-toy}
\begin{split}
\rd_t b + \calL[ b] = \calB [b] , 
\end{split}
\end{equation} where $\calL$ is antisymmetric and $\calB$ is a zeroth order operator; \begin{equation*}
\begin{split}
 \brk{\calL [f], g} = - \brk{f,\calL [g]}, \quad \nrm{\calB[f]}_{L^2} \lesssim \nrm{f}_{L^2} . 
\end{split}
\end{equation*} Then one can see that not only the $L^2$-norm of a solution $b$ is bounded in terms of the initial data, but also does the $L^2$-inner product of two solutions: for $b$ and $\tb$ satisfying \eqref{eq:linear-toy}, we have \begin{equation*}
\begin{split}
\frac{\ud}{\ud t} \brk{b,\tb} = \brk{\calB[b],\tb} - \brk{\calL[b],\tb} + \brk{b,\calB[\tb]} - \brk{b,\calL[\tb]} 
\end{split}
\end{equation*} so that after cancellations, \begin{equation}\label{eq:gee}
\begin{split}
\left| \frac{\ud}{\ud t} \brk{b,\tb} \right| \lesssim \nrm{b}_{L^2}\nrm{\tb}_{L^2} 
\end{split}
\end{equation} which implies (after integrating the above in time and then using that the $L^2$-norm is stable under \eqref{eq:linear-toy}) \begin{equation*}
\begin{split}
\left| \brk{b(t),\tb(t)} - \brk{b_0,\tb_0} \right| \lesssim t\nrm{b_0}_{L^2}\nrm{\tb_0}_{L^2}. 
\end{split}
\end{equation*} Finally, if the ``angle'' between $b_0$ and $\tb_0$ is small; that is, $\brk{b_0,\tb_0} \gtrsim \nrm{b_0}_{L^2}\nrm{\tb_0}_{L^2}$, we conclude that for some interval $I$ of time, \begin{equation}\label{eq:gee2}
\begin{split}
\inf_{t\in I} \, \brk{b(t),\tb(t)} \ge c\nrm{b_0}_{L^2}\nrm{\tb_0}_{L^2}
\end{split}
\end{equation} for some $c > 0$. In the following we shall frequently refer to \eqref{eq:gee} (or its simple consequence \eqref{eq:gee2}) as \textit{generalized energy estimate}. The merit in the generalized energy estimate lies in the flexibility involved in its derivation; say we cannot solve \eqref{eq:linear-toy} directly, but somehow are given an approximate solution $\tb$, satisfying instead \begin{equation*}
\begin{split}
\rd_t \tb + \calL[\tb] = \calB'[\tb]
\end{split}
\end{equation*} where $\calB'$ is some bounded operator in $L^2$ which does not need to be related in any way with $\calB$. But note that in this case, the estimates \eqref{eq:gee} and \eqref{eq:gee2} still hold under the same assumptions. Furthermore, one may observe that we actually do not need $\calB'$ and $\calB$ to be bounded operators in $L^2$, but just need that $\nrm{\calB^*[\tb]}_{L^2}, \nrm{\calB'[\tb]}_{L^2} \lesssim \nrm{\tb}_{L^2}$. In conclusion, as far as one is concerned with deriving the generalized energy estimate, one may regard \textit{a priori} and \textit{a  {posteriori}} bounded terms as errors. 

Having a generalized energy estimate in hand, the task of showing growth of $b$ in high-regularity Sobolev spaces reduces to that of showing \textit{decay} of an approximate solution $\tb$ in low-regularity Sobolev spaces. As a simple example, if we are able to show $\nrm{\tb(t)}_{L^1} \lesssim e^{-t}\nrm{\tb_0}_{L^2}$, we conclude at once from \eqref{eq:gee2} that $\nrm{b(t)}_{L^\infty} \gtrsim e^t \nrm{b_0}_{L^2}$. Of course in general showing degeneration in low regularity spaces could be a difficult task.

So far we have demonstrated that showing growth of a solution to \eqref{eq:linear-toy} can be reduces to construction and analysis of an appropriate approximate solution. There is a general such procedure, commonly called as WKB analysis, which we briefly illustrate in \ref{subsubsec:wkb}. Then we are in a position to discuss \textit{additional} difficulties in proving the generalized energy estimate in the $z$-dependent case, over the $z$-independent case.  

\subsubsection{WKB analysis.}\label{subsubsec:wkb}


Keeping the notation above, say we are interested in constructing an approximate solution $\tb$ for \eqref{eq:linear-toy}. The discussion in the above suggest that we can completely neglect the a priori bounded terms from the beginning and consider $z$-independent solutions; say the latter independence assumption simplifies the operator $\calL$ into $\calL'$; now we want to approximately solve \begin{equation}\label{eq:linear-toy2}
\begin{split}
\rd_t \tb + \calL' \tb = 0. 
\end{split}
\end{equation} We use the standard WKB analysis, which can be viewed as an expansion of the solution in frequency. Take $\lmb \gg 1$ and prepare an ansatz of the following form \begin{equation*}
\begin{split}
\tb = e^{i\lmb \Phi} (h^{(0)} + \lmb^{-1} h^{(1)} + \cdots ). 
\end{split}
\end{equation*} We have a freedom of choice for the initial data of $\Phi, h^{(0)}$, and so on, which we can use to our advantage. Formally plugging in the ansatz into \eqref{eq:linear-toy2} and organizing the terms by the powers of $\lmb$, one obtains a hierarchy of evolution equations; the first one involving only $\Phi$, the second describing the evolution of $h^{(0)}$ with coefficients depending on $\Phi$, and so on. While this inductively determines each function in the expansion, it turns out that in our concrete cases, $e^{i\lmb\Phi}h^{(0)}$ is already provides a good approximate solution, in the $L^2$-sense. 

\subsubsection{Bogovski\v{i} operator and broken antisymmetry.}\label{subsubsec:Bog}

As we have mentioned in the above, we would like to declare that $\chi \tb$ still provides a good approximation, where $\chi$ is a simple cutoff function in $z$, and $\tb$ is $z$-independent. An immediate problem one faces is that $\chi\tb$ cannot be divergence free, a property which is essential for the principal operator $\calL$ to be anti-symmetric to begin with. Moreover, the initial data $b_0$ should be divergence-free, close to $\chi\tb$ at least in the $L^2$-norm, and have compact support. It is tempting to just apply the Leray projector on $\chi\tb$ but in terms of the compact support property, a much better alternative is provided by the Bogovski\v{i} operator introduced in \cite{Bog}. This  is also used in an essential way  in the \eqref{eq:hall-mhd} case; $\chi\tu$ is not divergence-free, so to remove the pressure term, we need to test $u$ against the divergence-free part (defined by the Bogovski\v{i} operator) of $\chi\tu$. 

We still need to explain how the first difficulty is handled. The principal operator in the Hall current term is of second order and involves $\rd_{zx}$ and $\rd_{zy}$ derivatives. While going through the  generalized energy identity, there appear terms with $\rd_x$ and $\rd_y$ on $\tb$ when a $z$-derivatives falls on the cutoff (anti-symmetry is broken). However all such terms altogether cancel out to give rise an $L^2$-bounded term after some juggling of derivatives with the divergence-free condition on $b$. Of course this was anticipated from \ref{subsubsec:bicharacteristics}, or can be alternatively seen by repeating the WKB analysis in the $z$-dependent case. 

\subsubsection{A priori estimates for the axisymmetric Hall-MHD}\label{subsubsec:apriori-idea}

Utilizing the axisymmetric background magnetic fields (over the translationally symmetric ones) is the key to obtain illposedness for compact data in $\bbR^3$. However,  {compared to the setting of \cite{JO1}}, one immediately sees a problem: there are no non-trivial axisymmetric \emph{stationary} magnetic fields for \eqref{eq:e-mhd} with compact support in $\bbR^3$. This forces us to work with time-dependent background solutions, which sounds like a nonsense because in the end, we have proven \textit{nonexistence} of a solution to the initial value problem. Interestingly enough it is known that under the assumption of axisymmetry, existence and uniqueness of local-in-time smooth axisymmetric solutions can be established (\cite{CWe,JKL}).\footnote{This statement should be interpreted carefully, as uniqueness is not guaranteed for \eqref{eq:e-mhd} and \eqref{eq:hall-mhd}, so that there is no reason to believe that the solution is axisymmetric even if the initial data is so. This point is emphasized in \cite{JKL}.} We give a proof of a priori estimates both for the \eqref{eq:e-mhd} and \eqref{eq:hall-mhd} cases, although some cases have been already covered in \cite{JKL}. In the case of \eqref{eq:e-mhd}, the resulting equation for the axisymmetric magnetic field is exactly the (inviscid) Burgers equation in the $z$-direction, which is locally wellposed but definitely blows up in finite time for compact support initial data. This is not an issue for linear illposedness since the growth happens essentially instantaneously. Even though the magnetic field is time-dependent, all we need is that this background is constant in the $z$-direction for some intervals in $t$ and $z$, which holds for solutions to the Burgers equation. In view of this, serious difficulties arise in the \eqref{eq:hall-mhd} case. Even if the background velocity is initially zero, the velocity becomes immediately nonzero by the axisymmetric magnetic field for $t>0$, which in turn breaks $z$-invariance in any small interval of $z$. This is handled by quantifying smallness of the $z$-variation of the magnetic field for small $t>0$. A much more serious problem is that, in general, the velocity field will immediately move the point (a hyperplane in $\bbR^3$, strictly speaking) of degeneracy for the magnetic field. Recall that the illposedness mechanism is the convergence of wave packet solutions towards the degenerate point. The fact that the degenerate point is moving with time destroys the estimates for the wavepackets which were valid for an $O(1)$-interval of time; in the worst possible case, the estimates are only valid for a time interval of size $O(\lmb^{-1}\ln\lmb)$, where $\lmb \gg 1$ is the frequency of the wave packet at the initial time. Even showing this requires proving smallness-in-time estimates for various quantities involving the background axisymmetric solution. 

\begin{remark}
	It is desirable (and could be of independent interest) to have a situation where the degenerate point is forced to be fixed, even under the axisymmetric Hall-MHD evolution. This is the case for axis-degenerate magnetic fields, which are simply steady magnetic fields which are degenerate on the axis $\{ r = 0 \}$. One may perform the analysis of wave packets in this case, which gives admissible error estimates for $O(\lmb^{\alp})$-timescales where $\alp>0$ is a factor depending on the order of axis degeneracy. 
\end{remark}
 
\subsubsection{Nonlinear illposedness}

There is no general method for passing from a linear illposedness result to a corresponding nonlinear result; in certain cases a nonlinear system can be wellposed while it has a linearization which is illposed (\cite{GD}). In our case, we use the (generalized) energy estimate in a crucial way in deducing a nonlinear illposedness result from a linear one. Using a similar notation from the above, the equation for the perturbation takes the form \begin{equation*}
\begin{split}
\rd_t b + \calL[b] = \calB[b] + Q[b,b]
\end{split}
\end{equation*} where we write $Q[b,b]$ for the nonlinear terms. In practice, we have in $Q$ terms of the form $b\nabla^2b$ and $(\nabla b)^2$. Note that upon assuming towards contradiction that $b$ belongs to some $H^s$ with $s$ large, we may bound the contribution from $Q$ in deriving a generalized energy estimate for $b$ using Sobolev inequalities: \begin{equation*}
\begin{split}
\left|\brk{Q[b,b],\tb}\right| \lesssim \nrm{b}_{H^s} \nrm{b}_{L^2}\nrm{\tb}_{L^2} . 
\end{split}
\end{equation*} This allows us to again deduce  \eqref{eq:gee2}, which gives a contradiction to the boundedness of $b$ upon taking the initial frequency support of $b$ to be very large. This simple estimate makes precise the general belief that the nonlinear evolution follows approximately the linear one for a short time interval. In the three-dimensional case, one will see a slight difference in the minimum threshold for $s$ (compared to the $z$-independent case) since one needs to use three-dimensional embedding results rather than two-dimensional ones.  

The above gives an outline of how the unboundedness of the solution operator (assuming existence) can be proved. As we have done in \cite{JO1}, \textit{nonexistence} of the solution operator can be actually shown; the simple idea is to prepare a background solution with countably many points of degeneracy. Then for each degenerate point one puts a small (in some fixed $H^s$ with $s$ large enough) perturbation with increasing frequency. If the same argument as in the above can be repeated independently near each degenerate point, it will show that at some positive time moment (where the solution is assumed to \textit{exist}), the $H^s$-norm should be greater than any large number, which is a contradiction. 

Other than a few minor issues (related with ensuring the divergence-free conditions for some vector functions), there is not much additional difficulties in proving nonlinear illposedness in the $z$-dependent case over the $z$-independent case. On the contrary, having an unbounded direction $\bbR_z$ can be used to our advantage in the nonexistence argument, as we can place countably many points of degeneracy along the $z$-direction which was definitely not possible in our previous work.

\subsection{Notation, conventions and some useful vector calculus identities} \label{subsec:notation}
Here, we collect some notation, conventions and vector calculus identities that will be used freely in the remainder of the paper.

\subsubsection*{Notation and conventions}
By $A \aleq B$, we mean that there exists some positive constant $C > 0$ such that $\abs{A} \leq C B$. The dependency of the implicit constant $C$ is specified by subscripts, e.g. $A \aleq_{E} B$. By $A \aeq B$, we mean $A \aleq B$ and $A \ageq B$.

We denote by $\bbR$ the real line, $\bbZ$ the set of integers, $\bbT = \bbR /  {2 \pi \bbZ}$ the  {torus with length $2 \pi$}, $\bbN_{0} = \set{0, 1, 2, \ldots}$ the set of nonnegative integers and $\bbN = \set{1, 2, \ldots}$ the set of positive integers.

We write $M$ for the $3$-dimensional domain of the form $\bbT^{k} \times \bbR^{3-k}$ ($0 \leq k \leq 3$) equipped with the rectangular coordinates $(x, y, z)$, and $M^{2} = M^{2}_{x, y}$ for the two-dimensional projection of $M$ along the $z$-axis. We use the notation $\brk{u, v}$ and $\brk{u, v}_{M^2}$ for the standard $L^{2}$-inner product for vector fields on $M$ and $M^{2}$, respectively; i.e.,
\begin{equation*}
\brk{\bfu, \bfv} = \int_{M} \bfu \cdot \bfv \, \ud x \ud y \ud z, \quad \brk{u, v}_{M^2} = \int_{M^{2}} u \cdot v \, \ud x \ud y.
\end{equation*}
Given a vector $u$ on $M^{2}$, we define its perpendicular $u^{\perp}$ by $(-u^{y}, u^{x})^{\top}$ and the perpendicular gradient operator $\nb^{\perp}$ by $ (-\rd_{y},\rd_{x})$.
We use the usual notation $W^{s, p}$ for the $L^{p}$-based Sobolev space of regularity $s$; when $p = 2$, we write $H^{s} = W^{s, 2}$. The mixed Lebesgue norm $L^{p}_{x} L^{q}_{y}$ is defined as 
\begin{equation*}
\nrm{u}_{L^{p}_{x} L^{q}_{y}} = \nrm{\nrm{u(x, y)}_{L^{q}_{y}}}_{L^{p}_{x}}.
\end{equation*}
The norm $L^{p}_{t} H^{s}$ is defined similarly. 

Given any space $X$ of functions on $M$, we denote by $X_{comp}(M)$ the subspace of compactly supported elements of $X$, and by $X_{loc}(M)$ the space of functions $u$ such that $\chi u \in X_{loc}(M)$ for any smooth compactly supported function $\chi$ on $M$.

\subsubsection*{Vector calculus identities}
We recall some useful vector calculus identities:
\begin{align} 
\bfU \times (\bfV \times \bfW) & = \bfV (\bfU \cdot \bfW) - \bfW (\bfU \cdot \bfV) , \label{eq:cross-cross} \\
%
%
\nb \times (\bfU \times \bfV) 
& = (\bfV \cdot \nb) \bfU + \bfU (\nb \cdot \bfV) - (\bfU \cdot \nb) \bfV - \bfV (\nb \cdot \bfU) , \label{eq:curl-cross} \\
%
%
(\nb \times \bfU) \times \bfV 
& = (\bfV \cdot \nb) \bfU - \bfV_{j} \nb \bfU^{j} , \label{eq:cross-curl} \\
%
%
\nb \times (\nb \times \bfU) 
& = - \lap \bfU + \nb (\nb \cdot \bfU) . \label{eq:curl-curl}
\end{align}

\subsubsection*{Vector calculus in cylindrical coordinates}
{The cylindrical coordinates $(r, \theta, z)$ are defined by}
\begin{equation*}
r = \sqrt{x^{2} + y^{2}}, \quad \theta = \tan^{-1} \frac{y}{x}.
\end{equation*}
As in \cite{JO1}, we use the \emph{coordinate derivative basis} $(\rd_{r}, \rd_{\tht}, \rd_{z})$ to decompose vectors into components, i.e., given a vector $\bfU$ on $M$, we define its components $\bfU^{r}, \bfU^{\tht}, \bfU^{z}$ by
\begin{equation*}
\bfU = \bfU^{r} \rd_{r} + \bfU^{\tht} \rd_{\tht} + \bfU^{z} \rd_{z}.
\end{equation*} (Another common choice is $(e_{r}, e_{\tht}, e_{z}) = (\rd_{r}, r^{-1} \rd_{\tht}, \rd_{z})$.) With this convention, the inner product is given by 
\begin{equation} \label{eq:dot-cylin}
\bfU \cdot \bfV = \bfU^{r} \bfV^{r} + r^{2} \bfU^{\tht} \bfV^{\tht} + \bfU^{z} \bfV^{z},
\end{equation} where $\bfV = \bfV^{r} + \bfV^\tht \rd_\tht + \bfV^z \rd_z$. 
For a scalar-valued function $f$, the gradient and Laplacian in the cylindrical coordinates are \begin{equation} \label{eq:grad-cylin}
\nb f = (\rd_{r} f) \rd_{r} + (r^{-2} \rd_{\tht} f) \rd_{\tht} + (\rd_z f) \rd_z, 
\end{equation}
\begin{equation} \label{eq:Lap-cylin}
\begin{split}
\lap f = r^{-1}\rd_r(r \rd_rf) + r^{-2}\rd_{\tht\tht} f + \rd_{zz}f. 
\end{split}
\end{equation} When $f$ is independent of $z$, the perpendicular gradient takes the form
\begin{equation} \label{eq:grad-perp-cylin}
\nb^{\perp} f = - (r^{-1} \rd_{\tht} f) \rd_{r} + (r^{-1} \rd_{r} f) \rd_{\tht}.
\end{equation} Next, the curl and divergence in the cylindrical coordinates are
\begin{equation} 
\label{eq:curl-cylin}
\begin{aligned}
\nabla \times (\bfU^r{\rd_r}+\bfU^\theta {\rd_\theta} + \bfU^z{\rd_{z}}) =& (r^{-1}\rd_\theta \bfU^z - r \rd_{z} \bfU^{\tht}){\rd_r} + (r^{-1} \rd_{z} \bfU^{r} - r^{-1} \rd_r \bfU^z) {\rd_\theta} \\
& + r^{-1}(\rd_r(r^{2}\bfU^\theta) - \rd_\theta \bfU^r){\rd_{z}},
\end{aligned}\end{equation}
\begin{equation} \label{eq:div-cylin}
\nabla \cdot (\bfU^r{\rd_r}+\bfU^\theta {\rd_\theta} + \bfU^z{\rd_{z}}) = r^{-1}\rd_r(r\bfU^r) + \rd_\theta \bfU^\theta + \rd_{z} \bfU^{z}.
\end{equation} Finally, the material derivative is given by \begin{equation}\label{eq:mat-cylin}
\begin{split}
(\bfU\cdot\nabla)\bfV & = ( \bfU^r\rd_r\bfV^r + \bfU^\tht \rd_\tht \bfV^r + \bfU^z \rd_z \bfB^r - r\bfU^\tht \bfV^\tht ) \rd_r \\
& \qquad + ( r^{-1} \bfU^r \rd_r(r\bfV^\tht) + \bfU^\tht \rd_\tht \bfV^\tht + \bfU^z\rd_z\bfV^\tht + r^{-1}\bfU^\tht \bfV^r ) \rd_\tht \\
& \qquad + ( \bfU^r \rd_r\bfV^z + \bfU^\tht \rd_\tht \bfV^z + \bfU^z \rd_z\bfV^z ) \rd_z . 
\end{split}
\end{equation}


\org{The rest of this paper is organized as follows. In {\bf Section \ref{sec:prelim}}, we recall the Bogovski\v{i} operator and derive useful a priori estimates for axisymmetric solutions to \eqref{eq:e-mhd} and \eqref{eq:hall-mhd} respectively in {\bf Sections~\ref{subsec:bogov}} and {\bf \ref{subsec:apriori}}. Then the construction of degenerating wavepackets from \cite{JO1} is reviewed in {\bf Sections~\ref{subsec:reduction}--\ref{subsec:wkb}}. 
{\bf Section \ref{sec:linear}} is dedicated to the proofs of linear illposedness results. In {\bf Sections \ref{sec:nonlinear-unbounded}} and {\bf \ref{sec:nonlinear-nonexist}} we complete the proofs of unboundedness and nonexistence of the solution operator, respectively. }

\ackn{ I.-J. Jeong has been supported by the  National Research Foundation of Korea (NRF) grant No. 2022R1C1C1011051. S.-J. Oh was partially supported the Sloan Research Fellowship and the National Science Foundation CAREER Grant under NSF-DMS-1945615. 
}

\section{Preliminaries}\label{sec:prelim}

\subsection{Bogovski\v{i} operator}\label{subsec:bogov}

In later sections, it will be important to be able to systematically ``invert'' the divergence operator to ensure divergence-free condition on the initial perturbations. One may simply apply the Leray projector to a vector field, but then the compact support property will be immediately lost. To avoid this problem we  use the Bogovski\v{i} operator introduced in \cite{Bog}. Detailed information regarding the properties of this operator can be found in several textbooks; see for instance \cite{AcDu,Gal}. Given a bounded open set $U \subset \bbR^{3}$, we write $d(U) = \sup_{\mathbf{x},\mathbf{y} \in U}|\mathbf{x}-\mathbf{y}|$ and $i(U) = \sup_{\exists \mathbf{x} \in  {U}, B(\mathbf{x},R) \subset  {U} }R$. 

\begin{proposition}[{{\cite[Lemma 3.1]{Gal}}}]\label{prop:Bogov}
	Fix some smooth function $w \in C^\infty$ supported in the unit ball in $\bbR^3$ and satisfying $\int w = 1$, and let $U \subset \mathbb{R}^{3}$ be a bounded open set with Lipschitz boundary. Take $\mathbf{x}_0 \in U$ such that $B(\mathbf{x}_0,i(U)) \subset \overline{U}$  {and $\overline{U}$ is star-shaped with respect to $B(\mathbf{x}_0,i(U))$ (i.e., for every $\bfx \in \overline{U}$ and $\bfx' \in B(\mathbf{x}_0,i(U))$, the line segment from $\bfx$ to $\bfx'$ is contained in $\overline{U}$).} Rescale $w$ to define $w_{U}(\mathbf{x}) = i(U)^{-3} w(i(U)^{-1}(\mathbf{x}-\mathbf{x}_0))$. We define the kernel   \begin{equation}\label{eq:Bogov-kernel} 
	\begin{split}
	G_{U}(\mathbf{x},\mathbf{y})  = \int_0^1 \frac{\mathbf{x} - \mathbf{y}}{s} w\left(\mathbf{y}+\frac{\mathbf{x} - \mathbf{y}}{s}  \right) \frac{ds}{s^{3}}.
	\end{split}
	\end{equation} 
	Assume that $g \in L^2(\mathbb{R}^{3})$ is supported in $U$ and satisfies $\int_U g = 0$. Then, we have that $h$ defined by \begin{equation*} 
	\begin{split}
	&  h(\mathbf{x}) =  \int_{\mathbb{R}^{3}} G_U(\mathbf{x},\mathbf{y}) g(\mathbf{y}) \ud \mathbf{y}
	\end{split}
	\end{equation*} is supported in $U$ and solves $\mathrm{div} \, h = g$, with the estimates \begin{equation}\label{eq:Bogovskii-L2} 
	\begin{split}
	\nrm{h}_{L^2} \le \frac{C_n}{d( {U})} \left( \frac{d( {U})}{i( {U})} \right)^n \left(1 + \frac{d( {U})}{i( {U})} \right) \nrm{g}_{L^2}, \qquad \nrm{\nabla h}_{L^2} \le C \left( \frac{d( {U})}{i( {U})} \right)^{3} \left(1 + \frac{d( {U})}{i( {U})} \right) \nrm{g}_{L^2},
	\end{split}
	\end{equation} for a universal constant $C>0$.
\end{proposition}

For our applications, we shall take $\bfx_{0}$ to lie on the $z$-axis and $U$ to be an (axisymmetric) cylinder of the form $U = \set{(x, y, z) : \abs{(x, y)} < \ell, \, \abs{z} < \ell}$ for some $\ell > 0$. Abusing the terminology, we will refer to such cylinder $U$ as a \emph{cube} with side-length $\ell$. Choosing $i(U) = c \ell$ for some universal small constant $c > 0$, we simply have  \begin{equation}\label{eq:Bogovskii-L2-3}
		\begin{split}
			\nrm{\nabla^{m+1} h}_{L^2} \le C_{m} \nrm{\nabla^m g}_{L^2}
		\end{split}
	\end{equation} for all $m \ge 0$.  {Indeed, observe that $d(U)$ is comparable to $\ell$ (and hence to $i(U)$, and that the star-shape condition follows from the convexity of $U$.} 
	
	Moreover, we shall replace the kernel $G_{U}$ in \eqref{eq:Bogov-kernel} by \begin{equation*}
	\begin{split}
		\mathring{G}_{U}(r,\tht,z,r',\tht',z') := \int_{0}^{2\pi}  G_{U}(r,\tht + \tht_{0},z, r', \tht' + \tht_{0}, z')  \, \ud \tht_{0} 
	\end{split}
\end{equation*} and define
\begin{equation*}
	\mathrm{div}^{-1}_{U} g(\bfx) := \int_{\bbR^{3}} \mathring{G}_{U}(\bfx, \bfy) g(\bfy) \, \ud \bfy.
\end{equation*}
With this choice, we keep the right-inverse and support-preserving properties of the operator in Proposition~\ref{prop:Bogov}, and have the additional convenient property that $[\rd_{\tht}, \mathrm{div}^{-1}_{U}	] = 0$.
In what follows, if the choice of $U$ is clear from the context, we will simply write $\mathrm{div}^{-1}$.

\subsection{A priori estimates for axisymmetric solutions}\label{subsec:apriori}

The purpose of subsection is to derive a few elementary a priori estimates for axisymmetric solutions to \eqref{eq:e-mhd} and \eqref{eq:hall-mhd}, which will serve as background solutions in our proof of instability/illposedness. 

We begin by recalling that if $\bfU$ be of the form $\bfU = U(r, z) \rd_{\tht}$, then $(-\lap) \bfU =   - \left(\rd_{r}^{2} + \frac{3}{r} \rd_{r} + \rd_{z}^{2}\right) U  \rd_{\tht}$. We shall use the notation 
\begin{equation*}
	-\slashed{\lap} U
	= - \left(\rd_{r}^{2} + \frac{3}{r} \rd_{r} + \rd_{z}^{2}\right) U.
\end{equation*}

In what follows, we write $\bfPi$ instead $\bfB$ for an axisymmetric solution to \eqref{eq:e-mhd}. The system \eqref{eq:e-mhd}  for solutions of the form $\bfPi = \Pi(r, z) \rd_{\tht}$ may be rewritten as
\begin{equation} \label{eq:e-mhd-fradiss-axi}
	\rd_{t} \Pi - 2 \Pi \rd_{z} \Pi = 0 .
\end{equation}
For an axisymmetric solution to \eqref{eq:hall-mhd}, we write $(\bfV, \bfPi)$ instead $(\bfu, \bfB)$. The system \eqref{eq:hall-mhd}  for solutions of the form $(\bfV, \bfPi) = (V^{r}(r, z) \rd_{r} + V^{z}(r, z) \rd_{z}, \Pi(r, z) \rd_{\tht})$ may be rewritten as
\begin{equation} \label{eq:hall-mhd-fradiss-axi}
\left\{
\begin{aligned}
	\rd_{t} \Omg + (V^{r} \rd_{r} + V^{z} \rd_{z}) \Omg + 2 \Pi \rd_{z} \Pi &= \nu \slashed{\lap} \Omg, \\
	\rd_{t} \Pi + (V^{r} \rd_{r} + V^{z} \rd_{z}) \Pi - 2 \Pi \rd_{z} \Pi &= 0 , \\
	\rd_{r} V^{r} + \rd_{z} V^{z} &= 0,
\end{aligned}
\right.
\end{equation}
where $\nb \times \bfV = \Omg(r, z) \rd_{\tht}$. Conversely, under a mild decay assumption at spatial infinity (which will always be satisfied in our application), $\bfV$ may be recovered from $\Omg$ by introducing the stream function $\Phi(r, z)$ that solves $- \slashed{\lap} \Phi = \Omg,$ 
with respect to which we have 
\begin{equation*}
V^{r} = - r \rd_{z} \Phi, \quad V^{z} = \frac{1}{r} \rd_{r} (r^{2} \Phi).
\end{equation*}

\subsubsection{$H^{m}$-estimates and local wellposedness}
Unlike the general non-axisymmetric case, \eqref{eq:e-mhd-fradiss-axi} and \eqref{eq:hall-mhd-fradiss-axi} are locally wellposed in (sufficiently high regularity) Sobolev spaces. Here we derive $H^{m}$ a priori estimates and corresponding local wellposedness statements for \eqref{eq:e-mhd-fradiss-axi} and \eqref{eq:hall-mhd-fradiss-axi}.

We start with a lemma concerning the relationship between the standard $H^{m}$-norm for a vector field formulated with respect to the standard basis $(e_{x}, e_{y}, e_{z})$ and those for components with respect to $(\rd_{r}, \rd_{\tht}, \rd_{z})$ in the axisymmetric case.
\begin{lemma} \label{lem:axisymm-sob}
Let $\bfU, \bfV \in \calS(\bbR^{3})$ be of the form $\bfU = U^{\tht} (r, z) \rd_{\tht}$ and $\bfV = V^{r}(r, z) \rd_{r} + V^{z}(r, z) \rd_{z}$. 
\begin{enumerate}
\item For $U^{\tht}$ and $r^{-1} V^{r}$, we have the formulae
\begin{align} 
	U^{\tht}(r, z) &= \frac{1}{r^{2}} \int_{0}^{r} (- \rd_{y} \bfU^{x} + \rd_{x} \bfU^{y})(r', z) r' \, \ud r', \label{eq:axisymm-Utht-rep} \\
	\frac{1}{r} V^{r}(r, z) &= \frac{1}{r^{2}} \int_{0}^{r} (\rd_{x} \bfV^{x} + \rd_{y} \bfV^{y})(r', z) r' \, \ud r'. \label{eq:axisymm-Vr-rep}
\end{align}

\item For any nonnegative integer $m$, we have
\begin{align} 
\nrm{\bfU}_{H^{m}}^{2} 
&\aeq_{m} \nrm{r U^{\tht}}_{H^{m}(r \ud r \ud z)}^{2} + \nrm{U^{\tht}}_{H^{m-1}(r \ud r \ud z)}^{2}, \label{eq:axisymm-sob-Utht} \\
\nrm{\bfV}_{H^{m}}^{2} 
&\aeq_{m} \nrm{V^{r}}_{H^{m}(r \ud r \ud z)} + \nrm{r^{-1} V^{r}}_{H^{m-1}(r \ud r \ud z)}^{2} + \nrm{V^{z}}_{H^{m}(r \ud r \ud z)}^{2}, \label{eq:axisymm-sob-Vrz}
\end{align}
where in the case $m = 0$, the terms involving $H^{m-1}$ are ignored. Moreover, we have
\begin{align} 
	\nrm{\nb (U^{\tht} \rd_{\tht})}_{L^{2}}^{2}
	&\aeq \nrm{r \rd_{r} U^{\tht}}_{L^{2}(r \ud r \ud z)}^{2} + \nrm{r \rd_{z} U^{\tht}}_{L^{2}(r \ud r \ud z)}^{2} + \nrm{U^{\tht}}_{L^{2}(r \ud r \ud z)}^{2}, \label{eq:axisymm-sob-hardy-Utht} \\
	\nrm{\nb (V^{r} \rd_{r})}_{L^{2}}^{2}
	&\aeq \nrm{\rd_{r} V^{r}}_{L^{2}(r \ud r \ud z)}^{2} + \nrm{\rd_{z} V^{r}}_{L^{2}(r \ud r \ud z)}^{2} + \nrm{r^{-1} V^{r}}_{L^{2}(r \ud r \ud z)}^{2}. \label{eq:axisymm-sob-hardy-Vrz}
\end{align}

\end{enumerate}
\end{lemma}
\begin{remark}
For $V^{z}$, by Hardy's inequality on $\bbR^{3}$, we simply have
\begin{equation*}
	\nrm{(r^{2} + z^{2})^{-\frac{1}{2}} V^{z}}_{L^{2}(r \ud r \ud z)} \aleq \nrm{\nb V^{z}}_{L^{2}(r \ud r \ud z)} \aeq \nrm{\nb (V^{z} \rd_{z})}_{L^{2}}.
\end{equation*}

\end{remark}
\begin{proof}
The inequalities $\aleq_{m}$ in \eqref{eq:axisymm-sob-Utht}--\eqref{eq:axisymm-sob-Vrz} are straightforward consequences of the relation $(\rd_{r}, \rd_{\tht}, \rd_{z}) = (r^{-1} x e_{x} + r^{-1} y e_{y}, x e_{y} - y e_{x}, e_{z})$. To prove the reverse inequality, as well as \eqref{eq:axisymm-Utht-rep}--\eqref{eq:axisymm-Vr-rep} and \eqref{eq:axisymm-sob-hardy-Utht}--\eqref{eq:axisymm-sob-hardy-Vrz}, we introduce the two-dimensional scaling vector field $S = r \rd_{r} = x \rd_{x} + y \rd_{y}$ and note the algebraic identities
\begin{align*}
	(S + 2) U^{\tht} &= - \rd_{y} (- y U^{\tht}) + \rd_{x} (x U^{\tht}) = - \rd_{y} \bfU^{x} + \rd_{x} \bfU^{y}, \\
	(S + 2) (r^{-1} V^{r}) &= \rd_{x} (x r^{-1} V^{r}) + \rd_{y} (y r^{-1} V^{r}) = \rd_{x} \bfV^{x} + \rd_{y} \bfV^{y}.
\end{align*}
Due to the similarity of these identities, the proofs for $\bfU$ and $\bfV$ are similar after this point; henceforth, we only focus on the case of $\bfU$. Noting that $(S + 2) U^{\tht} = r^{-1} \rd_{r}(r^{2} U^{\tht})$, \eqref{eq:axisymm-Utht-rep} follows from the fundamental theorem of calculus and the vanishing of $r^{2} U^{\tht}$ at $r = 0$. Furthermore, for any multiindex $\bt$ of order $\ell-1$, we have the commutator identity
\begin{equation*}
	(S + \ell+1) \nb_{x, y}^{\bt} U^{\tht} = \nb_{x, y}^{\bt} (S + 2) U^{\tht} = \nb_{x, y}^{\bt} (- \rd_{y} \bfU^{x} + \rd_{x} \bfU^{y}).
\end{equation*}
Hence, 
\begin{equation*}
\nrm{(S + \ell+1) \nb_{x, y}^{\bt} U^{\tht}}_{L^{2}} \aleq \nrm{\nb_{x, y}^{\ell} \bfU}_{L^{2}}. 
\end{equation*}
We claim that, for any $u = u(r)$ that vanishes sufficiently fast as $r \to \infty$, the following Hardy's inequality holds:
\begin{equation*}
	\nrm{(r \rd_{r} + \ell+1) u}_{L^{2}(r \ud r)}^{2} \geq \ell^{2} \nrm{u}_{L^{2}(r \ud r)}^{2}.
\end{equation*}
Assuming the claim, \eqref{eq:axisymm-sob-hardy-Utht} and the inequality $\ageq_{m}$ in \eqref{eq:axisymm-sob-Utht} readily follow. To prove the claim, we follow the classical proof of Hardy's inequality and argue as follows:
\begin{align*}
\int ((r \rd_{r} + \ell+1) u)^{2} \, r \ud r
&= \int \left( ((r \rd_{r} + 1) u)^{2} + 2 \ell u (r \rd_{r} + 1) u + \ell^{2} u^{2} \right)r \, \ud r \\
&= \int ((r \rd_{r} + 1) u)^{2} r +  \ell \rd_{r} u^{2} r^{2} + \ell(\ell+2) u^{2} r \, \ud r \\
&\geq \int ((r \rd_{r} + 1) u)^{2} r + \ell^{2} u^{2} r \, \ud r,
\end{align*}
where we used an integration by parts and the vanishing of $u(r)$ as $r \to \infty$ on the last line. Since the first term on the last line is nonnegative, the claim follows. \qedhere
\end{proof}

We now prove an $H^{m}$ a priori estimate for the axisymmetric \eqref{eq:e-mhd}  {system} \eqref{eq:e-mhd-fradiss-axi}. Let $\bfPi = \Pi(r, z) \rd_{\tht} \in C_{t}(I; H^{\infty})$ and fix $m \geq 4$. For any multiindex $\alp$ of order at most $m$, we compute 
\begin{align*}
	&\frac{1}{2} \frac{\ud}{\ud t} \int \abs{\nb^{\alp} \bfPi}^{2} \, \ud x \ud y \ud z 
	= \frac{1}{2} \frac{\ud}{\ud t} \int \abs{\nb^{\alp}(\Pi \rd_{\tht})}^{2} \, \ud x \ud y \ud z = 2 \int \nb^{\alp} (\Pi \rd_{z} \Pi \, \rd_{\tht}) \cdot \nb^{\alp} (\Pi \rd_{\tht}) \, \ud x \ud y \ud z  \\
	&\quad = 2 \int \Pi \rd_{z} \nb^{\alp} (\Pi \, \rd_{\tht}) \cdot \nb^{\alp} (\Pi \rd_{\tht}) \, \ud x \ud y \ud z  + 2 \int \{ \nb^{\alp}(\Pi \rd_{z} (\Pi \, \rd_{\tht})) - \Pi \rd_{z} \nb^{\alp} (\Pi \, \rd_{\tht})\} \cdot \nb^{\alp} (\Pi \rd_{\tht}) \, \ud x \ud y \ud z. 
\end{align*}
For the first term, we perform an integration by parts in $z$ and estimate as follows:
\begin{align*}
\abs*{\int \Pi \rd_{z} \nb^{\alp} (\Pi \, \rd_{\tht}) \cdot \nb^{\alp} (\Pi \rd_{\tht}) \, \ud x \ud y \ud z}
&=  \abs*{\int (\rd_{z} \Pi) \nb^{\alp} (\Pi \, \rd_{\tht}) \cdot \nb^{\alp} (\Pi \rd_{\tht}) \, \ud x \ud y \ud z} \aleq \nrm{\rd_{z} \Pi}_{L^{\infty}} \nrm{\bfPi}_{H^{m}}^{2} 
\aleq \nrm{\bfPi}_{H^{m}}^{3},
\end{align*}
where we used the Sobolev inequality, $\nrm{\rd_{z} \Pi}_{H^{2}} \aleq \nrm{\bfPi}_{H^{4}}$ and $m \geq 4$ for the last inequality.
For the second term, we use the fact that no factor of $\Pi$ is differentiated $m+1$-times. Using $\rd_{\tht} = x \bfe_{y} - y \bfe_{x}$ and Lemma~\ref{lem:axisymm-sob}, we have
\begin{align*}
\abs*{\int \{ \nb^{\alp}(\Pi \rd_{z} (\Pi \, \rd_{\tht})) - \Pi \rd_{z} \nb^{\alp} (\Pi \, \rd_{\tht})\} \cdot \nb^{\alp} (\Pi \rd_{\tht}) \, \ud x \ud y \ud z} 
& \lesssim \sum_{\bt : \abs{\bt} \geq 1, \, \bt \leq \alp} \nrm{\nb^{\bt} \Pi \nb^{\alp - \bt} \rd_{z} \bfPi}_{L^{2}} \nrm{\bfPi}_{H^{m}} \lesssim \nrm{\bfPi}_{H^{m}}^{3},
\end{align*}
where we used $m \geq 4$ on the last line.  Integrating the resulting differential inequality in time, we obtain the a priori bound
\begin{equation} \label{eq:e-mhd-fradiss-axi-Hm}
	\sup_{t \in [0, T_{0}]} \nrm{\bfPi(t)}_{H^{m}} \aleq \nrm{\bfPi(0)}_{H^{m}}, \qquad T_{0} = T_{0}(\nrm{\bfPi(0)}_{H^{m}})>0.
\end{equation} Moreover, we have the following local wellposedness result:
\begin{proposition} \label{prop:e-mhd-fradiss-axi-lwp} Given initial data $\bfPi_{0} = \Pi_{0}(r, z) \rd_{\tht} \in H^{m}$ with $m \geq 4$, there exists a unique solution $\bfPi \in C_{t}([0, T_{0}]; H^{m})$  to the system \eqref{eq:e-mhd-fradiss-axi} with $\bfPi(0) = \bfPi_{0}$, where $T_0$ depends only on $\nrm{\bfPi_{0}}_{H^{m}}$. Moreover, the solution satisfies the a priori estimate \eqref{eq:e-mhd-fradiss-axi-Hm} on $[0, T_{0}]$.
\end{proposition}
Uniqueness can be proved along the similar lines as in the proof of the $H^{m}$ a priori estimate.\footnote{We note that the uniqueness statement  says that the solution to \eqref{eq:e-mhd} under the \textit{axisymmetric} assumption is unique. Of course, without a uniqueness statement for the original \eqref{eq:e-mhd} system, one cannot exclude the possibility that an axisymmetric initial data evolving into a non-axisymmetric solution as well.} For existence, one may use a vanishing viscosity method (as sketched below for the more complicated case of \eqref{eq:hall-mhd-fradiss-axi}). We omit the straightforward details.

\begin{remark} \label{rem:e-mhd-fradiss-axisymm-reg}
Using fractional Sobolev spaces, the regularity requirement may be optimized to $m > 7/2$. 
\end{remark}

Next, we turn to the axisymmetric \eqref{eq:hall-mhd} system, i.e., \eqref{eq:hall-mhd-fradiss-axi}. Let $(\bfomg, \bfPi) = (\Omg(r, z) \rd_{\tht}, \Pi(r, z) \rd_{\tht}) \in C_{t}(I; H^{\infty})$ and fix $m \geq 4$. Let $\bfV = V^{r}(r, z) \rd_{r} + V^{\tht}(r, z)$ be given from $\bfomg$ by the Biot--Savart law; hence
\begin{equation} \label{eq:hall-mhd-fradiss-axi-biotsavart}
	\nrm{\nb \bfV}_{H^{m-1}} \aleq \nrm{\bfomg}_{H^{m-1}}.
\end{equation}
By Lemma~\ref{lem:axisymm-sob}, it follows that $\nrm{\nb V^{r}}_{H^{m-1}} + \nrm{r^{-1} V^{r}}_{L^{2}} + \nrm{\nb V^{z}}_{H^{m-1}} \aleq \nrm{\nb \bfV}_{H^{m-1}}.$ For any multiindex $\alp$ of order at most $m-1$, we use the $\Omg$-equation in \eqref{eq:hall-mhd-fradiss-axi} to compute 
\begin{align*}
	& \frac{1}{2} \frac{\ud}{\ud t} \int \abs{\nb^{\alp} \bfomg}^{2} \, \ud x \ud y \ud z  = \int \nb^{\alp}(\rd_{t} \Omg \rd_{\tht}) \cdot \nb^{\alp} (\Omg \rd_{\tht}) \, \ud x \ud y \ud z  \\
	&= - \int (V^{r} \rd_{r} + V^{z} \rd_{z}) \nb^{\alp} (\Omg \rd_{\tht}) \cdot \nb^{\alp} (\Omg \rd_{\tht}) \, \ud x \ud y \ud z \\
	&\peq - \int \{ \nb^{\alp} ((V^{r} \rd_{r} + V^{z} \rd_{z}) \Omg \rd_{\tht}) - (V^{r} \rd_{r} + V^{z} \rd_{z}) \nb^{\alp} (\Omg \rd_{\tht}) \}  \cdot \nb^{\alp} (\Omg \rd_{\tht}) \, \ud x \ud y \ud z \\
	&\peq - 2 \int \nb^{\alp} (\Pi \rd_{z} \Pi \, \rd_{\tht}) \cdot \nb^{\alp} (\Omg \rd_{\tht}) \, \ud x \ud y \ud z  - \nu \int \abs{\nb \nb^{\alp} (\Omg \, \rd_{\tht})}^{2} \, \ud x \ud y \ud z.
\end{align*}
The first term vanishes after an integration by parts (since $\bfV$ is divergence-free) and the last term is nonpositive. Estimating the remaining terms as in the previous proof of \eqref{eq:e-mhd-fradiss-axi}, we obtain 
\begin{align*}
\frac{1}{2} \frac{\ud}{\ud t} \int \abs{\nb^{\alp} \bfomg}^{2} \, \ud x \ud y \ud z 
& \aleq \nrm{\bfomg}_{H^{m-1}}^{3} + \nrm{\bfPi}_{H^{m}}^{2} \nrm{\bfomg}_{H^{m-1}},
\end{align*}
provided that $m \geq 4$. Next, for any multiindex $\alp'$ of order at most $m$, we use the $\Pi$-equation in \eqref{eq:hall-mhd-fradiss-axi} to compute 
\begin{align*}
	& \frac{1}{2} \frac{\ud}{\ud t} \int \abs{\nb^{\alp'} \bfPi}^{2} \, \ud x \ud y \ud z = \int \nb^{\alp'}(\rd_{t} \Pi \rd_{\tht}) \cdot \nb^{\alp'}(\Pi \rd_{\tht}) \, \ud x \ud y \ud z \\
	&= - \int \nb^{\alp'} ((V^{r} \rd_{r} + V^{z} \rd_{z}) \Pi \rd_{\tht}) \cdot \nb^{\alp'} (\Pi \rd_{\tht}) \, \ud x \ud y \ud z  + 2 \int \nb^{\alp'} (\Pi \rd_{z} \Pi \, \rd_{\tht}) \cdot \nb^{\alp'} (\Pi \rd_{\tht}) \, \ud x \ud y \ud z  \\ 
	&= - \int (V^{r} \rd_{r} + V^{z} \rd_{z}) \nb^{\alp'} (\Pi \rd_{\tht}) \cdot \nb^{\alp'} (\Pi \rd_{\tht}) \, \ud x \ud y \ud z \\
	&\peq - \int \{ \nb^{\alp'} ((V^{r} \rd_{r} + V^{z} \rd_{z}) \Pi \rd_{\tht}) - (V^{r} \rd_{r} + V^{z} \rd_{z}) \nb^{\alp'} (\Pi \rd_{\tht}) \}  \cdot \nb^{\alp'} (\Pi \rd_{\tht}) \, \ud x \ud y \ud z \\
	& \peq + 2 \int \Pi \rd_{z} \nb^{\alp'} (\Pi \, \rd_{\tht}) \cdot \nb^{\alp'} (\Pi \rd_{\tht}) \, \ud x \ud y \ud z  + 2 \int \{ \nb^{\alp'}(\Pi \rd_{z} (\Pi \, \rd_{\tht})) - \Pi \rd_{z} \nb^{\alp'} (\Pi \, \rd_{\tht})\} \cdot \nb^{\alp'} (\Pi \rd_{\tht}) \, \ud x \ud y \ud z .
\end{align*}
The first term vanishes after an integration by parts (since $\bfV$ is divergence-free). Estimating the remaining terms as in the proof of \eqref{eq:e-mhd-fradiss-axi}, we obtain 
\begin{align*}
\frac{1}{2} \frac{\ud}{\ud t} \int \abs{\nb^{\alp'} \bfPi}^{2} \, \ud x \ud y \ud z
 \aleq (\nrm{\bfomg}_{H^{m-1}} + \nrm{\bfPi}_{H^{m}}) \nrm{\bfPi}_{H^{m}}^{2}.
\end{align*}
We now sum up these differential inequalities and integrate in time. As a result, for $m \geq 4$, we obtain the a priori bound
\begin{equation} \label{eq:hall-mhd-fradiss-axi-Hm}
\begin{aligned}
	 \sup_{t \in [0, T_{0}]} \left( \nrm{(\bfomg(t), \bfPi(t))}_{H^{m-1} \times H^{m}} \right) 
	+ \left( \nu \int_{0}^{T_{0}} \nrm{\bfomg(t)}_{H^{m}}^{2}  \, \ud t\right)^{\frac{1}{2}}
	 \aleq_{m} \nrm{(\bfomg(0), \bfPi(0))}_{H^{m-1} \times H^{m}}
\end{aligned}
\end{equation}
on a time interval $[0, T_{0}]$, for $T_{0}$ sufficiently small depending on $\nrm{(\bfomg(0), \bfPi(0))}_{H^{m-1} \times H^{m}}$. Moreover, we have the following local wellposedness result:
\begin{proposition} \label{prop:hall-mhd-fradiss-axi-lwp}
Given initial data $(\bfomg_{0} = \Omg(r, z) \rd_{\tht}, \bfPi_{0} = \Pi_{0}(r, z) \rd_{\tht}) \in H^{m-1} \times H^{m}$ with $m \geq 4$, there exists a unique solution $(\bfomg, \bfPi) \in C_{t}([0, T_{0}]; H^{m-1} \times H^{m})$ to \eqref{eq:hall-mhd-fradiss-axi} with $(\bfV(0), \bfPi(0)) = (\bfV_{0}, \bfPi_{0})$ such that the LHS of \eqref{eq:hall-mhd-fradiss-axi-Hm} is finite, where $T_0$ depends only on $\nrm{(\bfV_{0}, \bfPi_{0})}_{H^{m-1} \times H^{m}}$. Moreover, the solution satisfies the a priori estimate \eqref{eq:hall-mhd-fradiss-axi-Hm} on $[0, T_{0}]$ with an implicit constant that only depends on $m$.
\end{proposition}
As before, uniqueness can be proved along the similar lines as in the proof of the $H^{m}$ a priori estimate. For the proof of existence, we may add an artificial viscosity term $\kpp \lap \Pi$ ($\kpp > 0$) on the RHS of the $\Pi$-equation in \eqref{eq:hall-mhd-fradiss-axi}, for which local wellposedness is proved in \cite{CWW}. For this viscous system, an argument similar to the proof of the $H^{m}$ a priori estimate leads to a $\kpp$-independent $H^{m}$ a priori estimate of the form \eqref{eq:hall-mhd-fradiss-axi-Hm} on a short time interval $[0, T_{0}]$, with $T_{0}$ only depending on $\nrm{(\bfV_{0}, \bfPi_{0})}_{H^{m-1} \times H^{m}}$. Then by taking the limit $\kpp \to 0$, we obtain a solution to \eqref{eq:hall-mhd-fradiss-axi} in the desired function space. We omit the straightforward details.

Remark \ref{rem:e-mhd-fradiss-axisymm-reg} apply to the case of \eqref{eq:hall-mhd} as well.

\subsubsection{Further properties}
We now discuss further properties of the axisymmetric solutions that will be used later. We begin with the case of \eqref{eq:e-mhd}. Observe that \eqref{eq:e-mhd-fradiss-axi} with $\eta = 0$ is simply the inviscid Burgers' equation in $(t, z)$. As a result, it obeys the following finite speed of propagation property:
\begin{lemma} \label{lem:e-mhd-fradiss-axi-fsp}
For any $C^{1}$ solution $\Pi$ to \eqref{eq:e-mhd-fradiss-axi}, the value of $\Pi(t, r, z)$ is uniquely determined by $\Pi_{0}$ on $\set{(r, z') \in \set{r} \times \bbR : \abs{z - z'} \leq t \nrm{\Pi_{0}}_{L^{\infty}}}$.
\end{lemma}

In general, we have the following estimate near the initial time:
\begin{lemma} \label{lem:e-mhd-fradiss-axi-small-t}
	Let $m \geq 10$. Consider an initial data set $\bfPi_{0} = \Pi_{0}(r, z) \rd_{\tht} \in H^{m}$ and the corresponding solution $\bfPi$ to \eqref{eq:e-mhd-fradiss-axi} on $[0, T_{0}]$ given by Proposition~\ref{prop:e-mhd-fradiss-axi-lwp}. Then for $t \in [0, T_{0}]$, we have
	\begin{equation}\label{eq:e-mhd-small-t-Pi}
	\begin{split}
	\nrm*{\Pi - \left(\Pi_{0} + \Pi_{0} \rd_{z} \Pi_{0} \right)}_{H^{m-3}} \lesssim_{m} t^2  \nrm{\bfPi_{0}}_{H^{m}}^{3} . 
	\end{split}
\end{equation}
\end{lemma}

We remark that the factor $t^{2}$ allows us to treat the expression on the LHS as an acceptable error in the proof of Sobolev instability/illposedness of \eqref{eq:e-mhd} outside axisymmetry below.
\begin{proof}
In this proof, we suppress the dependence of implicit constants on $m$. By Taylor expansion in $t$, we have
\begin{equation} \label{eq:e-mhd-small-t-B}
	\nrm{\bfPi(t) - \bfPi(0) - t \rd_{t} \bfPi(0)}_{H^{m-2}}
	\aleq t^{2} \sup_{t' \in [0, t]} \nrm{\rd_{t}^{2} \bfPi(t')}_{H^{m-2}}.
\end{equation}
By \eqref{eq:e-mhd-fradiss-axi} and Lemma~\ref{lem:axisymm-sob}, the LHS of \eqref{eq:e-mhd-small-t-B} dominates the LHS of \eqref{eq:e-mhd-small-t-Pi}. Hence, it only remains to bound the RHS of \eqref{eq:e-mhd-small-t-B}. By \eqref{eq:e-mhd-fradiss-axi}, for each fixed $t \in [0, T_{0}]$, we have
\begin{align*}
	\nrm{\rd_{t} \bfPi}_{H^{m-1}}
	& \aleq  \nrm{\bfPi}_{H^{m}}^{2}, \qquad \nrm{\rd_{t}^{2} \bfPi}_{H^{m-2}}
	 \aleq \nrm{\bfPi}_{H^{m-1}} \nrm{\rd_{t} \bfPi}_{H^{m-1}}  \aleq  \nrm{\bfPi}_{H^{m}}^{3}.
\end{align*}
Then by the a priori estimate \eqref{eq:e-mhd-fradiss-axi-Hm}, the desired bound follows. \qedhere
\end{proof}

In the case of \eqref{eq:hall-mhd}, we consider initial data of the form $(\bfomg_{0}, \bfPi_{0}) = (0, \Pi(r, z) \rd_{\tht})$. For the corresponding solution, the following comparison with the solution to \eqref{eq:e-mhd-fradiss-axi} with the same initial data $\bfPi_{0} = \Pi_{0} \rd_{\tht}$ holds near the initial time:
\begin{lemma} \label{lem:hall-mhd-fradiss-axi-small-t}
	Let $\nu \geq 0$ and $m \geq 10$. Consider an initial data set $(\bfomg_{0}, \bfPi_{0}) = (0, \Pi(r, z) \rd_{\tht})$, where $\bfPi_{0} \in H^{m}$, and let $(\bfomg, \bfPi)$ be the corresponding solution to \eqref{eq:hall-mhd-fradiss-axi} on $[0, T_{0}]$ given by Proposition~\ref{prop:hall-mhd-fradiss-axi-lwp}. Moreover, let $\bfPi^{e} = \Pi^{e}(r, z) \rd_{\tht}$ be the solution to \eqref{eq:e-mhd-fradiss-axi} with $\bfPi^{e}(t=0) = \Pi_{0}(r, z) \rd_{\tht}$. Then for $t \in [0, T_{0}]$, we have
	\begin{align}
	\nrm{\nb \bfV}_{H^{m-3}}
	& \aleq_{m} t (\nrm{\bfPi_{0}}_{H^{m}}^{2} + \nu \nrm{\bfPi_{0}}_{H^{m}} \label{eq:small-t-V}), \\
	\nrm{\Pi - \Pi^{e}}_{H^{m-3}} 
	& \aleq_{m} t^{2} \exp(C \nrm{\bfPi_{0}}_{H^{m}} t) (\nrm{\bfPi_{0}}_{H^{m}}^{3} + \nu \nrm{\bfPi_{0}}_{H^{m}}^{2}). \label{eq:small-t-Pi}
\end{align}
\end{lemma}
We note that the factor $t^{2}$ in \eqref{eq:small-t-Pi} is again crucial to treat $\Pi - \Pi^{e}$ as an acceptable error in the proof of Sobolev instability/illposedness of \eqref{eq:hall-mhd} outside axisymmetry below.
\begin{proof}
In this proof, we suppress the dependence of implicit constants on $m$. 
By \eqref{eq:hall-mhd-fradiss-axi} and \eqref{eq:hall-mhd-fradiss-axi-biotsavart}, for each $t \in [0, T_{0}]$, we may estimate
\begin{align*}
	\nrm{\rd_{t} \nb \bfV}_{H^{m-3}} 
	\aleq \nrm{\rd_{t} \bfomg}_{H^{m-3}}  
	\aleq \nrm{\bfomg}_{H^{m-1}}^{2} + \nrm{\bfPi}_{H^{m}}^{2} + \nu \nrm{\bfomg}_{H^{m-1}}
\end{align*}
By hypothesis, $\bfomg_{0} = 0$ (hence $\bfV(0) = 0$). Hence, by integrating in time and using the a priori estimate \eqref{eq:hall-mhd-fradiss-axi-Hm}, obtain \eqref{eq:small-t-V}. Next, by subtracting \eqref{eq:e-mhd-fradiss-axi} from the $\Pi$-equation in \eqref{eq:hall-mhd-fradiss-axi}, we obtain
\begin{align*}
	\rd_{t} (\Pi - \Pi^{e}) = - (V^{r} \rd_{r} + V^{z} \rd_{z}) \Pi + 2 (\Pi - \Pi^{e}) \rd_{z} \Pi + 2 \Pi^{e} \rd_{z} (\Pi - \Pi^{e}) .
\end{align*}
We use this equation to compute $\frac{1}{2} \frac{\ud}{\ud t} \nrm{\bfPi - \bfPi^{e}}_{H^{m-2}}^{2}$. Proceeding as in the proof of the a priori estimate \eqref{eq:hall-mhd-fradiss-axi-Hm}, we obtain the differential inequality
\begin{align*}
	\frac{1}{2} \frac{\ud}{\ud t} \nrm{\bfPi - \bfPi^{e}}_{H^{m-2}}^{2} \aleq \nrm{\nb \bfV}_{H^{m-3}} \nrm{\bfPi}_{H^{m-1}} \nrm{\bfPi - \bfPi^{e}}_{H^{m-2}}
	+ (\nrm{\bfPi}_{H^{m-1}} + \nrm{\bfPi^{e}}_{H^{m-1}}) \nrm{\bfPi - \bfPi^{e}}_{H^{m-2}}^{2}
\end{align*}
while $(\bfPi - \bfPi^{e})(0) = 0$. Cancelling a factor of $\nrm{\bfPi - \bfPi^{e}}_{H^{m-1}}$, applying Gr\"onwall's inequality and using the a priori estimates \eqref{eq:hall-mhd-fradiss-axi-Hm} and \eqref{eq:small-t-V}, we obtain
\begin{equation*}
	\nrm{\bfPi - \bfPi^{e}}_{H^{m-2}} \aleq t^{2} \exp(C \nrm{\bfPi_{0}}_{H^{m}} t) (\nrm{\bfPi_{0}}_{H^{m}}^{3} + \nu \nrm{\bfPi_{0}}_{H^{m}}^{2}).
\end{equation*}
We note that the factor $t^{2}$ comes from integrating the factor $t$ from \eqref{eq:small-t-V} for $\nrm{\nb V}_{H^{m-3}}$. Then \eqref{eq:small-t-Pi} follows via Lemma~\ref{lem:axisymm-sob}. \qedhere
\end{proof}

\subsection{The $(2+\frac{1}{2})$-dimensional reduction}\label{subsec:reduction}  The  systems \eqref{eq:hall-mhd} and \eqref{eq:e-mhd}   significantly simplifies under the $(2+\frac{1}{2})$-dimensional reduction, by which we man that the solutions are assumed to be independent of the $z$-coordinate. 


\subsubsection{Linearized equations under the $(2+\frac{1}{2})$-dimensional reduction.}
In this section, we write down the simpler system of equations for \eqref{eq:hall-mhd-lin} and \eqref{eq:e-mhd-lin}, under the $(2+\frac{1}{2})$-dimensional assumption. 

We take $\bgB = f(r)\rd_\theta$ and introduce $\psi$ and $\omg$ satisfying \begin{equation*}
\begin{split}
(\nabla\times b)^{z} = -\lap\psi, \qquad (\nabla\times u)^{z} = \omg.
\end{split}
\end{equation*} 
Since $(\nabla\times\bgB)^{z} = r^{-1}\rd_r(r^{2} f),$
we obtain that $\psi = (-\lap)^{-1}((\nb \times \bgB)^{z} - r^{-1}\rd_r(r^{2} f)).$
Then, we have that \eqref{eq:hall-mhd-lin} and \eqref{eq:e-mhd-lin} simplify to become 
\begin{equation} \label{eq:hall-mhd-2.5d-lin-axisym}
\left\{
\begin{aligned}
&\rd_{t} u^{z} - f(r) \rd_{\tht} b^{z} - \nu \lap u^{z}= 0 , \\
&\rd_{t} \omg  - \left(f''(r) + \frac{3}{r} f'(r)\right) \rd_{\tht} \psi + f(r) \rd_{\tht} \lap \psi - \nu \lap \omg =0 ,  \\
&\rd_t  b^{z} - f(r)\rd_{\theta} u^{z} + \left( f''(r) + \frac{3}{r}f'(r) \right)\rd_{\theta}\psi- f(r)\rd_{\theta} \lap\psi = 0 ,\\
&\rd_t \psi - f(r) \rd_{\tht} (-\lap)^{-1} \omg + f(r)\rd_\theta b^{z}  = 0.
\end{aligned}
\right.
\end{equation} and 
\begin{equation} \label{eq:e-mhd-2.5d-lin-axisym}
\left\{
\begin{aligned}
&\rd_t b^{z}   - f(r)\rd_\theta \lap\psi  + \left( f''(r) + \frac{3}{r}f'(r) \right) \rd_\theta\psi = 0 ,\\
&\rd_t \psi  + f(r)\rd_\theta b^{z} = 0,
\end{aligned}
\right.
\end{equation} respectively. Note that $(b^{r},b^{\tht},u^{r},u^{\tht})$ can be recovered from $\psi$ and $\omega$ by the following relations: \begin{equation*}
\begin{split}
b^{r} = \frac{1}{r}\rd_\theta \psi, \quad b^{\theta} = -\frac{1}{r} \rd_r\psi, \quad u^{r} = \frac{1}{r} \rd_{\tht} (-\lap)^{-1} \omg, \quad
u^{\tht} = - \frac{1}{r} \rd_{r} (-\lap)^{-1} \omg. 
\end{split}
\end{equation*} 
Strictly speaking, some justification is necessary for these systems; one needs to make sure that the Biot--Savart type identities in the above are valid at the level of $(u,b) \in L^2$ and that a $z$-independent $L^2$-solution $(u,b)$ for \eqref{eq:hall-mhd-lin} ($b$ for \eqref{eq:e-mhd-lin}, resp.) with $\bgB = f(r)\rd_{\tht}$ gives rise to a unique solution $(u^z,\omega,b^z,\psi)$ for \eqref{eq:hall-mhd-2.5d-lin-axisym} ($(b^z,\psi)$ for \eqref{eq:e-mhd-2.5d-lin-axisym}, resp.). These issues are handled in \cite[Proposition 2.1]{JO1}.  

\subsubsection{Approximate solutions and  {generalized} energy identities} 

It will be convenient to consider quadruples $(\tu^z,\tomg, \tb^z,\tpsi)$ (pairs $(\tb^z,\tpsi)$, resp.) which solve \eqref{eq:hall-mhd-2.5d-lin-axisym} (\eqref{eq:e-mhd-2.5d-lin-axisym}, resp.) up to an \textit{$L^2$-error}. The tildes were inserted to emphasize that they are not exact solutions to the linear systems. We follow the notation for the errors used in \cite{JO1}: given $\bgB = f(r)\rd_{\tht}$ and $(\tu^z,\tomg,\tb^z,\tpsi)$, we define $(\errh_{\tu^z}^{(\nu)} ,   \errh_{\tomg}^{(\nu)}  , \errh_{\tb^z},  \errh_{\tpsi} )$ by 
\begin{equation} \label{eq:hall-mhd-2.5d-err-axisym}
\left\{
\begin{aligned}
	&\rd_{t} \tu^{z} - f(r) \rd_{\tht} \tb^{z} - \nu \lap \tu^{z}  = \errh_{\tu^z}^{(\nu)} , \\
	&\rd_{t} \tomg  - \left(f''(r) + \frac{3}{r} f'(r)\right) \rd_{\tht} \tpsi + f(r) \rd_{\tht} \lap \tpsi - \nu \lap \tomg  =  \errh_{\tomg}^{(\nu)} ,  \\
	&\rd_t  \tb^{z} - f(r)\rd_{\theta} \tu^{z} + \left( f''(r) + \frac{3}{r}f'(r) \right)\rd_{\theta}\tpsi- f(r)\rd_{\theta} \lap\tpsi = \errh_{\tb^z} ,\\
	&\rd_t \tpsi - f(r) \rd_{\tht} (-\lap)^{-1} \tomg + f(r)\rd_\theta \tb^{z}  = \errh_{\tpsi},
\end{aligned}
\right.
\end{equation} and in the case of \eqref{eq:e-mhd}, we introduce \begin{equation} \label{eq:e-mhd-2.5d-err-axisym}
\left\{\begin{aligned}
	&\rd_t \tb^{z}   - f(r)\rd_\theta \lap\tpsi  + \left( f''(r) + \frac{3}{r}f'(r) \right) \rd_\theta\tpsi =  \err_{\tb^z} ,\\
	&\rd_t \tpsi  + f(r)\rd_\theta \tb^{z} =  \err_{\tpsi}.
\end{aligned}
\right.
\end{equation} Moreover, we use the notation $\errh_{\tb} = (\errh_{\tb^x}, \errh_{\tb^y}, \errh_{\tb^z}) = (-\nabla^\perp \errh_{\tpsi}, \errh_{\tb}^z)$. Similarly $\err_{\tb}$ and $\errh_{\tu}$ are defined; $\err_{\tb} = (-\nabla^\perp\err_{\tpsi}, \err_{\tb^z})$ and $\errh_{\tu} = (-\nabla^\perp(-\lap)^{-1}\errh_{\tomg} , \errh_{\tu^z} )$.  With this notation, given a solution $(u,b)$ to \eqref{eq:hall-mhd-lin}, a straightforward (formal) computation shows the following identity: \begin{equation} \label{eq:en-hall-mhd-axi}
	\begin{aligned}
		& \frac{\ud}{\ud t} \left(\brk{\tb, b} + \brk{\tu, u} \right) + 2 \nu \brk{\nb \tu, \nb u} \\
		=& 	-\brk{(r f'' + 3 f') r^{-1} \rd_{\tht} \tpsi, b^{z}} 
		- \brk{\tb^{z}, (r f'' + 3 f') r^{-1} \rd_{\tht} \psi}  - \brk{(r f' + 2 f) \nb \tpsi, u^{r, \tht}}
		- \brk{\tu^{r, \tht}, (r f' + 2f) \nb \psi} \\
		& 
		+\brk{\nb^{\perp} \errh_{\tpsi}, \nb^{\perp} \psi} + \brk{\nb^{\perp} \tpsi, \nb^{\perp} \errh_{\psi}}  
		+ \brk{\errh_{\tb}, b^{z}} + \brk{\tb^{z}, \errh_{b}} \\
		& 
		- \brk{\nb^{\perp} (-\lap)^{-1} \errh_{\tomg}, u^{r, \tht}} 
		- \brk{\tu^{r, \tht}, \nb^{\perp} (-\lap)^{-1} \errh_{\omg}}
		+ \brk{\errh_{\tu}^{(\nu)}, u^{z}} + \brk{\tu^{z}, \errh_{u}^{(\nu)}}.
	\end{aligned}
\end{equation} Here, we note that $(u,b)$ does not have to be $z$-independent for the identity to hold. From \eqref{eq:en-hall-mhd-axi}, we obtain \begin{equation}\label{eq:ene-hall-mhd-parallel} 
\begin{split}
&\left| \frac{\ud}{\ud t} \left(\brk{\tb, b} + \brk{\tu, u} \right)  \right| \lesssim \nu\nrm{\nabla\tu}_{L^2}\nrm{\nabla u}_{L^2} + (\nrm{b}_{L^2} + \nrm{u}_{L^2})(\nrm{\tb}_{L^2} + \nrm{\tu}_{L^2} + \nrm{\errh_{\tb}}_{L^2} + \nrm{\errh_{\tu}}_{L^2}),
\end{split}
\end{equation} where the implicit constant depends on $f$. Similarly, when $b$ is a solution to \eqref{eq:e-mhd-lin}, we have\begin{equation} \label{eq:en-e-mhd-axi}
\begin{aligned}
	\frac{\ud}{\ud t} \brk{\tb, b}
	=& -\brk{(r f'' + 3 f') r^{-1} \rd_{\tht} \tpsi, b^{z}} - \brk{\tb^{z}, (r f'' + 3 f') r^{-1} \rd_{\tht} \psi} \\
	&	+\brk{\nb^{\perp} \err_{\tpsi}, \nb^{\perp} \psi} + \brk{\nb^{\perp} \tpsi, \nb^{\perp} \err_{\psi}}  	+ \brk{\err_{\tb}, b^{z}} + \brk{\tb^{z}, \err_{b}},
\end{aligned}
\end{equation} which gives the estimate \begin{equation}\label{eq:ene-e-mhd-parallel}  
\begin{split}
& \left| \frac{\ud}{\ud t} \brk{\tb, b} \right| \lesssim \nrm{b}_{L^2} ( \nrm{\tb}_{L^2} + \nrm{\err_{\tb}}_{L^2}). 
\end{split}
\end{equation}
This shows that a quadruple $(\tu^z,\tomg,\tb^z,\tpsi)$ (a pair $(\tb^z,\tpsi)$, resp.) is a good approximation to a solution to \eqref{eq:hall-mhd-lin} (\eqref{eq:e-mhd-lin}, resp.) if (i) $b$ and $\tb$ are close in $L^2$ at the initial time, and (ii) the $L^2$-norm of the error is under control. 
For this reason, as far as the argument involving the generalized energy identity goes, it is natural and convenient to absorb terms that are bounded either \textit{a priori} or \textit{a  {posteriori}} by the $L^2$-norm of $\tb$ into the error (similar to the way the $O(1)$-notation works).

Some regularity is necessary to justify the above identities (and the resulting estimates), which is handled in \cite[Proposition 2.3]{JO1}; here, we shall just assume the approximate solutions have some additional regularity:
\begin{proposition} \label{prop:justify}
	The energy identity \eqref{eq:en-hall-mhd-axi} holds under the following assumptions: $(u^{z}, \omg, b^{z}, \psi)$ is derived from an $L^{2}$-solution $(u, b)$, and $(\tu^{z}, \tomg , \tb^{z}, \tpsi)$ obeys\footnote{Here, by the assertion $\nb (-\lap)^{-1} \tomg \in X$, we mean $\tomg$ is of the form $-\lap w$ where $\nb w \in X$.}
		\begin{equation*}
		(\tu^{z}, \nb (-\lap)^{-1} \tomg) \in C_{t} (I; L^{2}), \quad
		(\tb^{z}, \nb \tpsi) \in C_{t}(I; L^{2}) \cap L^{1}_{t}(I; H^{1}),
		\end{equation*}
		and the error  obeys $
		(\errh_{\tu}^{(\nu)},  \errh_{\tb}) \in L^{1}_{t}(I; L^{2}),$
		and when $\nu > 0$, also $
		\nb \tu^{z}, \tomg \in L^{2}_{t}(I; L^{2}).$ 
	Analogously,  \eqref{eq:en-e-mhd-axi} holds when  $(b^{z}, \psi)$ is derived from an $L^{2}$-solution $b$  and $(\tb^{z}, \tpsi)$ obeys
		\begin{equation*}
		(\tb^{z}, \nb \tpsi) \in C_{t}(I; L^{2}) \cap L^{1}_{t}(I; H^{1})
		\end{equation*}
		and the error  obeys $
		\err_{\tb}  \in L^{1}_{t}(I; L^{2}).$
\end{proposition} 

\subsection{Degenerating wave packets for the axisymmetric background}\label{subsec:wkb}

In this section, we review the construction of degenerating wave packets for the axisymmetric background from  \cite[Sections 3--4]{JO1}, which is the key ingredient in the proof of illposedness. For simplicity, we shall impose the following assumptions on the stationary magnetic background. 

\medskip

\noindent \textbf{Assumptions on} $\bgB = f(r) \rd_{\tht}$. We assume that $f$ is a smooth function of $r$ and satisfies the following.\begin{itemize}
	\item There exist some $1 \ge \ell = \ell_{f} > 0 , r_{0} = r_{0,f} $ such that $20\ell \ge r_{0} \ge 2\ell$ and $f(r_{0}) = 0$. 
	\item The support of $f$ is contained in $[r_{0} - \ell, r_{0} + 2\ell]$.
	\item For $r_{1} = r_{0} + \ell$,  we have \begin{equation*}
		\begin{split}
			\frac12 \le f'(r) \le 1, \qquad r \in [r_{0} , r_{1}]. 
		\end{split}
	\end{equation*} 
	\item Lastly, we set  $C_{M} = \ell^{-1} \max_{1\le j\le M}  (\ell^{j} \nrm{f^{(j)}(r)}_{L^{\infty}[r_{0}, r_{1}]}). $ We also let $C_{M}$ to depend on $M$. 
\end{itemize} 
Basically, we are assuming that $f$ is characterized by a single length scale $\ell$. (This is actually not a very stringent restriction, since the scaling property of the linearized equation gives the wave packets for the rescaled $f$ by a simple rescaling of the time variable.)

\medskip 

%
%
%


\medskip

For convenience of the reader, we review the construction of degenerating wave packet, under the above simplifying assumptions for $f(r)$. The starting point is to take another time derivative on the $\psi$ equation in \eqref{eq:e-mhd-2.5d-lin-axisym}: \begin{equation}\label{eq:e-mhd-2.5d-lin-second-axisym}
	\begin{split}
		\rd_{t}^{2} \psi + f^{2} \rd_{\tht}^{2} \lap \psi = 0, 
	\end{split}
\end{equation} where we have neglected the last term on the left hand side of the $b^{z}$ equation (which is an allowable error). Once we find a solution $\tpsi$ to this equation up to an allowable error, we shall take $\tb^{z}$ by $-(f\rd_{\theta})^{-1} \rd_{t} \tpsi$, again up to  an allowable error.  Expanding the Laplacian in the cylindrical coordinates, and introducing $\tau = \lmb   t$ for some $\lmb \in \bbN$, \eqref{eq:e-mhd-2.5d-lin-second-axisym} becomes
\begin{equation*}
\rd_{\tau}^2 \tpsi  + f^2(\lmb^{-1}\rd_\theta)^2 \rd_{r}^{2}\tpsi 
+ \frac{f^{2}}{r} (\lmb^{-1}\rd_\theta)^2 \rd_{r} \tpsi 
+ \frac{\lmb^{2} f^{2}}{r^{2}}(\lmb^{-1}\rd_\theta)^{4} \tpsi = 0 .
\end{equation*}
In the domain $r \in [r_{0}, r_{1}]$, we make a change of variables $\eta = \eta(r)$, where
\begin{equation*}
\eta'(r) = \frac{1}{f(r)}, \qquad \eta(r_{1}) = 0.
\end{equation*}
Note that $\eta \to - \infty$ as $r \to r_{0}^{+}$. Moreover, $\rd_{\eta} f = f \rd_{r} f$ and $\rd_{\eta} r = f$. Introducing a new dependent variable 
\begin{equation*}
 \varphi(\tau,\eta) = e^{- i\lmb\theta} \left(\frac{r}{f}\right)^{\frac{1}{2}} \tpsi,
\end{equation*}
$\varphi$ solves the following remarkably simple equation in the $(\tau, \eta)$-coordinates: introducing $f_* = f/r$,
\begin{equation} \label{eq:e-mhd-eta-conj2-axisym}
\begin{aligned}
& \rd_{\tau}^{2} \varphi - \rd_{\eta}^{2} \varphi + \lmb^{2} f_{*}^{2} \varphi = 0.
\end{aligned}\end{equation} Note that $f_{*}(\eta) \le \frac13 \exp(\eta) \le \frac13$ for $\eta \le 0$ from the assumption $r_{0} \ge 3 \ell$. Furthermore, derivatives decay exponentially as well as $\eta\to-\infty$: $|\rd_{\eta}^{(m)} f_{*}(\eta)| \lesssim_{m} \exp(\eta)$. Now, we take the following ansatz for $\varphi$: 
\begin{equation}\label{eq:varphi}
\varphi(\tau,\eta) = \lmb^{-1} e^{i \lmb \Phi(\tau, \eta) } h(\tau, \eta),
\end{equation} where $\Phi$ and $h$ are independent of $\lmb$. The prefactor $\lmb^{-1}$ is simply a convenient normalization. Plugging this ansatz into \eqref{eq:e-mhd-eta-conj2-axisym} and organizing the terms with the same order in $\lmb$, we obtain the Hamilton--Jacobi equation \begin{equation}\label{eq:HJ}
	\begin{split}
		-(\rd_{\tau}\Phi)^{2} + (\rd_{\eta} \Phi)^{2} + f_{*}^{2} = 0
	\end{split}
\end{equation} and the associated transport equation \begin{equation}\label{eq:AT}
\begin{split}
	(\rd_{\tau}\Phi\rd_{\tau} - \rd_{\eta}\Phi \rd_{\eta})h = - \frac12( \rd_{\tau}^{2}\Phi - \rd_{\eta}^{2}\Phi ) h 
\end{split}
\end{equation}
with some remaining terms, which will be incorporated into the error term. We take the initial condition $h(0, \eta)$ for \eqref{eq:AT} to be a smooth function $h_{0}$ with compact support in the region $\{ -1 < \eta < 0 \}$, and we simply {choose} the function \begin{equation}\label{eq:Phase-choice}
	\begin{split}
		\Phi(\tau,\eta) = \tau + \eta + \int_{-\infty}^{\eta} \left((1 - f_{*}^{2}(\eta'))^{\frac12} - 1\right) \, \ud \eta' 
	\end{split}
\end{equation} as the solution of \eqref{eq:HJ}. The integral is finite thanks to exponential decay of $f_{*}$. Then, \eqref{eq:AT} becomes \begin{equation*}
\begin{split}
	\rd_\tau h - (1-f_{*}^{2})^{\frac12} \rd_{\eta} h = -\frac12 \frac{f_{*}\rd_{\eta} f_{*} }{(1-f_{*}^{2})^{\frac12}} h 
\end{split}
\end{equation*} and using that $f_{*}$ and its $\eta$-derivatives decay exponentially as $\eta\to-\infty$, we obtain (see \cite[Lemmas 3.6, 3.7]{JO1}) \begin{lemma}
The solution $h(\tau,\eta)$ to \eqref{eq:AT} with \eqref{eq:Phase-choice} exists for all $\tau \ge 0$ and satisfies the estimates \begin{equation*}
	\begin{split}
		\max_{0 \le k \le m} \sup_{\tau\ge0} \nrm{ \rd_{\tau}^{k} \rd_{\eta}^{m-k} h  }_{ L^{p}(\bbR_{\eta})} \le C_{m}  \nrm{h_{0}}_{W^{m,p}(\bbR_{\eta})}, \qquad \supp_{\eta}(h(\tau,\cdot)) \subset (-1-\tau,-\tau/2).
	\end{split}
\end{equation*} 
\end{lemma}
Since we have specified $\Phi$ and $h$, this gives $\varphi$ via \eqref{eq:varphi}. Returning to the $(t,r,\tht)$ coordinates system, we define \begin{equation}\label{eq:wp-def}
	\begin{split}
		\tpsi = \Re[ f^{\frac12}\varphi ]= \Re[\lmb^{-1}  e^{i\lmb( \tht + \Phi )} \ell f^{\frac12}h], \qquad \tb^{z} = - \Re[(i\lmb f)^{-1} \lmb^{-1} \rd_{t} (e^{i\lmb( \tht + \Phi )}) \ell f^{\frac12}h ].
	\end{split}
\end{equation}
At the initial time, we set \begin{equation}\label{eq:g-h}
	\begin{split}
		g_{0} (r) =  (f/r)^{\frac12} h(0, r(\eta)), \qquad G(r) = \Phi(0, \eta(r)). 
	\end{split}
\end{equation} This will be convenient since \begin{equation*}
\begin{split}
	\nrm{g_{0}}_{L^2}^2 \sim \int |g_{0}(r)|^2 r\ud r = \int f h^2 \ud r = \int h(0,\eta)^2 \, \ud \eta = \nrm{h(0,\cdot)}_{L^2_\eta}^2.
\end{split}
\end{equation*} Similarly, we have that \begin{equation*}
\begin{split}
	\frac1{C_m} \sum_{j=0}^{m} \ell^{-j}\nrm{\rd_\eta^{(j)}h(0,\eta)}_{L^2_\eta} \le  \nrm{g_{0}}_{H^{m}} \le C_{m}\sum_{j=0}^{m} \ell^{-j}\nrm{\rd_\eta^{(j)}h(0,\eta)}_{L^2_\eta}.
\end{split}
\end{equation*}  
Writing \eqref{eq:wp-def} at $t=0$ using $g_{0}$ gives 		\begin{equation*}
	\begin{split}	
		&\tb^{z}_{(\lmb)}(0) = - f^{-1}\ell \Re [  e^{i \lmb (\theta + G(r))} g_{0}]  , \qquad 			\tpsi_{(\lmb)}(0) = \lmb^{-1}\ell \Re  [e^{i \lmb (\theta + G(r))} g_{0} ].
	\end{split}
\end{equation*} Note that given a smooth $g_{0}$, we may obtain $h(0,\cdot)$ using \eqref{eq:g-h} and then define the associated wave packet $(\tb^{z}_{(\lmb)}, \tpsi_{(\lmb)})$. Based on these considerations, we arrive at the following key propositions, see \cite[Propositions 3.1, 3.4]{JO1}. 


\begin{proposition}[Degenerating wave packets for \eqref{eq:e-mhd}] \label{prop:wavepackets} Given a smooth radial function $g_{0}(r)$ satisfying $\supp g_{0} \subseteq (\tfrac{1}{2}(r_{0}+r_{1}), r_{1})$ and $\lmb \in \bbN$, the associated wave packet $(\tb^{z}_{(\lmb)}, \tpsi_{(\lmb)})[g_{0}]$ constructed in the above satisfies the following
	\begin{itemize}
		\item (initial data) at $t =0$, we have 
		\begin{equation*}
			c \nrm{g_{0}}_{L^{2}} - C_{1} (\lmb\ell)^{-1} \nrm{g_{0}}_{H^{1}} \le \nrm{\tb^{z}_{(\lmb)}(0)}_{L^{2}} + \nrm{\nb \tpsi_{(\lmb)}(0)}_{L^{2}} \le C\nrm{g_{0}}_{L^{2}} + C_{1} (\lmb\ell)^{-1} \nrm{g_{0}}_{H^{1}};
		\end{equation*} 
		\item (regularity estimates) for any $0 \le m < M$ and $ {t \geq 0}$,
		\begin{align*}
			\sup_{0 \leq j \leq m} 
			\nrm{(\lmb^{-2}   \rd_{t})^{k} (\lmb^{-1} \rd_{\tht})^{j} (\lmb^{-1}  f \rd_{r})^{m - k - j} (\tb^{z}_{(\lmb)}, \nb \tpsi_{(\lmb)}) (t)}_{L^{2}} 
			\le C_{M} (1 + (\lmb\ell)^{-1})^{m+1} \nrm{g_{0}}_{H^{m+1}};
		\end{align*}
		\item (degeneration) for any $1 \leq p \leq 2$ and $ {t \geq 0}$,
		\begin{equation}\label{eq:deg-e-mhd}
			\nrm{\tb^{z}_{(\lmb)}(t)}_{L^{2}_{\tht} L^{p}_{rdr}}
			+ \nrm{\nb \tpsi_{(\lmb)}(t)}_{L^{2}_{\tht} L^{p}_{rdr}}
			\le C_{1} e^{- c_{0} (\frac{1}{p} - \frac{1}{2}) \lmb t} \nrm{g_{0}}_{H^{1}};
		\end{equation}
		\item (error bounds)  for $ {t \geq 0}$, we have \begin{equation}\label{eq:error-e-mhd}
			\begin{split}
				\nrm{\err_{\tb}[\tb^{z}_{(\lmb)}, \tpsi_{(\lmb)}](t)}_{L^{2}} \le C_{2} \nrm{g_{0}}_{H^{2}}.
			\end{split}
		\end{equation}  
	\end{itemize}
	In the above statements, $c, c_{0}, C>0$ are absolute constants. 
\end{proposition}

\begin{proof}[Sketch of the proof]
	In the estimates, it is implicitly used that $\ell\le1$ and $\lmb\ge1$. 
	
	To observe the initial data estimate, we consider $\rd_{r} \tilde{\psi}_{(\lmb)}(0)$. The main term is when the derivative falls on $\Phi$, which gives $i\rd_{r}\Phi \ell (f/r)^{\frac12}h e^{i\lmb(\tht+\Phi)}$. Taking its $L^2$ norm gives \begin{equation*}
		\begin{split}
			\iint  |\rd_r\Phi|^2 \ell^2  |h|^2  \frac{f}{r} \, r \ud r \ud \tht  = 2\pi \int |\rd_\eta\Phi|^2 |\rd_\eta r|^{-2} \ell^2 |h(0,\eta)|^2  \, \ud \eta \sim \int |h(0,\eta)|^2 \, \ud \eta \sim  \nrm{g_{0}}_{L^2}^2
		\end{split}
	\end{equation*}  where we have used that $|\rd_\eta r| \sim \ell$ and $|\rd_\eta\Phi|\sim 1$ on the support of $h(0,\cdot)$. 
\end{proof}

\begin{proposition}[Degenerating wave packets for \eqref{eq:hall-mhd}] \label{prop:wavepackets-hall}
	Under the assumptions of Proposition~\ref{prop:wavepackets}, in addition to $(\tb^{z}_{(\lmb)}, \tpsi_{(\lmb)})$, take
	\begin{equation} \label{eq:wavepackets-hall-u-omg}
		\tu^{z}_{(\lmb)}[g_{0}] = - \tpsi_{(\lmb)}[g_{0}], \quad \tomg_{(\lmb)}[g_{0}] = -\tb^{z}_{(\lmb)}[g_{0}].
	\end{equation}
	Then the following properties hold:
	\begin{itemize}
		\item (smoothing for fluid components) for $ {t \geq 0}$, we have
		\begin{align*}
			\nrm{\tu^{z}_{(\lmb)}(t)}_{L^{2}} + \nrm{\nb^{\perp} (-\lap)^{-1} \tomg_{(\lmb)}(t)}_{L^{2}} & \le C_{1}  \lmb^{-1}\ell \nrm{ g_{0} }_{H^{1}} ,\\
			\nrm{\nb \tu^{z}_{(\lmb)}(t)}_{L^{2}} + \nrm{\tomg_{(\lmb)}(t)}_{L^{2}} & \le C_{1} \nrm{ g_{0} }_{H^{1}} ;
		\end{align*}
		\item (error estimates) for $ {t \geq 0}$, we have $\errh_{\tu^{z}}^{(\nu)}[\tu^{z}_{(\lmb)}, \tomg_{(\lmb)}, \tb^{z}_{(\lmb)}, \tpsi_{(\lmb)}] + \nu \lap \tpsi = 0 $ and 
		\begin{align*} 
			\nrm{\nb^{\perp} (-\lap)^{-1} (\errh_{\tomg}^{(\nu)}[\tu^{z}_{(\lmb)}, \tomg_{(\lmb)}, \tb^{z}_{(\lmb)}, \tpsi_{(\lmb)}] + \nu \lap \tomg)(t)}_{L^{2}} & \le C_{2} \lmb^{-1} \nrm{ g_{0} }_{H^{2}}, \\
			\nrm{(\errh_{\tb^{z}}^{(\nu)}, \nb \errh_{\tpsi}^{(\nu)} )[\tu^{z}_{(\lmb)}, \tomg_{(\lmb)}, \tb^{z}_{(\lmb)}, \tpsi_{(\lmb)}](t)}_{L^{2}} & \le C_{2} \nrm{ g_{0} }_{H^{2}}.
		\end{align*}
	\end{itemize} 
\end{proposition}

\begin{remark}[Rescaling wave packets] 
	If $\tilde{f}$ is a smooth function such that $\tilde{f} = \mu f$ for some $\mu>0$ with $f$ satisfying the assumptions in the beginning of Section \ref{subsec:wkb}, then defining wave packets corresponding to $\tilde{f}$ by simply rescaling the time variable $t$ of wave packets for $f$ by $\mu t$, the estimates in the above Propositions hold (with $t$ replaced by $\mu t$). This will be useful in the nonexistence proof from Section \ref{sec:nonlinear-nonexist}.  
\end{remark}

\section{Proof of linear statements}\label{sec:linear}

In this section, we prove the linear illposedness result, Theorem \ref{thm:norm-growth-prime}. 

\subsection{Case of \eqref{eq:e-mhd}} We first handle the  \eqref{eq:e-mhd} case. Let us give an overview of the proof: \begin{itemize}
	\item In \ref{subsubsec:initial}, we fix our choice of background magnetic fields $\bgB$ (compact support in $z$) and $\tilde{\bfB}$ (independent in $z$) as well as initial data $b_0$  (compact support in $z$) and $\tb_0$  (independent in $z$).
	\item In \ref{subsubsec:gei-emhd}, we prove that $b$ and $\chi\tb$ remains close in $L^2$ for some interval of time, where $\chi$ is a cutoff in the $z$-direction adapted to the support of $\bgB$ and $b$.
	\item Then in \ref{subsubsec:growth-Sobolev}, we conclude the proof using Sobolev inequalities. 
\end{itemize}

\subsubsection*{Notations and conventions} Throughout the section, we shall assume that $f(r)$ satisfies the assumptions from Section \ref{subsec:wkb} and furthermore that $\nrm{f}_{L^\infty }\le 1$, $\ell = \ell_f \le 1$, and $\lmb \ge 1$ is taken sufficiently large with respect to $\ell$ so that $\lmb\ell\gg 1$. For a given $\ell$, we fix a smooth cutoff function $\chi = \chi_{\ell}:\bbR_{z}\to \bbR_{\ge0}$ such that $\chi(z) = 1$ for $|z| \le \ell $ and $\chi(z)=0$ for $|z|>2\ell$. In this section, we use tildes to denote $z$-independent functions, and for those, we use the notation $\nrm{\cdot}_{L^2_{x,y}}$. Note that $\nrm{\chi\tb}_{L^2}^{2} \sim \ell \nrm{\tb}_{L^2_{x,y}}^{2} $.

\subsubsection{Choice of initial data}\label{subsubsec:initial}

Let $b$ be an $L^2$-solution defined on $[0,\dlt]$ of the linearized E-MHD system \eqref{eq:e-mhd-lin}, written in the following convenient form: \begin{equation*} 
\begin{split}
& \rd_t b + ( \bfB \cdot\nabla) (\nabla\times b) - ((\nabla\times b)\cdot\nabla) \bfB + (b\cdot\nabla)(\nabla\times \bfB) - ((\nabla\times \bfB)\cdot\nabla) b = 0.
\end{split}
\end{equation*} We shall take $\bfB$ to be the time-dependent background \begin{equation*} 
\begin{split}
&\bfB(t) = \Pi(t,r,z) \rd_\theta
\end{split}
\end{equation*} with $\Pi$ being the unique smooth local solution to \eqref{eq:e-mhd-axisym-background} with initial data $\Pi_0(r,z) = f(r)\chi(z/10)$. The lifespan of $\Pi$ is proportional to $\nrm{\rd_z\Pi_0}_{L^\infty}^{-1} \gtrsim \ell$, and hence as long as $\ell \gtrsim \delta$, the background is well-defined on the time interval of $b$. Furthermore, we shall use the following simple observation which states that $\Pi$ is independent of $z$ in the support of $\chi$ for an interval of time.
\begin{lemma}
We have $\chi(z)\rd_z\Pi(t,r,z) = 0$ for $0 \le t \le 4\ell$. 
\end{lemma}
\begin{proof}
Differentiating \eqref{eq:e-mhd-axisym-background} in $z$, we have \begin{equation*} 
	\begin{split}
		& \rd_t  (\rd_z\Pi) + 2\Pi \rd_z (\rd_z\Pi) = - 2 (\rd_z\Pi)^2. 
	\end{split}
\end{equation*} Defining the characteristics by $Z(t,z_0) = 2\Pi(t,Z(t,z_0))$ and $Z(0,z_0) = z_0$, we have from the above that $\rd_z\Pi(t) = 0$ on the set $Z(t, [-10\ell,10\ell])$. Since $\chi$ vanishes for $|z| \ge 2\ell$, it suffices that $$2t = \int_0^t \nrm{2\Pi(s)}_{L^\infty} ds \le 8\ell$$ for $\chi(z)\rd_z\Pi(t,r,z) = 0$. 
\end{proof} In light of the above lemma, we shall (implicitly) assume that $\delta \le 4\ell$.   

\medskip

We take the degenerating wave packets corresponding to $\tld{\bfB} = f(r) \rd_{\tht}$ from the previous section, and introduce the simplifying notation $\tilde{b} = (-\nb^\perp \tpsi, \tb^{z} )$. Here, we fix some universal smooth function $h(0,\cdot)$ and use \eqref{eq:g-h} to define $g_{0}$. We shall freely use the estimates $\nrm{g_{0}}_{H^{m}_{x,y}} \lesssim_{m} \ell^{-m}\nrm{g_{0}}_{L^2_{x,y}}$, and suppress from writing out the dependence of constants in $m, C_{m}$. 

\medskip 

The equation for $\tb$ can be written in the following economical form \begin{equation*} 
\begin{split}
\rd_t\tb + (\tilde{\bfB}\cdot\nabla) (\nabla\times \tb) - ((\nabla\times \tb)\cdot\nabla) \tilde{\bfB} = \err_{\tb},
\end{split}
\end{equation*} with $$\sup_{t \in [0,\delta]} \nrm{\err_{\tb}}_{L^2_{x,y}} \lesssim \ell^{-2}\nrm{g_0}_{L^2_{x,y}}.$$ 
 
We now fix our choice of initial data $b_{0}$. To take it to be close in $L^2$ to $\chi(z)\tb_0$, under the divergence-free constraint,  we write \begin{equation*} 
\begin{split}
& b_0 = \chi(z)\tb_0 + \zeta
\end{split}
\end{equation*} for some $\zeta$. Then, taking the divergence of both sides, we see that $\zeta$ must satisfy \begin{equation*} 
\begin{split}
& 0 = \chi'(z)\tb_0^z + \mathrm{div}\,\zeta,
\end{split}
\end{equation*}  which suggests the choice $b_0 := \chi(z)\tb_0 - \mathrm{div}_U^{-1}(\chi'\tb_0^z)$, where $U$ can be taken to be any cube of side length $20\ell$ (say) containing the support of $\chi'\tb_0^z$. Using \eqref{eq:Bogovskii-L2-3}, we estimate that \begin{equation*} 
\begin{split}
\nrm{\rd_{\tht} (\mathrm{div}_U^{-1}\rd_{\tht}^{-1}(\chi'\tb_0^z))}_{L^2} \lesssim \lmb^{-1}\nrm{ \chi' \lmb\rd_{\tht}^{-1} \tb_0^z   }_{L^2} \lesssim \lmb^{-1} \ell^{\frac{1}{2}} \nrm{\tb_0^z}_{L^2_{x,y}}. 
\end{split}
\end{equation*} Hence we obtain in particular that \begin{equation}\label{eq:initial-e-mhd} 
\begin{split}
& \brk{b_0, \chi\tb_0} = \ell^{\frac{1}{2}}\nrm{b_0}_{L^2}\nrm{\tb_0}_{L^2_{x,y}} ( 1 + O(\lmb^{-1})). 
\end{split}
\end{equation}

\subsubsection{Generalized energy estimate}\label{subsubsec:gei-emhd} Now that we have specified $b$ and $\tb$, let us proceed with an argument involving the general energy identity: we compute \begin{equation*} 
\begin{split}
\frac{\ud}{\ud t} \brk{b, \chi\tb} & = -\brk{ \nabla\times ((\nabla\times b)\times \bgB), \chi\tb} - \brk{b,\chi \nabla\times ((\nabla\times \tb) \times \tilde{\bfB})} \\
& \qquad  + \brk{ ((\nabla\times\bgB)\cdot\nabla) b  , \chi\tb } - \brk{ (b\cdot\nabla)(\nabla\times \bgB), \chi\tb} + \brk{b, \chi \err_{\tb}}.
\end{split}
\end{equation*} 
First, we take care of the terms that are trivially bounded in $L^2$: \begin{equation*} 
\begin{split}
 \left| - \brk{ (b\cdot\nabla)(\nabla\times \bgB), \chi\tb} + \brk{b, \chi \err_{\tb}} \right| & \lesssim \nrm{\nabla^2\bgB}_{L^\infty} \nrm{b}_{L^2} \nrm{\chi\tb}_{L^2} + \nrm{b}_{L^2} \nrm{\chi \err_{\tb}}_{L^2} \lesssim \ell^{-4 + \frac{1}{2}}  \nrm{b}_{L^2} \nrm{g_0}_{L^2_{x,y}}
\end{split}
\end{equation*} Then, we rewrite the first term as \begin{equation*} 
\begin{split}
-\brk{ \nabla\times ((\nabla\times b)\times \bgB), \chi\tb} & = \brk{b,  \nabla\times  ((\nabla\times (\chi\tb)) \times \bgB) } = \brk{b, (\bgB\cdot\nabla)(\nabla\times(\chi\tb))} - \brk{b, (\nabla\times(\chi\tb)) \cdot\nabla \bgB }
\end{split}
\end{equation*} and organize the remaining terms as follows: \begin{equation*} 
\begin{split}
 \mathrm{I} = \brk{ ((\nabla\times\bgB)\cdot\nabla) b  , \chi\tb } , 
\end{split}
\end{equation*} \begin{equation*} 
\begin{split}
 \mathrm{II} =   - \brk{b,   (\nabla\times(\chi\tb))\cdot\nabla\bgB - \chi(\nabla\times\tb)\cdot\nabla\tilde{\bfB}} , 
\end{split}
\end{equation*} and \begin{equation*} 
\begin{split}
 \mathrm{III} =   \brk{b, (\bgB\cdot\nabla)(\nabla\times(\chi\tb))  -  \chi (\tilde{\bfB}\cdot\nabla) (\nabla\times \tb) } .
\end{split}
\end{equation*} 

\medskip

\noindent \textit{Estimate for I.} Using $\nabla\times \bgB = -r\rd_z\Pi \rd_r + r^{-1}\rd_r(r^2\Pi)\rd_z = r^{-1}\rd_r(r^2f)\rd_z$ on the support of $\chi$,  \begin{equation*} 
\begin{split}
|\mathrm{I}| = |\brk{ r^{-1}\rd_r(r^2\Pi) \rd_z b, \chi\tb }| &\lesssim \ell^{-\frac{1}{2}}(\nrm{rf'(r)}_{L^\infty} + \nrm{f}_{L^\infty})  \nrm{b}_{L^2} \nrm{g_0}_{L^2_{x,y}} \lesssim \ell^{-\frac{1}{2}}\nrm{b}_{L^2} \nrm{g_0}_{L^2_{x,y}} . 
\end{split}
\end{equation*}

\medskip

\noindent \textit{Estimate for II.} Note that we obtain cancellations unless a derivative falls on $\chi$ in the expression $\nabla\times(\chi\tb)$, simply because $\bgB = \tilde{\bfB}$ on the support of $\chi$. Then there are no derivatives falling on $b, \tb$, so that \begin{equation*} 
\begin{split}
& |\mathrm{ II} | \lesssim  \ell^{-\frac{1}{2}}\nrm{b}_{L^2}  \nrm{g_0}_{L^2_{x,y}} . 
\end{split}
\end{equation*}  

\medskip

\noindent  \textit{Estimate for III.} We begin with writing $\nabla \times (\chi \tb) = \chi \nabla\times \tb + \chi'(z) ( -r\tb^\tht \rd_r + r^{-1}\tb^r \rd_\tht  )$ and \begin{equation*} 
\begin{split}
\mathrm{III }&  = -  \brk{  b^r, \chi'(z) f(r)\rd_\tht(r\tb^\tht) } + \brk{ rb^\tht , r\chi'(z) f(r)\rd_\tht (r^{-1}\tb^r) } \\
& = \brk{b^r, \chi'(z)f(r)(r\rd_r\tb^r + \tb^r)} + \brk{\rd_rb^r + r^{-1}b^r + \rd_z b^z, \chi'(z)f(r) r\tb^r} . 
\end{split}
\end{equation*} Note that an integration by parts gives \begin{equation*} 
\begin{split}
& \brk{b^r, \chi'(z)f(r)r\rd_r\tb^r} = - \brk{\rd_rb^r, \chi'(z)f(r)r\tb^r} - 2\brk{b^r, \chi'(z)f(r)\tb^r} - \brk{b^r, \chi'(z)f'(r)r\tb^r} ,
\end{split}
\end{equation*}  so that after cancellations, \begin{equation*} 
\begin{split}
 |\mathrm{III}| = |\brk{b^r,\chi'(z)f'(r)r\tb^r}-\brk{b^z , \chi''(z)f(r) r\tb^r}| & \lesssim  ( \ell^{-\frac{3}{2}}\nrm{rf}_{L^\infty} + \ell^{-\frac{1}{2}}\nrm{rf'(r)}_{L^\infty} )\nrm{b}_{L^2}\nrm{g_0}_{L^2_{x,y}} \\ & \lesssim \ell^{-\frac{1}{2}}\nrm{b}_{L^2}\nrm{g_0}_{L^2_{x,y}}. 
\end{split}
\end{equation*} Collecting all the estimates, we conclude that \begin{equation}\label{eq:gei-e-mhd} 
\begin{split}
\left| \frac{\ud}{\ud t} \brk{b, \chi \tb} \right|   \lesssim \ell^{-\frac{3}{2}}\nrm{b}_{L^2}\nrm{g_0}_{L^2_{x,y}} . 
\end{split}
\end{equation}

\subsubsection{Growth of Sobolev norms}\label{subsubsec:growth-Sobolev} Restricting further to a time interval of length $\eps \ell^2 $ with some small absolute constant $ \eps > 0$, we may ensure that $\nrm{b}_{L^2} \le 2\nrm{b_0}_{L^2}$. Within this time interval, we integrate \eqref{eq:gei-e-mhd} in time and use \eqref{eq:initial-e-mhd} to obtain \begin{equation*} 
\begin{split}
& \brk{b,\chi\tb} (t) \ge \frac12 \ell^{\frac{1}{2}} \nrm{b_0}_{L^2} \nrm{g_0}_{L^2_{x,y}}, 
\end{split}
\end{equation*} by further restricting the interval of time to satisfy $  t \le \eps \ell^{2}$. The rest of the proof is straightforward using the degeneration estimate from Proposition \ref{prop:wavepackets} and completely parallel to the corresponding argument from \cite[proof of Theorem \ref{thm:norm-growth}]{JO1}. 

\medskip 

In the case $s = 0$ and $p>2$, we use the degeneration estimate \eqref{eq:deg-e-mhd} for $\nrm{\tb}_{L^2_{\tht }L^{p'}_{rdr}}$ where $p'<2$ is the dual exponent of $p$, and \begin{equation*} 
\begin{split}
&\nrm{b(t)}_{L^{p}} \nrm{\chi\tb(t)}_{L^{p'}}  \ge \brk{b,\chi\tb} (t) \ge \frac12 \ell^{\frac{1}{2}} \nrm{b_0}_{L^2} \nrm{g_0}_{L^2_{x,y}}
\end{split}
\end{equation*} with \begin{equation*}
\begin{split}
	 \nrm{\chi\tb(t)}_{L^{p'}} \lesssim_{\ell}  \nrm{\tb(t)}_{L^2_{\tht} L^{p'}_{rdr}} 
\end{split}
\end{equation*}  to deduce that, up to some powers of $\ell$, \begin{equation*} 
\begin{split}
& \nrm{b(t)}_{L^{p}} \gtrsim_{\ell} e^{c_0(\frac{1}{2}-\frac{1}{p})\lmb t} \nrm{b_0}_{L^2} , \qquad t \in [0, \eps \ell^2].
\end{split}
\end{equation*}

The general case of $s \ge 0$ and $s + \frac{1}{2} - \frac{1}{p} > 0$ can be handled in a similar way. \hfill \qedsymbol  

\subsection{Case of \eqref{eq:hall-mhd}}\label{subsec:hallmhd-axisym} In this subsection, we treat the \eqref{eq:hall-mhd} case.  While the computations are significantly more involved in this case, the overall strategy is essentially the same with the case of \eqref{eq:e-mhd}, and in particular, the initial data will be taken to be $u_0 = 0$ and $b_0$ the same as before. The key difference lies in the way a suitable lower bound is obtained in the $L^2$-inner product $\brk{b,\chi\tb}$.

For convenience, let us recall the system of equations \eqref{eq:axisym-hall-pert} satisfied by an $L^2$-solution $(u,b)$ around a time-dependent axisymmetric background: 
\begin{equation}\label{eq:axisym-hall-pert-recall}
\left\{
\begin{aligned}
& \rd_t b + \bfPi \cdot\nabla (\nabla\times b) - (\nabla\times b) \cdot \nabla \bfPi - (\nabla \times \bfPi) \cdot \nabla b + \bfV\cdot \nabla b - \bfPi\cdot\nabla u \\
& \quad  = -b\cdot\nabla(\nabla\times\bfPi) + b\cdot\nabla\bfV - u\cdot\nabla\bfPi  , \\
& \rd_t u + \bfV \cdot \nabla u + \nabla p - \nu\lap u - (\nabla\times b)\times \bfPi = -u\cdot\nabla\bfV + (\nabla\times\bfPi)\times b  .
\end{aligned}
\right.
\end{equation} Here, we used the notation $\bfPi = \Pi \rd_\tht$ and $\bfV = \nabla\times(-\lap)^{-1}(\Omega\rd_\tht)$, where the pair of scalar-valued functions $(\Pi,\Omega)$ provides a solution to \eqref{eq:hall-mhd-axisym-background}. 

Furthermore, recall the system for $(\tu,\tb)$ (omitting the superscript $(\nu)$ for simplicity) \begin{equation*}
\left\{
\begin{aligned}
& \rd_t\tu - f\rd_\tht \tb = \errh_{\tu} \\
& \rd_t\tb + f\rd_\tht (\nabla\times\tb) + r^{-1} f\rd_\tht \tb \rd_\tht = \errh_{\tb}.
\end{aligned}
\right.
\end{equation*} We shall need the following simple estimate, which is a straightforward consequence of Proposition \ref{prop:wavepackets}:  \begin{equation}\label{eq:rd_r_bruteforce}
\begin{split}
\nrm{\rd_r \tb}_{L^2_{x,y}} & \lesssim \lmb \ell^{-2} \nrm{f^{-1}}_{L^\infty(\mathrm{supp}(\tb(t)))} \nrm{g_0}_{L^2_{x,y}} \lesssim \lmb  \ell^{-2} e^{c_{1} \lmb t} \nrm{g_0}_{L^2_{x,y}}. 
\end{split}
\end{equation} Similarly, we have  \begin{equation}\label{eq:rd_tht_bruteforce}
\begin{split}
\nrm{\rd_\tht \tb}_{L^2_{x,y}} & \lesssim \lmb \ell^{-2} \nrm{g_0}_{L^2_{x,y}} , \qquad \nrm{\rd_\tht\rd_r \tb}_{L^2_{x,y}} \lesssim \lmb^2  \ell^{-3} e^{c_{1}\lmb t} \nrm{g_0}_{L^2_{x,y}}. 
\end{split}
\end{equation} 
We write out the terms as follows: \begin{equation*} 
\begin{split}
& \frac{\ud}{\ud t} \left( \brk{u, \chi \tu - \mathrm{div}^{-1}(\chi' \tu^z)} + \brk{b, \chi\tb} \right) = \mathrm{I} + \mathrm{II} + \mathrm{III}+ \mathrm{IV}+ \mathrm{V}+ \mathrm{VI}, 
\end{split}
\end{equation*} where \begin{equation*} 
\begin{split}
& \mathrm{I} = \brk{\rd_t u + \nabla p, - \mathrm{div}^{-1}(\chi' \tu^z) } + \brk{u, \mathrm{div}^{-1}(\chi' \rd_t \tu^z)} , 
\end{split}
\end{equation*} \begin{equation*} 
\begin{split}
& \mathrm{II} = \brk{u, \chi \errh_{\tu}} + \brk{b, \chi \errh_{\tb}}, 
\end{split}
\end{equation*} \begin{equation*} 
\begin{split}
& \mathrm{III} = \brk{ -b\cdot\nabla(\nabla\times\bfPi) + b\cdot\nabla\bfV - u\cdot\nabla\bfPi  , \chi \tu} + \brk{-u\cdot\nabla\bfV + (\nabla\times\bfPi)\times b  , \chi \tb},
\end{split}
\end{equation*} \begin{equation*} 
\begin{split}
& \mathrm{IV} = -\brk{\bfV\cdot \nabla u , \chi \tu} - \brk{\bfV\cdot\nabla b, \chi \tb},
\end{split}
\end{equation*} \begin{equation*} 
\begin{split}
& \mathrm{V} = \brk{\nu \lap u, \chi \tu} + \brk{(\nabla\times b) \times \bfPi, \chi \tu} + \brk{(\nabla\times \bfPi)\cdot\nabla b, \chi \tb},
\end{split}
\end{equation*} and finally further decompose the remaining terms as $\mathrm{VI} = \mathrm{VI}_1 + \mathrm{VI}_2 + \mathrm{VI}_3$ where  \begin{equation*} 
\begin{split}
\mathrm{VI}_1 & = \brk{u, \chi f \rd_\tht \tb} + \brk{\bfPi\cdot\nabla u, \chi \tb}, \\
\mathrm{VI}_2 &= -\brk{\bfPi\cdot\nabla \nabla\times b, \chi\tb} + \brk{b, -\chi f\rd_\tht \nabla\times\tb}  ,\\
\mathrm{VI}_3 & = \brk{\nabla\times b\cdot\nabla \bfPi, \chi \tb} - \brk{b, \chi r^{-1}f \rd_\tht \tb \rd_\tht }.
\end{split}
\end{equation*}
We begin with rewriting $\mathrm{I}$ as \begin{equation*} 
\begin{split}
\mathrm{I} &= \brk{ \bfV \cdot\nabla u - \nu\lap u - (\nabla\times b)\times \bfPi + u\cdot\nabla\bfV - (\nabla\times\bfPi) \times b, \mathrm{div}^{-1}(\chi'\tu^z) } \\
&\quad  + \brk{u, \mathrm{div}^{-1}(\chi' f \rd_\tht \tb)} + \brk{u, \mathrm{div}^{-1}(\chi' \errh_{\tu})}.
\end{split}
\end{equation*} Using \begin{equation*} 
\begin{split}
& \nrm{\nabla \mathrm{div}^{-1}(\chi'\tu^z)}_{L^{2}_{x,y}} \lesssim \nrm{\chi'\tu^z}_{L^{2}_{x,y}} \lesssim \lmb^{-1}\ell^{\frac{1}{2}} \nrm{g_0}_{L^{2}_{x,y}}, \quad \nrm{\mathrm{div}^{-1}(\chi'\tu^z)}_{L^{2}_{x,y}} \lesssim \ell^{-1}\nrm{\chi'\tu}_{L^{2}_{x,y}} \lesssim \lmb^{-1}\ell^{-\frac{1}{2}}\nrm{g_0}_{L^{2}_{x,y}},
\end{split}
\end{equation*} we bound \begin{equation*} 
\begin{split}
& \left| \brk{ \bfV \cdot\nabla u - \nu\lap u - (\nabla\times b)\times \bfPi + u\cdot\nabla\bfV - (\nabla\times\bfPi) \times b, \mathrm{div}^{-1}(\chi'\tu^z) } \right| \\
&\quad \lesssim \lmb^{-1} \ell^{\frac{1}{2}} \left( (\nrm{\nabla\bfV}_{L^\infty} + \ell^{-1}\nrm{\bfV}_{L^\infty})\nrm{u}_{L^2} + (\nrm{\nabla\bfPi}_{L^\infty} + \ell^{-1}\nrm{\bfPi}_{L^\infty})\nrm{b}_{L^2} + \nu \nrm{\nabla u}_{L^2} \right) \nrm{g_0}_{L^{2}_{x,y}}
\end{split}
\end{equation*} and \begin{equation*} 
\begin{split}
&\left| \brk{u, \mathrm{div}^{-1}(\chi' f \rd_\tht \tb)} + \brk{u, \mathrm{div}^{-1}(\chi' \errh_{\tu})} \right| \lesssim \ell^{\frac{1}{2}}(1 + \ell^{-1})\nrm{u}_{L^2}\nrm{g_0}_{L^{2}_{x,y}}.
\end{split}
\end{equation*} Terms in $\mathrm{II}$ and $\mathrm{III}$ can be handled in a straightforward manner: we have \begin{equation*} 
\begin{split}
&|\mathrm{II}| \lesssim \ell^{-\frac{3}{2}}(\nrm{u}_{L^2} + \nrm{b}_{L^2})\nrm{g_0}_{L^{2}_{x,y}}
\end{split}
\end{equation*} and \begin{equation*} 
\begin{split}
& |\mathrm{III}| \lesssim \ell^{\frac{1}{2}} \left( (\nrm{\nabla^2\bfPi}_{L^\infty} + \nrm{\nabla\bfPi}_{L^\infty} + \nrm{\nabla\bfV}_{L^\infty} )\nrm{b}_{L^2} + (\nrm{\nabla\bfPi}_{L^\infty} + \nrm{\nabla\bfV}_{L^\infty})\nrm{u}_{L^2} \right)\nrm{g_0}_{L^{2}_{x,y}} ,
\end{split}
\end{equation*} respectively. We now treat the terms in $\mathrm{IV}$. Noting that \begin{equation*} 
\begin{split}
&\brk{\bfV\cdot\nabla b,\chi\tb} = -\brk{b, \chi V^r \rd_r \tb + V^z\chi' \tb }, \quad \brk{\bfV\cdot\nabla u,\chi\tu} = -\brk{u, \chi V^r \rd_r \tu + V^z\chi' \tu }
\end{split}
\end{equation*} and using \eqref{eq:rd_r_bruteforce} to bound $\brk{b,\chi V^r\rd_r\tb}$, we obtain \begin{equation*} 
\begin{split}
& |\mathrm{IV}| \lesssim  \ell^{\frac{1}{2}}((1+\ell^{-1}) + \lmb e^{c\lmb t} ) \nrm{\bfV}_{L^\infty}(\nrm{u}_{L^2} + \nrm{b}_{L^2})\nrm{g_0}_{L^{2}_{x,y}} .
\end{split}
\end{equation*} Next, recalling that \begin{equation*} 
\begin{split}
&\nabla\times(\Pi\rd_\theta) = (-r\rd_z\Pi)\rd_r + (r^{-1}\rd_r(r^2\Pi))\rd_z,
\end{split}
\end{equation*} we write \begin{equation*} 
\begin{split}
& \brk{(\nabla\times \bfPi)\cdot\nabla b, \chi \tb} = \brk{-r\rd_z\Pi \rd_r b, \chi \tb} + \brk{r^{-1}\rd_r(r^2\Pi)\rd_z b, \chi \tb} 
\end{split}
\end{equation*} and we may bound \begin{equation*} 
\begin{split}
& \left| \brk{-r\rd_z\Pi \rd_r b, \chi \tb}  \right|   \lesssim \left| \brk{b, \chi \rd_r(r\rd_z\Pi \tb)} \right| \lesssim \ell^{\frac{1}{2}} (\ell^{-1}\nrm{\bfPi}_{L^\infty} + \lmb e^{c_{1} \lmb t} \nrm{\rd_z\bfPi}_{L^\infty}) \nrm{b}_{L^2} \nrm{g_0}_{L^{2}_{x,y}}.
\end{split}
\end{equation*} The remaining terms from $\mathrm{V}$ are easy to bound; we end up with \begin{equation*} 
\begin{split}
& \left| \mathrm{V} \right| \lesssim \ell^{\frac{1}{2}}\left( \nu \nrm{\nabla u}_{L^2} +  ( \lmb e^{c_{1} \lmb t} \nrm{\rd_z\bfPi}_{L^\infty} + \ell^{-1}\nrm{\bfPi}_{L^\infty} )\nrm{b}_{L^2}  \right) \nrm{g_0}_{L^{2}_{x,y}} .
\end{split}
\end{equation*} Finally, we write \begin{equation*} 
\begin{split}
&\mathrm{VI} = \mathrm{VI}^e + \brk{(\bfPi-\bfPi^e)\cdot\nabla u, \chi \tb} -\brk{(\bfPi-\bfPi^e)\cdot\nabla \nabla\times b, \chi\tb}  + \brk{\nabla\times b\cdot\nabla (\bfPi-\bfPi^e), \chi \tb}
\end{split}
\end{equation*} where $\bfPi^e$ is as in Lemma \ref{lem:hall-mhd-fradiss-axi-small-t} and \begin{equation*} 
\begin{split}
\mathrm{VI}^e & := \brk{u, \chi f \rd_\tht \tb} + \brk{\bfPi^e\cdot\nabla u, \chi \tb}  -\brk{\bfPi^e\cdot\nabla \nabla\times b, \chi\tb}  - \brk{b, \chi f\rd_\tht \nabla\times\tb}  + \brk{\nabla\times b\cdot\nabla \bfPi^e, \chi \tb} - \brk{b, \chi r^{-1}f \rd_\tht \tb \rd_\tht } .
\end{split}
\end{equation*} From the computations in the case of E-MHD in \ref{subsubsec:gei-emhd}, \begin{equation*} 
\begin{split}
&\left| \mathrm{VI}^e \right|  \lesssim \ell^{-\frac{3}{2}} \nrm{b}_{L^2} \nrm{g_0}_{L^2_{x,y}}. 
\end{split}
\end{equation*} For the other terms, we have \begin{equation*} 
\begin{split}
&\left| \brk{(\bfPi - \bfPi^e)\cdot\nabla u, \chi \tb} \right| = \left| \brk{r^{-1}(\bfPi - \bfPi^e)u, \chi \rd_\tht\tb} \right| \lesssim \lmb\ell^\frac{1}{2} \nrm{\nabla(\bfPi - \bfPi^e)}_{L^\infty} \nrm{u}_{L^2} \nrm{g_0}_{L^2_{x,y}} ,
\end{split}
\end{equation*} \begin{equation*} 
\begin{split}
&\left| \brk{(\bfPi-\bfPi^e)\cdot\nabla \nabla\times b, \chi\tb}  \right| = \left| \brk{b, \nabla\times(r^{-1}(\Pi - \Pi^e)\chi \rd_\tht\tb)} \right| \\ &\qquad \lesssim \ell^{\frac{1}{2}}(\lmb(1+\lmb) e^{c_{1} \lmb t} + \ell^{-1} ) \nrm{\nabla^2(\bfPi - \bfPi^e)}_{L^\infty} \nrm{b}_{L^2} \nrm{g_0}_{L^2_{x,y}},
\end{split}
\end{equation*} and \begin{equation*} 
\begin{split}
& \left| \brk{\nabla\times b\cdot\nabla (\bfPi-\bfPi^e), \chi \tb} \right| = \left| \brk{b, \nabla\times ((\chi \tb)\cdot \nabla(\bfPi - \bfPi^e))} \right| \\
&\qquad \lesssim  \ell^{\frac{1}{2}}((\lmb e^{c\lmb t} + \ell^{-1} ) \nrm{\nabla (\bfPi - \bfPi^e)}_{L^\infty} + \nrm{ \nabla^2(\bfPi - \bfPi^e)  }_{L^\infty} )\nrm{b}_{L^2} \nrm{g_0}_{L^2_{x,y}} .
\end{split}
\end{equation*} Collecting the terms, and recalling our assumptions $\lmb \ge 1$, $\ell \le 1$, \begin{equation*} 
\begin{split}
& \left|\frac{\ud}{\ud t} \left( \brk{u, \chi \tu - \mathrm{div}^{-1}(\chi' \tu^z)} + \brk{b, \chi\tb} \right) \right|   \lesssim \ell^{-\frac{3}{2}}\left[ 1 +  \nu\nrm{\nabla u}_{L^2} + (\nrm{\bfPi}_{W^{2,\infty}} + \nrm{\bfPi-\bfPi^e}_{W^{2,\infty}} + \nrm{\bfV}_{W^{1,\infty}}) \right. \\
& \qquad  \left. + e^{c_{1} \lmb t}( \lmb\nrm{\rd_z\Pi}_{L^\infty} + \lmb^2 \nrm{\bfPi-\bfPi^e}_{W^{1,\infty}} + \lmb \nrm{\bfV}_{L^\infty}  ) (\nrm{u}_{L^2} + \nrm{b}_{L^2}) \right]   \nrm{g_0}_{L^2_{x,y}} . 
\end{split}
\end{equation*} (Here we note that the norms on $\bfPi, \bfV$, etc are taken only on the support of $\chi$.)

We now look for a  time interval where the previous estimate is effective. To begin with, we use \eqref{eq:apriori-hall-axi} to obtain \begin{equation*} 
\begin{split}
& \nrm{u(t)}_{L^2}^2 + \nrm{b(t)}_{L^2}^2 + \nu\int_0^t \nrm{\nabla u(s)}_{L^2} \ud s \le 2 \nrm{b_0}_{L^2}^2 .
\end{split}
\end{equation*} Furthermore, we shall restrict $t \in [0,T]$, where $T = c \ell^{m_0}$ for some small absolute constant $c > 0$ and large $m_0 \ge 1$ ($m_{0}=10$ suffices) so that
\begin{equation*} 
	\begin{split}
		& \nrm{\bfPi}_{L^\infty([0,T];W^{2,\infty})} + \nrm{\bfPi-\bfPi^e}_{L^\infty([0,T];W^{2,\infty})} + \nrm{\bfV}_{L^\infty([0,T];W^{1,\infty})} \lesssim \ell^{-m_0}
	\end{split}
\end{equation*} holds, using the estimate \eqref{eq:hall-mhd-fradiss-axi-Hm} from Proposition \ref{prop:hall-mhd-fradiss-axi-lwp}. In what follows, we shall take $T$ even smaller whenever it becomes necessary.  Next, from Lemma \ref{lem:hall-mhd-fradiss-axi-small-t}, we obtain smallness estimates \begin{equation*} 
\begin{split}
& \nrm{\bfV}_{L^\infty}  \lesssim t  \ell^{-m_0}, \qquad \nrm{\rd_z\Pi}_{L^\infty(\mathrm{supp}(\chi))} = \nrm{\rd_z(\Pi-\Pi^e)}_{L^\infty(\mathrm{supp}(\chi))} \lesssim \nrm{\bfPi-\bfPi^e}_{W^{1,\infty}} \lesssim t^2 \ell^{-m_0}
\end{split}
\end{equation*}
Furthermore, recalling the simple bound  \begin{equation*} 
\begin{split}
& \left|\brk{u,\chi \tu -  \mathrm{div}^{-1}(\chi' \tu^z)  } \right| \lesssim \lmb^{-1} \ell^{\frac{1}{2}} \nrm{u}_{L^2}  \nrm{g_0}_{L^2_{x,y}}, 
\end{split}
\end{equation*} we obtain that \begin{equation*} 
\begin{split}
&\left| \brk{b(t),\chi\tb(t)} - \brk{b_0,\chi\tb_0} \right|   \lesssim \ell^{-m_0+\frac{1}{2}} \left( \nu t^{\frac{1}{2}} + t + \int_0^t \left(1 + \lmb s + ( \lmb s )^{2} \right) e^{c_{1} \lmb s }  \ud s  \right) \nrm{b_0}_{L^2}\nrm{g_0}_{L^2_{x,y}}.
\end{split}
\end{equation*} We estimate the time integral as \begin{equation*} 
\begin{split}
 &\int_0^t \left(1 + \lmb s  + ( \lmb s )^2  \right) e^{c_{1} \lmb s }  \ud s  \lesssim \int_0^t   e^{2 c_{1}   \lmb s }  \ud s \lesssim \lmb^{-1}\ell ( e^{2 c_{1} \lmb t  } -1) 
\end{split}
\end{equation*} so that by taking $T$ smaller than $c \lmb^{-1}\ln\lmb $ for some small $c > 0$ and using \begin{equation*} 
\begin{split}
& \brk{b_0,\chi\tb_0} =  \ell^{\frac{1}{2}} \nrm{b_0}_{L^2}\nrm{g_0}_{L^2_{x,y}} ( 1 + O(\lmb^{-1})),
\end{split}
\end{equation*} we deduce that, for $t \in [0,T]$, \begin{equation*} 
\begin{split}
&\brk{b,\chi\tb}(t) \ge \frac12 \ell^{\frac{1}{2}} \nrm{b_0}_{L^2}\nrm{g_0}_{L^2_{x,y}} .
\end{split}
\end{equation*}  (Here it is implicit that $\lmb$ is taken sufficiently large with respect to $\nu$  and $\ell^{-1}$.) The rest of the argument is completely parallel to \ref{subsubsec:growth-Sobolev}.  \hfill \qedsymbol

\section{Unboundedness of the solution map} \label{sec:nonlinear-unbounded}

In this section, we give the proof of Theorem \ref{thm:illposed-strong-prime}. We first handle the simpler case of \eqref{eq:e-mhd} in Section \ref{subsec:nonlinear-emhd} and then the \eqref{eq:hall-mhd} case is treated in Section \ref{subsec:nonlinear-hallmhd}. 

\subsection{Case of \eqref{eq:e-mhd}}\label{subsec:nonlinear-emhd}
 
As before, we take the background magnetic field to be $\bfPi(t) = \Pi(t,r,z)\rd_\tht$ with initial data $f(r)\chi(z/10)$, where $f, \chi$ are the same as in \ref{subsubsec:background}. Towards a contradiction, we shall assume that the solution map is well-defined and bounded: to be precise, we are given some $\eps, \dlt,  s  > 0$, {$s_0 > 7/2$}, and $M>0$ such that the solution map is defined as an operator \begin{equation*} 
	\begin{split}
		& \calB_{\eps}(\bfPi_0; H^s_{comp}) \rightarrow L^\infty_t([0,\dlt];H^{s_0}) 
	\end{split}
\end{equation*} with the $L^\infty_t([0,\dlt];H^{s_0}) $ norm of the solution bounded by $M$. 

We shall take $\lmb \gg \ell^{-1}$ and derive a contradiction in the end by taking $\lmb\to\infty$. For a given $\lmb$, we take  the same profile $b = b_{(\lmb)}$ as in the proof of Theorem \ref{thm:norm-growth-prime}; note that we have at the initial time (using its explicit form in \ref{subsubsec:initial}) \begin{equation*} 
	\begin{split}
		&\nrm{b_0}_{H^s} \lesssim_s (\lmb + \ell^{-1})^{s} \nrm{b_0}_{L^2} \lesssim \lmb^{s} \nrm{b_{0}}_{L^{2}} . 
	\end{split}
\end{equation*}   Therefore, given $\eps>0$, we shall simply replace $b$ by $c_{s}\eps\lmb^{-s}b$ ($c_{s}>0$ is a constant depending only on $s$) so that at the initial time we have $\nrm{b_0}_{H^{s}} <  \eps$.  
 
In particular, $\bfPi_{0} + b_{0} \in \calB_{\eps}(\bfPi_0; H^s_{comp})$ and there exists a solution $\bfB$ belonging to $L^\infty_t([0,\dlt];H^{s_0})$ to \eqref{eq:e-mhd} with initial data $\bfPi_0 + b_0$. We then set $b := \bfB - \bfPi$, so that $b(t = 0) = b_0$. We have that $b$ solves \begin{equation}\label{eq:e-mhd-nonlinear-pert} 
\begin{split}
 \rd_t b + (\bfPi\cdot\nabla) (\nabla\times b) - ((\nabla\times b)\cdot\nabla) \bfPi + (b\cdot\nabla)(\nabla\times \bfPi) - ((\nabla\times\bfPi)\cdot\nabla) b = \err_b,
\end{split}
\end{equation} \begin{equation*} 
\begin{split}
 \err_b := -(b\cdot\nabla)\nabla\times b + (\nabla\times b)\cdot\nabla b .
\end{split}
\end{equation*} Next, we recall the $z$-independent approximate solution $\tb$: \begin{equation*} 
\begin{split}
\rd_t\tb + (\tilde{\bfB}\cdot\nabla) (\nabla\times \tb) - ((\nabla\times \tb)\cdot\nabla) \tilde{\bfB} = \err_{\tb},
\end{split}
\end{equation*}\begin{equation*} 
\begin{split}
&\sup_{t \in [0,\delta]} \nrm{\err_{\tb}}_{L^2_{x,y}} \lesssim \ell^{-2}\nrm{g_0}_{L^2_{x,y}},
\end{split}
\end{equation*} with  $\tilde{\bfB} =  f(r)\rd_\theta$. As before, the goal is to estimate the time derivative of $\brk{b, \chi\tb}$ and the only additional term relative to the linear case from \ref{subsubsec:gei-emhd} is simply the contribution from the nonlinear error, which can be bounded by \begin{equation*} 
\begin{split}
& \left|\brk{\err_b, \chi\tb}\right| \lesssim \ell^{\frac{1}{2}} \nrm{\err_b}_{L^2}\nrm{g_0}_{L^2_{x,y}}. 
\end{split}
\end{equation*} We now observe that \begin{equation*} 
\begin{split}
& \nrm{b}_{H^{s_0}} \le \nrm{\bfB}_{H^{s_0}} + \nrm{\bfPi}_{H^{s_0}} \lesssim_{s_0} M + \ell^{\frac{1}{2}-s_0}
\end{split}
\end{equation*} where we have used the assumption $M\sup_{t \in [0,\delta]} \nrm{\bfB}_{H^{s_0}} \le M$ and the a priori estimates from \ref{subsec:apriori}, by shrinking $\dlt >0$ if necessary (in a way depending only on $\ell$). We then estimate \begin{equation*} 
\begin{split}
& \nrm{(b\cdot\nabla)\nabla\times b}_{L^2} \lesssim \nrm{\nabla^2 b}_{L^\infty}\nrm{b}_{L^2} \lesssim_{s_0} (M + \ell^{\frac{1}{2}-s_0})\nrm{b}_{L^2} 
\end{split}
\end{equation*} and \begin{equation*} 
\begin{split}
& \nrm{(\nabla\times b)\cdot\nabla b }_{L^2} \lesssim \nrm{\nabla b}_{L^4}^2  \lesssim_{s_0} (M + \ell^{\frac{1}{2}-s_0})\nrm{b}_{L^2}  ,
\end{split}
\end{equation*} where we have used simple estimates \begin{equation*} 
\begin{split}
&\nrm{\nabla^2 b}_{L^\infty} \lesssim  \nrm{b}_{H^{s_0}}, \quad \nrm{\nabla b}_{L^4} \lesssim \nrm{b}_{L^2}^{\frac{1}{2}} \nrm{b}_{H^{s_0}}^{\frac{1}{2}} 
\end{split}
\end{equation*} which hold for $s_0 > 7/2$. This shows that \begin{equation*} 
\begin{split}
& \nrm{\err_b}_{L^2} \lesssim (M + \ell^{\frac{1}{2}-s_0})\nrm{b}_{L^2}  .
\end{split}
\end{equation*}Now recalling the estimate \eqref{eq:gei-e-mhd}, we have \begin{equation*} 
\begin{split}
& \left| \frac{\ud}{\ud t} \brk{b, \chi\tb} \right| \lesssim \ell^{\frac{1}{2}}\left( M\ell^{-\frac{1}{2}} + \ell^{-s_0} + \ell^{-2}  \right)\nrm{b}_{L^2} \nrm{g_0}_{L^2_{x,y}}. 
\end{split}
\end{equation*} Since $\brk{\err_b,b} = 0$, we may shrink $\dlt > 0$ if necessary to guarantee that $\nrm{b}_{L^2} \lesssim 2\nrm{b_0}_{L^2}$ on the time interval $[0,\dlt]$. Then, proceeding in the same way as in \ref{subsubsec:growth-Sobolev}, we conclude that \begin{equation*} 
\begin{split}
& M + \ell^{\frac{1}{2}-s_0} \gtrsim_{s_0} \nrm{b(t)}_{H^{s_0}} \gtrsim_{s_0} e^{c_0  s_0\lmb t} \nrm{b_0}_{ H^{s_{0}} } \gtrsim_{s_0,s} \eps e^{c_0  s_0\lmb t} \lmb^{ s_{0} - s } 
\end{split}
\end{equation*} by further shrinking $\dlt > 0$ if necessary, but in a way depending only on $f$ and $M$. We obtain a contradiction since we may fix some $0 < t^* < \dlt $ and take $\lmb \rightarrow +\infty$. The proof is complete. \hfill \qedsymbol

\subsection{Case of \eqref{eq:hall-mhd}}\label{subsec:nonlinear-hallmhd}

We take initial data $(\bfu_0,\bfB_0) = (0, \bfPi_0 + b_0)$ where $\bfPi_0, b_0$ are the same as in the case of \eqref{eq:e-mhd}. In particular, $b_{0}$ is normalized in a way that $\nrm{b_{0}}_{H^{s}}<\eps$. 

Assume that there exist some $\dlt > 0, s_0 > 7/2$, and a solution $(\bfu,\bfB)$ belonging to $L^\infty_t([0,\dlt]; H^{s_0-1} \times H^{s_0})$ such that $\bfu(t = 0) = \bfu_0$ and $\bfB(t = 0) = \bfPi_0 + b_0$. We set $u(t) = \bfu(t) - \bfV(t)$ and $b(t) = \bfB(t) - \bfPi(t)$. Due to the a priori estimates in Proposition \ref{prop:hall-mhd-fradiss-axi-lwp}, we have $\nrm{\nabla\bfV}_{H^{s_0}} + \nrm{\bfPi}_{H^{s_0}} \le 2 \nrm{\bfPi_0}_{H^{s_0}}$ on $[0,\dlt]$ by taking smaller $\dlt = \dlt(\ell)>0$ if necessary; in particular, the pair $(u,b)$ is well-defined on $[0,\dlt]$ and belongs to $L^\infty_t([0,\dlt];H^{s_0-1} \times H^{s_0})$. As before, setting $M:= \sup_{t \in [0,\delta]} \left( \nrm{\bfB}_{H^{s_0}}  + \nrm{\bfu}_{H^{s_0-1}} \right)$ gives that \begin{equation*}
\begin{split}
\nrm{u}_{H^{s_0-1}} + \nrm{b}_{H^{s_0}} \lesssim_{s_0} M + \ell^{\frac{1}{2}-s_0}.
\end{split}
\end{equation*} We then have that $(u,b)$ satisfies 
\begin{equation}\label{eq:axisym-hall-pert-nonlin}
\left\{
\begin{aligned}
& \rd_t u + \bfV \cdot \nabla u + \nabla \bfp - \nu\lap u - (\nabla\times b)\times \bfPi = -u\cdot\nabla\bfV + (\nabla\times\bfPi)\times b  + \errh_u \\
& \rd_t b + \bfPi \cdot\nabla (\nabla\times b) - (\nabla\times b) \cdot \nabla \bfPi - (\nabla \times \bfPi) \cdot \nabla b + \bfV\cdot \nabla b - \bfPi\cdot\nabla u \\
& \quad  = -b\cdot\nabla(\nabla\times\bfPi) + b\cdot\nabla\bfV - u\cdot\nabla\bfPi +\errh_b , \\
\end{aligned}
\right.
\end{equation} where \begin{equation} \label{eq:p-si2}
\begin{split}
\bfp := \sum_{i,j} R_{i} R_{j} ( \bfu^i \bfu^j) - \sum_{i,j} R_{i} R_{j} (\bfB^i\bfB^j) - \frac{|\bfB|^2}{2}
\end{split}
\end{equation} and  \begin{equation}\label{eq:hall-mhd-nonlin-errors}
\left\{
\begin{aligned}
	\errh_u &:= -u\cdot\nabla u + b\cdot\nabla b ,\\
	\errh_b &:= -(b\cdot\nabla)\nabla\times b + (\nabla\times b)\cdot\nabla b + b\cdot\nabla u - u\cdot\nabla b.
\end{aligned}
\right.
\end{equation} Now with \begin{equation*} 
\begin{split}
& \left| \brk{\errh_b, \chi\tb} \right|+\left| \brk{\errh_u, \chi\tu} \right| \lesssim  \ell^{\frac{1}{2}}(M\ell^{-\frac{1}{2}} + \ell^{-s_0}) (\nrm{u}_{L^2}+\nrm{b}_{L^2})\nrm{g_0}_{L^2_{x,y}},
\end{split}
\end{equation*} and recalling the estimates from Proposition \ref{prop:hall-mhd-fradiss-axi-lwp} and Lemma \ref{lem:hall-mhd-fradiss-axi-small-t}, we obtain the bound \begin{equation*} 
\begin{split}
& \left| \frac{\ud}{\ud t} (\brk{u,\chi\tu -  \mathrm{div}^{-1}(\chi'\tu^z)} + \brk{b,\chi\tb}) \right| \lesssim \ell^{\frac{1}{2}}(M\ell^{-\frac{1}{2}} + \ell^{-s_0} + \ell^{-m_0}(1 + \lmb t)^2 e^{c_{1} \lmb t}  ) (\nrm{u}_{L^2}+\nrm{b}_{L^2})\nrm{g_0}_{L^2_{x,y}},
\end{split}
\end{equation*} for some large $m_0$, which can be taken to be $10$. Then, proceeding as in Section \ref{subsec:hallmhd-axisym}, on the time interval $[0,T]$ with $T = c  \lmb^{-1}\ln\lmb $, \begin{equation*} 
\begin{split}
& \brk{b,\chi\tb} \ge (1-c)\ell^{\frac{1}{2}} \nrm{b_0}_{L^2} \nrm{g_0}_{L^2_{x,y}} 
\end{split}
\end{equation*} by taking $\lmb \gg 1$ depending on $\ell, M, s_0, m_0,$ and also on $\nu$ when $\nu > 0$. Then, arguing as in \ref{subsubsec:growth-Sobolev}, we have that $\nrm{b(t)}_{H^{s_0}} \gtrsim_{s_{0}} \lmb^{s_0}  e^{c_0 s_{0} \lmb t} \nrm{b_0}_{L^2} \gtrsim_{s_{0}, s} \eps\lmb^{s_0-s}  e^{c_0 s_{0} \lmb t}$ so that at $t = T$, \begin{equation*} 
\begin{split}
&M+\ell^{\frac12-s_{0}} \nrm{b(T)}_{H^{s_0}} \gtrsim_{s_{0},s} \eps \lmb^{(1+\alp_{0})s_0 - s} 
\end{split}
\end{equation*} for an absolute constant $\alp_{0}>0$. From the assumption $(1+\alp_{0})s_0 >s$, we get a contradiction by taking $\lmb \rightarrow +\infty$. \hfill \qedsymbol

\section{Nonexistence}\label{sec:nonlinear-nonexist}

\subsection{Case of \eqref{eq:e-mhd}}

We now prove Theorem \ref{thm:illposed-strong2-prime} in the \eqref{eq:e-mhd} case. We may assume that $(\bbT,\bbR)_z = \bbR_z$ since the case of $\bbT_z$ was covered in \cite{JO1}. Let $s > 7/2, \dlt > 0$ and $\eps > 0$ be given by the statement, and for simplicity fix an axisymmetric background of the form \begin{equation*} 
\begin{split}
& \tilde{\bfPi}_0 = f_0(r)\chi_0(z/10)\rd_\tht, 
\end{split}
\end{equation*} where  $f_0 \in C^\infty_{comp}$ satisfies the assumptions in Section \ref{subsec:wkb} with $\ell = 1/40$ and $r_{0} = 1/20 = 2\ell$; more specifically, we may choose $f_{0}$ to satisfy $f_0(r) = r - 1/20$ for $1/40 \le r \le 3/40$.  Similarly, we take $0 \le \chi_0 \le 1$ to be a smooth bump function supported in $[-1/100,1/100]$ and $\chi_0 = 1$ in $[-1/200,1/200]$. In particular, $\tilde{\bfPi}_0 \in C^\infty_{comp}$. We then take \begin{equation*} 
\begin{split}
& \bfPi_0 = \sum_{k = k_0}^{\infty} \tilde{\bfPi}_0^{(k)} := \sum_{k = k_0}^{\infty} 2^{-sk} \tilde{\bfPi}_0(2^k x, 2^k y, 2^k(z - z_k)), \quad z_k := 1 - 2^{-k}  
\end{split}
\end{equation*} so that \begin{equation*} 
\begin{split}
& \nrm{\tilde{\bfPi}_0^{(k)}}_{H^s} \lesssim 2^{-\frac{k}{2}} 
\end{split}
\end{equation*} and in particular, by taking $k_0 \gg 1$ we may ensure that $\nrm{\bfPi_0}_{H^s} < \frac{\eps}{2}$. Here we remark in advance that $k_0$ will be taken to be sufficiently large depending on given parameters $\eps, \dlt, s$ and also on a large parameter $N = N(s) \gg 1$ to be introduced below. Note that (by redefining $\chi$ if necessary) $\bfPi_0$ is compactly supported and the supports of $\tilde{\bfPi}_0^{(k)}$ are disjoint from each other; we shall also write \begin{equation*} 
\begin{split}
& \tilde{\bfPi}^{(k)}_0 = f_k(r)\tilde{\chi}_k(z) \rd_\tht, 
\end{split}
\end{equation*} where \begin{equation*} 
\begin{split}
& f_k(r) := 2^{-sk}f_0(2^kr), \quad \tilde{\chi}_k(z) := \chi_0(\frac{2^k}{10}(z-z_k)).
\end{split}
\end{equation*} Note that $\ell_{f_{k}}\sim 2^{-k}$.  Similarly, we set \begin{equation*} 
\begin{split}
&\chi_k(z) := \chi_0(2^k(z-z_k))
\end{split}
\end{equation*} to be a cutoff near $z = z_k$. Note that the lifespan of $\tilde{\bfPi}^{(k)}_0$ under \eqref{eq:e-mhd-axisym-background} increases as $k \rightarrow +\infty$ and there is some non-empty interval of time during which the supports of the solutions $\{ \tilde{\bfPi}^{(k)} \}_{k \ge 0}$ do not overlap with each other. This guarantees that the unique local-in-time solution $\bfPi(t)$ of \eqref{eq:e-mhd-axisym-background} with initial data $\bfPi_0$ (with some abuse of notation) is given by the superposition \begin{equation*} 
\begin{split}
& \bfPi(t) = \sum_{k = k_0}^{\infty} \tilde{\bfPi}^{(k)}(t) := \sum_{k = k_0}^{\infty} 2^{-sk} \tilde{\bfPi}(2^{-(s-1)k}t, 2^k x, 2^k y, 2^k(z - z_k))
\end{split}
\end{equation*} where $\tilde{\bfPi}(t)$ is the solution of \eqref{eq:e-mhd-axisym-background} with initial data $\tilde{\bfPi}_0$, which is smooth for an $O(1)$ interval of time. Moreover, by taking $k_0 \gg 1$, it is guaranteed that $\bfPi(t)$ remains smooth in $[0,\dlt]$ for any $\dlt > 0$ and  $\nrm{\tilde{\bfPi}^{(k)}}_{L^\infty([0,\dlt];H^s)} \le 2\nrm{\tilde{\bfPi}^{(k)}_0}_{H^s}$; in particular, $\nrm{\bfPi(t)}_{L^\infty([0,\dlt];H^s)} < \eps$. We can also make sure that  $\chi_k\rd_z \bfPi(t) \equiv 0$ for any $k \ge k_0$ on $[0,\dlt]$. 

We now rescale the $z$-independent degenerating wave packet solutions $\{\tb_{(\lmb)} \}$ associated with $f_0(r)\rd_\tht$  as follows: \begin{equation}\label{eq:dwp-rescaled} 
\begin{split}
& \tb_k(t) := 2^{\frac{k}{2}} \tb_{(2^{-k}\lmb_k)}(2^{-sk}2^{2k}t, 2^kx,2^ky), \quad \lmb_k := 2^{ Nk}
\end{split}
\end{equation} where $N = N(s) \gg 1$ is to be specified later. Consider the initial data \begin{equation}\label{eq:nonexist-emhd-initial}
\begin{split}
\bfB_0 = \bfPi_0 + \sum_{k = k_0}^\infty 2^{-k} (2^{-k}\lmb_k)^{-s} \left( \chi_k \tb_k(0) -  \mathrm{div}_{U_k}^{-1} (\chi_k' \tb^{z}_{k}(0)) \right);
\end{split}
\end{equation} here $U_k$ is a cube of side-length $O(2^{-k})$ containing the support of $\chi_k\tb_k(t)$ for $t \in [0,\dlt]$. For convenience we shall set $b_k =  \chi_k \tb_k  -  \mathrm{div}_{U_k}^{-1} (\chi_k' \tb^{z}_{k})$. It can be arranged that the cubes $U_k$ do not overlap with each other.  Note that we have  \begin{equation*} 
\begin{split}
(2^{-k}\lmb_k)^{-s}\nrm{\chi_k \tb_k(0) -   \mathrm{div}_{U_k}^{-1} (\chi_k' \tb^{z}_{k}(0))}_{H^s} \lesssim & \nrm{ (2^{-k}\lmb_k)^{-s}\chi_k \tb_k(0) }_{H^s} \lesssim 1
\end{split}
\end{equation*} and thanks to the extra factor of $2^{-k}$, $\bfB_0 \in H^s$ with $\nrm{\bfB_0}_{H^s} < \eps $ by taking $k_0 \gg 1$. From the normalizations in \eqref{eq:dwp-rescaled}, we have that \begin{equation}\label{eq:dwp-rescaled-estimates}
\begin{split}
\nrm{\tb_k(t)}_{L^2_{x,y}} + \nrm{\err_{\tb}(t)}_{L^2_{x,y}} \lesssim 2^{C_sk}
\end{split}
\end{equation} on $[0,\dlt]$ with some $C_s > 0$ depending only on $s$. The degeneration estimate takes the form \begin{equation*} 
\begin{split}
& \nrm{\tb_k(t)}_{L^2_{\tht} L^1_{rdr}} \lesssim 2^{C_sk} \exp(-2^{-c_sk}\lmb_k t). 
\end{split}
\end{equation*} 

Now that we have fixed our choice of initial data, towards a contradiction, assume that  there is a solution $\bfB \in L^\infty_t([0,\dlt];H^s)$ to \eqref{eq:e-mhd} with $\bfB(t = 0)  = \bfB_0$ given by \eqref{eq:nonexist-emhd-initial}. Denoting $b(t) := \bfB(t) - \bfPi(t)$, we have that $b(t)$ satisfies \begin{equation}\label{eq:e-mhd-pert} 
\begin{split}
&   \rd_t b + (b\cdot\nabla)(\nabla\times\bfPi) - (\nabla\times\bfPi)\cdot\nabla b + (\bfPi\cdot\nabla)(\nabla\times b) - (\nabla\times b)\cdot\nabla\bfPi   = \nabla \times ((\nabla\times b)\times b). 
\end{split}
\end{equation}

Before we proceed to the proof of (localized) generalized energy identity, we need to localize the energy identity for $b$ itself. Unfortunately, the $L^2$-norm of $b$ localized to the support of $\tilde{\bfPi}^{(k)}$ does not satisfy an energy identity by itself, as there are contributions from neighboring pieces. As in \cite[Section 5.2]{JO1}, we use cutoff functions with fast decaying tails which can accommodate such interactions. To this end, we prepare a $C^\infty$ positive function $\widetilde{\chi} : \mathbb{R} \rightarrow \mathbb{R}_+$ with the following properties: \begin{itemize}
	\item $\widetilde{\chi}(z) =1$ for $z \in [-1/8,1/8]$,
	\item $|\widetilde{\chi}'(z)| \le |\widetilde{\chi}(z)|$ for all $z \in \mathbb{R}$, and
	\item  $\widetilde{\chi} \le 2^{-|z|}$ for $z> 1/2$. 
\end{itemize} Note that instead of requiring $\widetilde{\chi}$ to have compact support, we contend with exponential decay; this is why we can pointwise bound the derivative in terms of the function itself. We then define rescaled cutoffs by \begin{equation*} 
\begin{split}
&\widetilde{\chi}_k(z) = \widetilde{\chi}(2^k(z-z_k)).
\end{split}
\end{equation*} Note that \begin{itemize}
	\item $ \tilde{\bfPi}^{(k)}\widetilde{\chi}_k = \tilde{\bfPi}^{(k)}$ and $\bfPi\widetilde{\chi}_k' = 0$, 
	\item $\widetilde{\chi}_k = 1$ on the support of $\chi_k$; in particular $\widetilde{\chi}_k \chi_k'  = \widetilde{\chi}_k'\chi_k = 0$. 
\end{itemize} To proceed, we multiply both sides of \eqref{eq:e-mhd-pert} by $\widetilde{\chi}_k$ and test against $\widetilde{\chi}_kb$. The first four terms involving $\bfPi$ are  easy to handle: we have 
\begin{equation*} 
\begin{split}
& \left| \brk{ \widetilde{\chi}_k  (b\cdot\nabla)(\nabla\times\bfPi) , \widetilde{\chi}_kb} \right|  \lesssim \nrm{\bfPi}_{W^{2,\infty}}\nrm{\widetilde{\chi}_kb}_{L^2}^2 ,
\end{split}
\end{equation*} \begin{equation*} 
\begin{split}
& \left| \brk{ \widetilde{\chi}_k (\nabla\times\bfPi)\cdot\nabla b , \widetilde{\chi}_kb} \right| \lesssim \nrm{\bfPi}_{W^{1,\infty}}\nrm{\widetilde{\chi}_k'b}_{L^2}\nrm{\widetilde{\chi}_kb}_{L^2}\lesssim 2^k \nrm{\bfPi}_{W^{1,\infty}}\nrm{\widetilde{\chi}_kb}_{L^2}^2 
\end{split}
\end{equation*} (using the pointwise bound $|\widetilde{\chi}_k'(z)| \le 2^k|\widetilde{\chi}_k(z)|$), and similarly \begin{equation*} 
\begin{split}
&  \left| \brk{ \widetilde{\chi}_k\left[   (\bfPi\cdot\nabla)(\nabla\times b) - (\nabla\times b)\cdot\nabla\bfPi \right] , \widetilde{\chi}_kb} \right| \\
&\qquad  \lesssim (2^k \nrm{\bfPi}_{W^{1,\infty}} + 2^{2k}\nrm{\bfPi}_{L^\infty})\nrm{\widetilde{\chi}_kb}_{L^2}^2 + 2^k\nrm{\bfPi}_{L^\infty} \nrm{\widetilde{\chi}_kb}_{L^2}\nrm{\nabla(\widetilde{\chi}_kb)}_{L^2} . 
\end{split}
\end{equation*} We further bound the last term as follows: \begin{equation*} 
\begin{split}
&2^k\nrm{\bfPi}_{L^\infty}  \nrm{\widetilde{\chi}_kb}_{L^2}\nrm{\nabla(\widetilde{\chi}_kb)}_{L^2}  \lesssim 2^k\nrm{\bfPi}_{L^\infty} \nrm{\widetilde{\chi}_kb}_{L^2}^{2-\frac{1}{s}} \nrm{\widetilde{\chi}_kb}_{H^s}^{\frac{1}{s}} \lesssim 2^{2k}\nrm{\bfPi}_{L^\infty} \nrm{b}_{H^s}^{\frac{1}{s}} \nrm{\widetilde{\chi}_kb}_{L^2}^{2-\frac{1}{s}}. 
\end{split}
\end{equation*} Next, similar estimates hold for the nonlinear term: we have \begin{equation*} 
\begin{split}
& \left|\brk{\widetilde{\chi}_k \nabla\times((\nabla\times b)\times b), \widetilde{\chi}_k b}\right| \lesssim \left|\brk{\widetilde{\chi}_k' b\nabla b, \widetilde{\chi}_k b}\right| + \nrm{\nabla^2b}_{L^\infty} \nrm{\widetilde{\chi}_kb}_{L^2}^2 
\end{split}
\end{equation*} (since there is no $\rd_{zz}$ derivative) and therefore \begin{equation*} 
\begin{split}
& \left|\brk{\widetilde{\chi}_k \nabla\times((\nabla\times b)\times b), \widetilde{\chi}_k b}\right| \lesssim \nrm{\nabla^2b}_{L^\infty} \nrm{\widetilde{\chi}_kb}_{L^2}^2 + 2^k\nrm{\nabla b}_{L^\infty} \nrm{\widetilde{\chi}_kb}_{L^2}^2 .
\end{split}
\end{equation*} Collecting the estimates, we have  \begin{equation}\label{eq:local-energy-estimate}
\begin{split}
\frac{\ud}{\ud t}\nrm{\widetilde{\chi}_k b}_{L^2} \lesssim 2^{2k}\left( \nrm{\widetilde{\chi}_k b}_{L^2}^{1 - \frac{1}{s}} + \nrm{\widetilde{\chi}_k b}_{L^2} \right) 
\end{split}
\end{equation} where the  {implicit} constant  depends on $\nrm{b}_{L^\infty_t H^s}$ and $\nrm{\bfPi}_{W^{2,\infty}}$. At the initial time, we estimate the local energy: \begin{equation*}
\begin{split}
\nrm{\widetilde{\chi}_k b(t = 0)}_{L^2}^2 =  c{2^{-4k}} \lmb_k^{-2s} + \sum_{k_0 \le k' , k' \ne k }  {2^{-2k'}} \lmb_{k'}^{-2s}\nrm{\widetilde{\chi}_k b_{k'}(t = 0)}^2_{L^2} 
\end{split}
\end{equation*} where $c \sim 1$ is from the choice of $\widetilde{\chi}$.  Note that the contribution  for $ k' > k$ is negligible relative to $ {2^{-4k}} \lmb_k^{-2s}$ from the decay of $ {2^{-2k'}} \lmb_{k'}^{-2s}$ in $k'$. On the other hand, for $k' < k$, we use the decay of $\chi_k$: for any $k' < k$, the support of $b_{k'}$ is separated from $z_k$ by at least $c2^{-\frac{k}{2}}$ with $c > 0$ independent of $k$. Hence, \begin{equation*}
\begin{split}
|\widetilde{\chi}_k| \lesssim 2^{-c2^{\frac{k}{2}}} \lesssim_{N,s} 2^{-4Nsk} \ll   {2^{- 4k}} \lmb_k^{-2s} 
\end{split}
\end{equation*} on the support of $b_{k'}$ with any $k' > k  $, and we obtain that $\nrm{\chi_kb(t = 0)}_{L^2}   {\aleq 2^{-2k}} \lmb_k^{-s}$ by choosing $k_0 $ sufficiently large with respect to $N, s$. Using this together with \eqref{eq:local-energy-estimate} yields that \begin{equation}\label{eq:localenergy-time}
\begin{split}\nrm{\widetilde{\chi}_kb {(t)}}_{L^2} \lesssim_s \left(  {2^{-\frac{k}{s}}} \lmb_k^{-1} + 2^{2k}t \right)^s \lesssim_s  {2^{-2k}} \lmb_k^{-s} + 2^{2ks}t^s. 
\end{split}
\end{equation} 

We are in a position to prove a localized version of the generalized energy estimate: recall that \begin{equation*} 
\begin{split}
&\rd_t \tb + (\tilde{\bfPi}\cdot\nabla)(\nabla\times \tb) - ((\nabla\times\tb)\cdot\nabla)\tilde{\bfPi} = \err_{\tb}. 
\end{split}
\end{equation*} where \begin{equation*} 
\begin{split}
& \tilde{\bfPi} = \sum_{k = k_0}^\infty \tilde{\bfPi}_k =  \sum_{k = k_0}^\infty f_k(r)\rd_\tht,\quad \err_{\tb} = \sum_{k = k_0}^\infty \err_{\tb,k}
\end{split}
\end{equation*} with $\err_{\tb,k}$ being the error associated with $\tb_k$. We use the above and \eqref{eq:e-mhd-pert} to compute \begin{equation*} 
\begin{split}
&\frac{\ud}{\ud t} \brk{\widetilde{\chi}_kb, \chi_k \tb_k}  = -\brk{ \widetilde{\chi}_k\left[  (b\cdot\nabla)(\nabla\times\bfPi) - (\nabla\times\bfPi)\cdot\nabla b + (\bfPi\cdot\nabla)(\nabla\times b) - (\nabla\times b)\cdot\nabla\bfPi \right]  , \chi_k \tb_k} \\
& \qquad + \brk{\widetilde{\chi}_k \nabla \times ((\nabla\times b)\times b) , \chi_k \tb_k } + \brk{\widetilde{\chi}_kb, -(\tilde{\bfPi}_k\cdot\nabla)(\nabla\times \tb_k) + ((\nabla\times\tb_k)\cdot\nabla)\tilde{\bfPi}_k} + \brk{\widetilde{\chi}_kb, \err_{\tb,k}} . 
\end{split}
\end{equation*} We estimate terms similarly as in the proof of the localized energy estimate; using that $\tb_k$ vanishes on the support of $\widetilde{\chi}_k'$ and that $\bfPi = \tilde{\bfPi}_k$ on the support of $\chi_k$ cancels several terms. We end up (see \cite[Section 5.2]{JO1} for details) with \begin{equation*} 
\begin{split}
& \left| \frac{\ud}{\ud t} \brk{\widetilde{\chi}_kb, \chi_k \tb_k}  \right| \lesssim 2^{C_sk} \nrm{\widetilde{\chi}_kb}_{L^2} \nrm{\chi_k\tb_k}_{L^2} 
\end{split}
\end{equation*} where the implicit constant depends on $\nrm{b}_{L^\infty_tH^s}$ and $\nrm{\bfPi}_{W^{2,\infty}} + \nrm{\tilde{\bfPi}}_{W^{2,\infty}}$. Combining \eqref{eq:localenergy-time} with the above gives that \begin{equation*} 
\begin{split}
\left|\brk{\widetilde{\chi}_k b(t),\chi_k \tb_{k}(t)} -2^{-ck} \lmb_k^{-s}\right| \lesssim_s   {2^{C_{s} k}} \int_0^t \left(  {2^{-k}} \lmb_k^{-s} + 2^{2ks}t'^s \right) dt' \lesssim_s  {2^{C_{s} k}}( {2^{-k}} \lmb_k^{-s}t + 2^{2ks}t^{s+1}), 
\end{split}
\end{equation*} so that  $\brk{\widetilde{\chi}_k b, \chi_k \tb_{k} }  \ge {2^{-(ck+1)}} \lmb_k^{-s}$ on  $[0,t^*_k]$ with 
\begin{equation} \label{eq:t-ast-k}
{t^{*}_{k} = 2^{-c_{s} k} \lmb_{k}^{-\frac{s}{s+1}}}.
\end{equation}
Taking $k_0 $ larger if necessary, we have $t^*_k \le \delta$ for all $k \ge k_0$. Then 
\begin{equation*} 
\begin{split}
2^{-(ck+1)} \lmb_{k}^{-s} &\aleq 
\nrm{\chi_k \tb_{k}(t^{*}_{k})}_{L^{2}_{\tht, z} L^{1}_{rdr}} 
\nrm{\widetilde{\chi}_k b(t^{*}_{k})}_{L^{2}_{\tht,z} L^{\infty}_{r}} \aleq \nrm{\chi_k \tb_{k}(t^{*}_{k})}_{L^{2}_{\tht,z} L^{1}_{rdr}} 
\nrm{\widetilde{\chi}_k b(t^{*}_{k}) }_{L^{2}}^{1-\frac{1}{2s}} 
\nrm{\widetilde{\chi}_k b(t^{*}_{k}) }_{H^{s}}^{\frac{1}{2 s}}.
\end{split}
\end{equation*}
By the degeneration property  \eqref{eq:localenergy-time}, it follows that
\begin{equation*}
\nrm{\widetilde{\chi}_{k} b(t^{*}_{k})}_{H^{s}} 
\ageq_{s} 2^{-C_{s} k} \lmb_{k}^{-  \frac{3s^{2}}{s+1}} \exp\left( 2^{-c_{s} k} \lmb_{k}^{\frac{1}{s+1}}\right).
\end{equation*} 
By the algebra property of $H^{s}$, we may replace the LHS by $\nrm{b(t^{*}_{k})}_{H^{s}}$ (with a different $C_{s}$). Taking $N = N(s) \gg 1$ (recall that $\lmb_k = 2^{Nk}$), we may ensure that
\begin{equation*}
\nrm{b(t^{*}_{k})}_{H^{s}} \ageq_{s} 2^{c_{s} k}
\end{equation*}
for some $c_{s} > 0$ independent of $k \geq k_{0}$, which contradicts boundedness of $\nrm{b}_{L^\infty_t H^s}$.  \hfill \qedsymbol

\subsection{Case of \eqref{eq:hall-mhd}}

We now give the proof of nonexistence in the Hall-MHD case. The general idea is the same as in the E-MHD case; fix some background magnetic field with countably many degenerate points and place a degenerating wave packet solution near each degenerate point to conclude unbounded growth in any short amount of time. Due to the strong non-locality in the velocity interactions (caused by the pressure term), we are only able to handle data with non-compact support in the Hall-MHD case. But as we shall see below, this will significantly simplify the analysis as we can take the points of degeneracy very far from each other. 

To begin, let $s > 7/2$, $\dlt > 0$, and $\eps > 0$ be given and take $\tilde{\bfPi}_0$ exactly as in the above case of \eqref{eq:e-mhd}. We then take \begin{equation}\label{eq:hallmhd-initial-mag} 
\begin{split}
& {\bfPi}_0 = \sum_{k = k_0}^{\infty} \tilde{\bfPi}_0^{(k)} = \sum_{k = k_0}^{\infty}  2^{-k} \tilde{\bfPi}_0(x,y,z-z_k),\quad z_k = z_{k-1} + \mu^k 
\end{split}
\end{equation} where $\mu = \mu(N,s) \gg 1$ and $k_0 = k_0(N,s,\eps,\dlt) \gg 1 $ are  to be determined. Note that there are no spatial rescalings. Next, we take the $z$-independent degenerating wave packet solutions $\{ \tb_{(\lmb)} , \tu_{(\lmb)} \}$ associated with $f_0(r)\rd_\tht$ and set \begin{equation*} 
\begin{split}
& \tb_k(t) := \tb_{(\lmb_k)}(2^{-k}t, x, y, z- z_k), \quad \tu_k(t) := \tu_{(\lmb_k)}(2^{-k}t, x, y, z-z_k)
\end{split}
\end{equation*} where $\lmb_k = 2^{Nk}$ with $N = N(s) \gg 1$ to be determined. This time, we simply translate $\chi_0$ (which is the same as the previous section) to define $\chi_k(z) := \chi_0( z- z_k).$ Then, $U_k$ is defined to be a cube of side-length $O(1)$ containing the support of $\chi_k\tb_k(t)$ for $t \in [0,\dlt]$. It is easy to guarantee that the cubes $U_k$ do not overlap with each other, by taking $\mu >1$ in \eqref{eq:hallmhd-initial-mag} larger if necessary. We then take the initial data $(\mathbf{0}, \bfB_0)$ to \eqref{eq:hall-mhd} where \begin{equation}\label{eq:nonexist-hallmhd-initial}
\begin{split}
\bfB_0 := \bfPi_0 + \sum_{k = k_0}^\infty 2^{-k} \lmb_k^{-s} \left( \chi_k\tb_k(0) -   \mathrm{div}_{U_k}^{-1}  (\chi_k'\tb^{z}_{k}(0)) \right).
\end{split}
\end{equation} Clearly, we have $\nrm{\bfB_0}_{H^s} \le\eps $ by taking $k_0 = k_0(\eps) \gg 1$. 


We use the letter $\bfV$ to denote the corresponding velocity field, which is smooth on the time interval $[0,\dlt]$. Next, applying Lemma \ref{lem:hall-mhd-fradiss-axi-small-t} with a choice of $m = 10$ gives smallness in time estimate \eqref{eq:small-t-V}, \eqref{eq:small-t-Pi}. 

Now we assume that there exists a solution $(\bfu,\bfB) \in L^\infty_t([0,\dlt];H^{s-1} \times H^s)$ with initial data $(0,\bfB_0)$. Setting $u = \bfu - \bfV$ and $b = \bfB - \bfPi$, we have that $(u,b)$ belongs to $L^\infty([0,\dlt];H^{s-1} \times H^s)$ and satisfies \eqref{eq:axisym-hall-pert-nonlin}--\eqref{eq:p-si2}. We note the following simple estimate for the pressure: \begin{equation*}
\begin{split}
\nrm{\bfp}_{L^2} \lesssim \nrm{|\bfu|^2}_{L^2} + \nrm{|\bfB|^2}_{L^2} \lesssim \nrm{\bfu}_{H^{s-1}}^2 + \nrm{\bfB}_{H^s}^2 
\end{split}
\end{equation*} where we have used the Sobolev embedding $H^{s-1} \subset L^\infty$. Next, we take cutoff functions \begin{equation*}
\begin{split}
\widetilde{\chi}_k(z) = \chi_0(2 \mu^{-k}(z-z_k)). 
\end{split}
\end{equation*} As we shall see below, the proof of a local version of the energy identity for $(u,b)$ is straightforward since we may take $\mu \gg 1$: simply note that $\nrm{\widetilde{\chi}_k'}_{L^\infty} \lesssim \mu^{-k}$, and we now compute \begin{equation*}
\begin{split}
\frac{1}{2}\frac{\ud}{\ud t} \left( \nrm{\widetilde{\chi}_k u}_{L^2}^2 +\nrm{\widetilde{\chi}_k b}_{L^2}^2  \right)
\end{split}
\end{equation*} using \eqref{eq:axisym-hall-pert-nonlin}; after straightforward estimates using integration by parts, we first obtain \begin{equation*}
\begin{split}
&\left| \frac{1}{2}\frac{\ud}{\ud t} \nrm{\widetilde{\chi}_k u}_{L^2}^2 - \brk{\widetilde{\chi}_k(\bfPi\cdot\nabla b) ,\widetilde{\chi}_ku}  - \brk{\widetilde{\chi}_k\errh_{u},\widetilde{\chi}_k u } \right| \\
&\qquad \lesssim \nrm{\widetilde{\chi}_k'}_{L^\infty} \nrm{\widetilde{\chi}_k u}_{L^2} \left( \nrm{\bfp}_{L^2} + \nu \nrm{\nabla u}_{L^2} +1 \right) + \nrm{\widetilde{\chi}_k u}_{L^2}(\nrm{\widetilde{\chi}_k u}_{L^2}+\nrm{\widetilde{\chi}_k b}_{L^2}) \\
&\qquad \lesssim (\mu^{-k} + \nrm{\widetilde{\chi}_k u}_{L^2}+\nrm{\widetilde{\chi}_k b}_{L^2})\nrm{\widetilde{\chi}_k u}_{L^2} ,
\end{split}
\end{equation*} where the implicit constant depends on $M + \nrm{\bfV}_{W^{1,\infty}} + \nrm{\bfPi}_{W^{2,\infty}}$. Similarly, \begin{equation*}
\begin{split}
&\left| \frac{1}{2}\frac{\ud}{\ud t} \nrm{\widetilde{\chi}_k b}_{L^2}^2 - \brk{\widetilde{\chi}_k(\bfPi\cdot\nabla u) ,\widetilde{\chi}_kb}  - \brk{\widetilde{\chi}_k\errh_{b},\widetilde{\chi}_k b } \right| \lesssim  (\mu^{-k} + \nrm{\widetilde{\chi}_k u}_{L^2}+\nrm{\widetilde{\chi}_k b}_{L^2})\nrm{\widetilde{\chi}_k b}_{L^2} 
\end{split}
\end{equation*}where again the implicit constant depends on $M + \nrm{\bfV}_{W^{1,\infty}} + \nrm{\bfPi}_{W^{2,\infty}}$. Then note that \begin{equation*}
\begin{split}
\left| \brk{\widetilde{\chi}_k(\bfPi\cdot\nabla u) ,\widetilde{\chi}_kb}  + \brk{\widetilde{\chi}_k(\bfPi\cdot\nabla b) ,\widetilde{\chi}_ku} \right| \lesssim \nrm{\widetilde{\chi}_k'}_{L^\infty} \nrm{|\bfPi|b}_{L^2} \nrm{\widetilde{\chi}_k u}_{L^2} \lesssim \mu^{-k}\nrm{\widetilde{\chi}_k u}_{L^2} 
\end{split}
\end{equation*} after integrating by parts. Moreover, recall the cancellation property $\brk{\errh_u,u} + \brk{\errh_b,b} = 0$: from this it is straightforward to see that \begin{equation*}
\begin{split}
\left| \brk{\widetilde{\chi}_k\errh_{b},\widetilde{\chi}_k b }  + \brk{\widetilde{\chi}_k\errh_{u},\widetilde{\chi}_k u }  \right| \lesssim \nrm{\widetilde{\chi}_k'}_{L^\infty}\left( \nrm{\widetilde{\chi}_k u}_{L^2} + \nrm{\widetilde{\chi}_k b}_{L^2} \right)
\end{split}
\end{equation*} where the implicit constant depends on $M$. Collecting the estimates, we conclude that \begin{equation*}
\begin{split}
\left|\frac{1}{2}\frac{\ud}{\ud t} \left( \nrm{\widetilde{\chi}_k u}_{L^2}^2 +\nrm{\widetilde{\chi}_k b}_{L^2}^2  \right)\right|\lesssim (\mu^{-k} + \nrm{\widetilde{\chi}_k u}_{L^2} +\nrm{\widetilde{\chi}_k b}_{L^2})(\nrm{\widetilde{\chi}_k u}_{L^2} +\nrm{\widetilde{\chi}_k b}_{L^2}). 
\end{split}
\end{equation*} Since at the initial time, we have for some absolute constant $c >0$, \begin{equation*}
\begin{split}
(\nrm{\widetilde{\chi}_k u}_{L^2} +\nrm{\widetilde{\chi}_k b}_{L^2})(t = 0) = \nrm{\widetilde{\chi}_k b}_{L^2}(t = 0) \simeq  2^{-ck-Nsk}, 
\end{split}
\end{equation*} we may take $\mu \gg 1$ in a way depending only on $N$ and $s$ that $\mu^{-k} \lesssim \nrm{\widetilde{\chi}_k b}_{L^2}(t = 0) $. Then from Gronwall's inequality, we obtain \begin{equation*}
\begin{split}
\nrm{\widetilde{\chi}_k u}_{L^2} +\nrm{\widetilde{\chi}_k b}_{L^2} \lesssim \nrm{\widetilde{\chi}_k b(t = 0)}_{L^2},
\end{split}
\end{equation*} uniformly on $[0,\dlt]$ with an implicit constant independent of $k$. 


Lastly it only remains to localize a version of the generalized energy estimate; we claim that \begin{equation*}
\begin{split}
\left| \frac{\ud}{\ud t} \left( \brk{\widetilde{\chi}_k u, \chi_k \tu_k - \mathrm{div}_{U_k}^{-1}(\chi_k'\tu_k^z) } +\brk{\widetilde{\chi}_k b, \chi_k \tb_k } \right) \right| \lesssim 2^{C_sk} \nrm{\widetilde{\chi}_k b(t = 0)}_{L^2}\nrm{\chi_k\tb_k(t=0)}_{L^2}. 
\end{split}
\end{equation*} The proof is completely parallel to Section \ref{subsec:hallmhd-axisym}; the background solution $(\bfV, \bfPi)$ satisfies Sobolev and smallness estimates uniformly, both in $t \in [0,\dlt]$ and $k$. Moreover, the additional terms appearing in the above (in comparison with Section \ref{subsec:hallmhd-axisym}) only come from the additional localization $\widetilde{\chi}_k$ and hence such contributions can be bounded in terms of $\nrm{\widetilde{\chi}_k'}_{L^\infty}$ which is at the same order with $\nrm{\widetilde{\chi}_k b(t = 0)}_{L^2}$. 

Given the above, we have at once that \begin{equation*} 
\begin{split}
& \brk{\widetilde{\chi}_k b, \chi_k \tb_k } \ge (1 - c) \nrm{\widetilde{\chi}_k b(t = 0)}_{L^2}\nrm{\chi_k\tb_k(t=0)}_{L^2}
\end{split}
\end{equation*} on the time interval $[0,T_k]$ where $T_k = c(2^{-k}\lmb_k)^{-1} \ln(2^{-k}\lmb_k)$ with some absolute constant $c>0$. Certainly $T_k$ is decreasing in $k$ and $T_k \le \dlt $ for $k$ sufficiently large. Recall that the degeneration estimate is given by \begin{equation*} 
\begin{split}
& \nrm{\tpsi_k(t)}_{L^2_{x,y}} \lesssim \lmb_k^{-1} e^{-c_02^{-k}\lmb_k t} \nrm{\tb_k(t=0)}_{L^2_{x,y}}. 
\end{split}
\end{equation*} We then argue along the lines of Sections \ref{subsec:hallmhd-axisym} and \ref{subsec:nonlinear-hallmhd}, resulting in \begin{equation*} 
\begin{split}
& \nrm{\widetilde{\chi}_k b(T_k)}_{H^{s}} \gtrsim \eps \lmb_k^c \nrm{\widetilde{\chi}_kb_0}_{H^s} \gtrsim \eps \lmb_k^{\frac{c}{2}}
\end{split}
\end{equation*} where $c>0$ is an absolute constant. Finally, observing that \begin{equation*} 
\begin{split}
& \nrm{\widetilde{\chi}_k b(T_k)}_{H^{s}} \le \nrm{\widetilde{\chi}_k}_{H^s} \nrm{b(T_k)}_{H^s} \lesssim (1+\mu^{-1}) \nrm{b(T_k)}_{H^s}
\end{split}
\end{equation*} and taking $k \rightarrow +\infty $ shows that $\liminf_{t \rightarrow 0^+} \nrm{b(t)}_{H^s} = +\infty ,$ which is a contradiction. \hfill \qedsymbol

\bibliographystyle{amsplain}
\bibliography{hallmhd}

\end{document}